\documentclass[reqno, 11pt]{amsart}

\usepackage{srcltx}
\usepackage[mathscr]{eucal}
\usepackage{amscd}
\usepackage{color}
\usepackage{amsfonts}
\usepackage{amsmath, amsthm, amssymb}
\usepackage{latexsym}
\usepackage{enumerate}
\usepackage{graphics}
\usepackage{epsfig}
\usepackage{xcolor}
\usepackage{graphicx}
\usepackage{epstopdf}
 \newtheorem{thm}{Theorem}[section]
 
 \newtheorem{lem}[thm]{Lemma}
 \newtheorem{prop}[thm]{Proposition}
 \newtheorem{defn}[thm]{Definition}
 \newtheorem{rem}[thm]{Remark}
 
 \numberwithin{equation}{section}

\def\R{{\Bbb R}}

\def\la{{\lambda}}
\def\pl{{\partial}}
\def\R{{\mathbb R}}
\def\N{{\Bbb N}}

\def\ep{{\epsilon}}

\def\no{{\nonumber}}

\def\Om{{\Omega}}
\def\bge{\begin{eqnarray}}
\def\bgee{\begin{eqnarray*}}
\def\ege{\end{eqnarray}}
\def\egee{\end{eqnarray*}}
\newcommand{\dd}{{\mathrm{d}}}
\newcommand{\ee}{{\mathrm{e}}}

\newcommand{\dif}[2]{\frac{\dd #1}{\dd #2}}

\newcommand{\I}[1]{{\, \dd #1}}

\newcommand{\Ir}{{\, \dd r}}

\newcommand{\Ix}{{\, \dd x}}

\def\be{\begin{equation}}
\def\ee{\end{equation}}
\def\bse{\begin{subequations}}
\def\ese{\end{subequations}}
\def\bge{\begin{eqnarray}}
\def\bgee{\begin{eqnarray*}}
\def\ege{\end{eqnarray}}
\def\egee{\end{eqnarray*}}

\usepackage{lineno} %[pagewise]
\date{\today}
\begin{document}
\title[A nonlocal problem with Robin boundary conditions]
{A study of a nonlocal problem with Robin boundary conditions arising from MEMS technology}

\author{Ourania Drosinou}
\address{Department of Mathematics, University of Aegean,
Gr-83200 Karlovassi, Samos, Greece
}
\email{rdrosinou@aegean.gr}

\author{Nikos I. Kavallaris}
\address{Department of Mathematical and Physical Sciences, University of Chester, Thornton Science Park,
Pool Lane, Ince, Chester  CH2 4NU, UK}
\email{n.kavallaris@chester.ac.uk}

\author{Christos V. Nikolopoulos}
\address{Department of Mathematics, University of Aegean,
Gr-83200 Karlovassi, Samos, Greece
}

\email{cnikolo@aegean.gr}



\date{\today}

%\date{}
\pagenumbering{arabic}

%_________________________________________________________________

\begin{abstract}
In the current work we study a nonlocal parabolic problem with Robin boundary conditions. The problem arises from the study of an idealized electrically actuated MEMS  (Micro-Electro-Mechanical System)  device, when the ends of the device  are  attached or  pinned to a cantilever. Initially the steady-state problem is investigated estimates of the pull-in voltage are derived.
In particular, a Poho\v{z}aev's type identity is also  obtained which then facilitates the derivation of an estimate of the pull-in voltage for radially symmetric   $N-$dimensional domains. Next  a detailed study of  the time-dependent problem is delivered  and global-in-time as well as quenching results are obtained  for generic and radially symmetric domains. The current work closes with a numerical investigation of the presented nonlocal model via an adaptive numerical method. Various numerical experiments are presented,  verifying the previously derived  analytical results as well as providing new insights on the qualitative behaviour of the studied nonlocal model.
\end{abstract}

\maketitle
\vspace{0.5in}

 \section{Introduction}
In this work we study the following nonlocal parabolic problem:
\bse\label{o.oneN1}
 \be\label{o.one1N}
u_t =\Delta u+ \frac{\lambda}{\left(1-u\right)^2 \left[1+ \alpha\int_{\Omega} 1/(1-u\,) dx\right]^{2}},\quad
 \quad \mbox{in} \quad Q_T:=\Om \times (0,T),\;T>0,
 \ee
 \be \label{o.one2N}
\frac{\partial u}{ \partial \nu}+ \beta u=0, \quad \mbox{on} \quad \Gamma_T:=\pl\Om \times (0,T),
  \ee
  \be\label{o.one3N}
 u(x,0)=u_0(x), \quad x \in \Omega,
   \ee \ese
where $\lambda>0$, $\alpha>0$, $\beta>0$,  are given positive constants. Especially, $\la$ is proportional to the applied voltage into the system, called pull-in voltage parameter, and it is actually the controlling parameter for the operation of the considered MEMS device. The initial data
$u_0(x)$ is assumed to be a smooth function such that $0<u_0(x)< 1$ for all $x \in \bar{\Omega}$ and $\frac{\partial u_0}{ \partial \nu}+ \beta u_0=0,\;\;\mbox{for}\;\;x \in \ \partial \Omega;$ here  $\nu=\nu(x)$ stands for the unit outward normal vector on the boundary of the $N-$dimensional domain $\Om.$  Notably, from the applications point of view only the cases $N=1,2$ are viable, however from the point of view of mathematical analysis cases $N\geq 3$ are also interesting and so they will be investigated.  Moreover,  here $T$ denotes the maximum existence time of solution $u.$

When $\alpha=0$ problem (\ref{o.oneN1}) reduces to the local parabolic problem
\bse\label{o.onel}
 \be\label{o.onel1}
 u_{t}=\Delta u + \frac{\lambda}{ (1-u)^2}  , \quad \mbox{in} \quad Q_T,
 \ee
 \be \label{o.onel2}
\frac{\partial u}{ \partial \nu}+ \beta u=0, \quad \mbox{on} \quad \Gamma_T,
  \ee
  \be\label{o.onel3}
 u(x,0)=u_0(x), \quad x \in \Omega.
   \ee \ese
   It is worth mentioning  that for Robin type  boundary conditions, as the ones considered above for $\beta>0,$  there is a limited study for the local problem, cf. \cite{Guo91}, while to the best of our knowledge no published works dealing  with the nonlocal problem \eqref{o.oneN1} can be found in the literature.  Our motivation for studying  \eqref{o.oneN1} comes from the fact that it  is actually linked with special applications in MEMS industry, as pointed  below. Furthermore, due the imposed Robin-type boundary conditions extra technical difficulties arise compared to the study of the Dirichlet problem, a fact that is indicated through the manuscript.

 Problem \eqref{o.oneN1}  arises as a mathematical model which describes the operation of some electrostatic actuated micro-electro-mechanical systems (MEMS). Those MEMS  systems are precision devices which combine mechanical processes with electrical circuits. MEMS devices range in size from
millimeters down to microns, and involve precision mechanical
components that can be constructed using semiconductor
manufacturing technologies.

In particular, electrostatic actuation is
a popular application of MEMS.
 Various electrostatic actuated MEMS
have been developed and used in a wide variety of devices applied as sensors
and have fluid-mechanical, optical, radio frequency (RF),
data-storage, and biotechnology applications.
 Examples of microdevices of this kind include microphones,
temperature sensors, RF switches, resonators, accelerometers,
micromirrors, micropumps, microvalves,  etc., see for example
\cite{EGG10,JAP-DHB02,y}.

%____________________________________________________________________
In the sequel a derivation for the nonlocal  model \eqref{o.oneN1}, for the one-dimensional case, is presented  and also
the association of that model with  applications in MEMS industry is explained. The main body of the derivation is standard (see for example \cite{KLNT11, KLNT15, KS18, JAP-DHB02}), however
  in order to justify  the inclusion for the Robin boundary conditions in the model
and for completeness it is presented here as well.
 The modifications  of this  modelling approach are presented  in detail in the next section.

%____________________________________________________
\subsection{Derivation of the model}

 We consider an idealized electrostatiaclly  MEMS device which  consists of an elastic  membrane and a rigid plate placed parallel to each other as it can be seen in Figure \ref{Figmems1}. The membrane has two parallel sides usually attached or  pinned to a cantilever, while the other sides are free.
 Both membrane and plate have width $w$ and length L, and in the undeformed state (for the membrane) the distance between the membrane and the plate is $l.$ We assume here that the gap between the plate and the membrane is small, that is $l\ll L$ and $l\ll w.$  Besides, the area between the elastic mebrane and the rigid plate is occupied by some inviscid material with dielectric constant one, so permittivity is that of free space, $ \epsilon_0$.
 %____________________________________________________________
 \begin{figure}[h!]%\vspace{-1cm}
  \centering
\includegraphics[width=.7\textwidth]{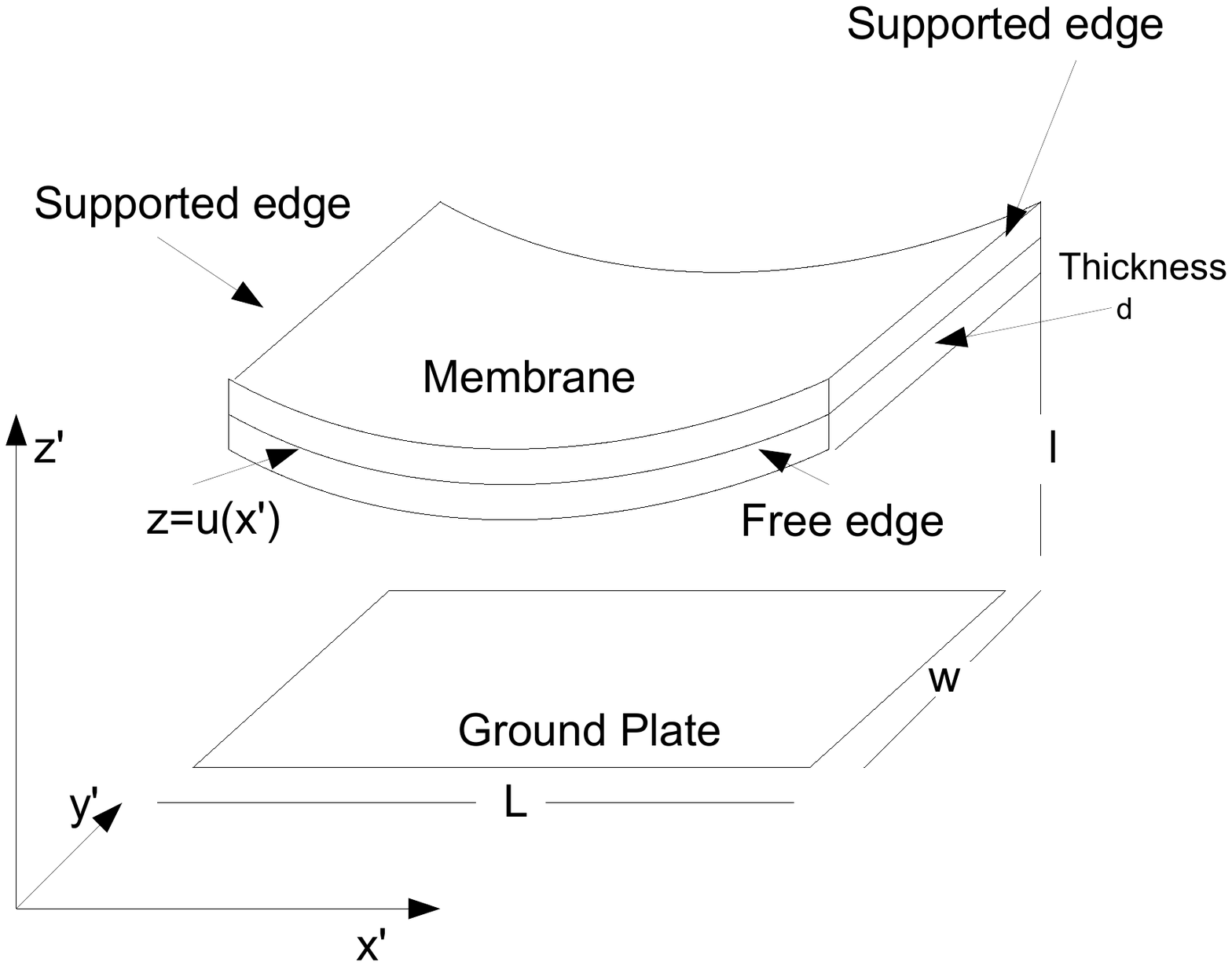}\vspace{-6cm}
  \caption{\it Schematic representation of a MEMS device }
\label{Figmems1}
\end{figure}
%_____________________________________________________________
 A potential difference $V$  is applied  between the top surface  and the rigid plate and we further assume that the plate is earthed. Besides,  the small aspect ratio of the
gap gives potential
\begin{equation}
\phi = V(l - z')/(l - u'), \label{potential}
\end{equation}
to leading order, where $u'$ is the displacement of the membrane towards the plate ($u' = l$ corresponds to touch-down, i.e. when the top surface touches the rigid plate) and $z'$ is the distance measured from the undisturbed membrane position towards the plate. The electrostatic force per unit area on the membrane (in the $z'$ direction) is
then
\begin{center}
$\frac 12 \times \mbox{surface charge density} \times \mbox{electic field} = \frac 12 \epsilon_0
\phi_{z'}^2 = \frac 12 \epsilon_0 V^2/(l - u')^2 \, ,$
\end{center}
recalling that  $\epsilon_0$ is the permittivity of the free space.

We take the sides of  width $w$, say at  $x'=0$ and $x'=L$, to be
 connected with the support of the device,
  with those of length $L$, say
at $y'=0$ and $y'=w$, to be free. We also assume there is no variation in the $y'$ direction, so $u' =
u'(x',z',t')$ for time $t'.$ The surface density of the membrane is denoted by $\rho$, while
$T_m$ stands for constant surface tension of the membrane. Then its displacement satisfies the forced wave equation with  damping (proportional to the membrane speed),
\begin{equation}
\rho u'_{t't'} + a u'_{t'}= T_m u'_{x'x'} + \frac 12 \epsilon_0 V^2/(l - u')^2 . \label{wave1}
\end{equation}
In many situations is observed  that  the  damping term is dominant compared with the inertia term. According to this ansatz   we get the following parabolic equation
\begin{equation}
 a u'_{t'}= T_m u'_{x'x'} + \frac 12 \epsilon_0 V^2/(l - u')^2 . \label{par1}
\end{equation}
 In addition to the derived equation \eqref{par1}, appropriate boundary conditions should be imposed. The standard way to do so
 is to assume that since the edges of the membrane or beam are fixed at the support of the device, Dirichlet boundary conditions, in the case of the flexible membrane or clamped boundary conditions, in the case of a beam should be considered. Although as it is stated in \cite[Chapter 6]{y}  it is evident that the support or cantilever of MEMS devises might be nonideal and flexible.

 More specifically cantilever microbeams can tilt upward or downward due to the deformation of their support since
 the anchors or supports of them  can have some flexibility making the assumption of perfect clamping inaccurate.
 This  flexibility of the supports of microbeams are accounted for by assuming springs at the beam boundaries and consequently
 modeling  a flexible nonideal support can be done in general by assuming torsional and translational springs at the membrane or beam edge.

 As a first step towards this modelling approach in this work we will assume that we have a device for which its movable upper part   is thin enough,  so that it can be considered to behave as a membrane while its ends are connected with a flexible nonideal  support behaving as a spring moving in the $x'$-direction, see Figure \ref{Figmems2}(a). Torsional or other kind of behaviour is assumed to be negligible at this occasion.

 Therefore according to the above assumptions the appropriate boundary conditions should be those of Robin type and thus we set
  $$u'_{x'} (-L,t')=k u'(-L,t'), \quad u'_{x'} (L,t')=- k u'(L,t'),$$
where $k$ is the spring constant.
%   %____________________________________________________________

\begin{figure}[!htb]\hspace{.5cm}
   \begin{minipage}{0.48\textwidth}
     \centering
     \includegraphics[width=1\linewidth]{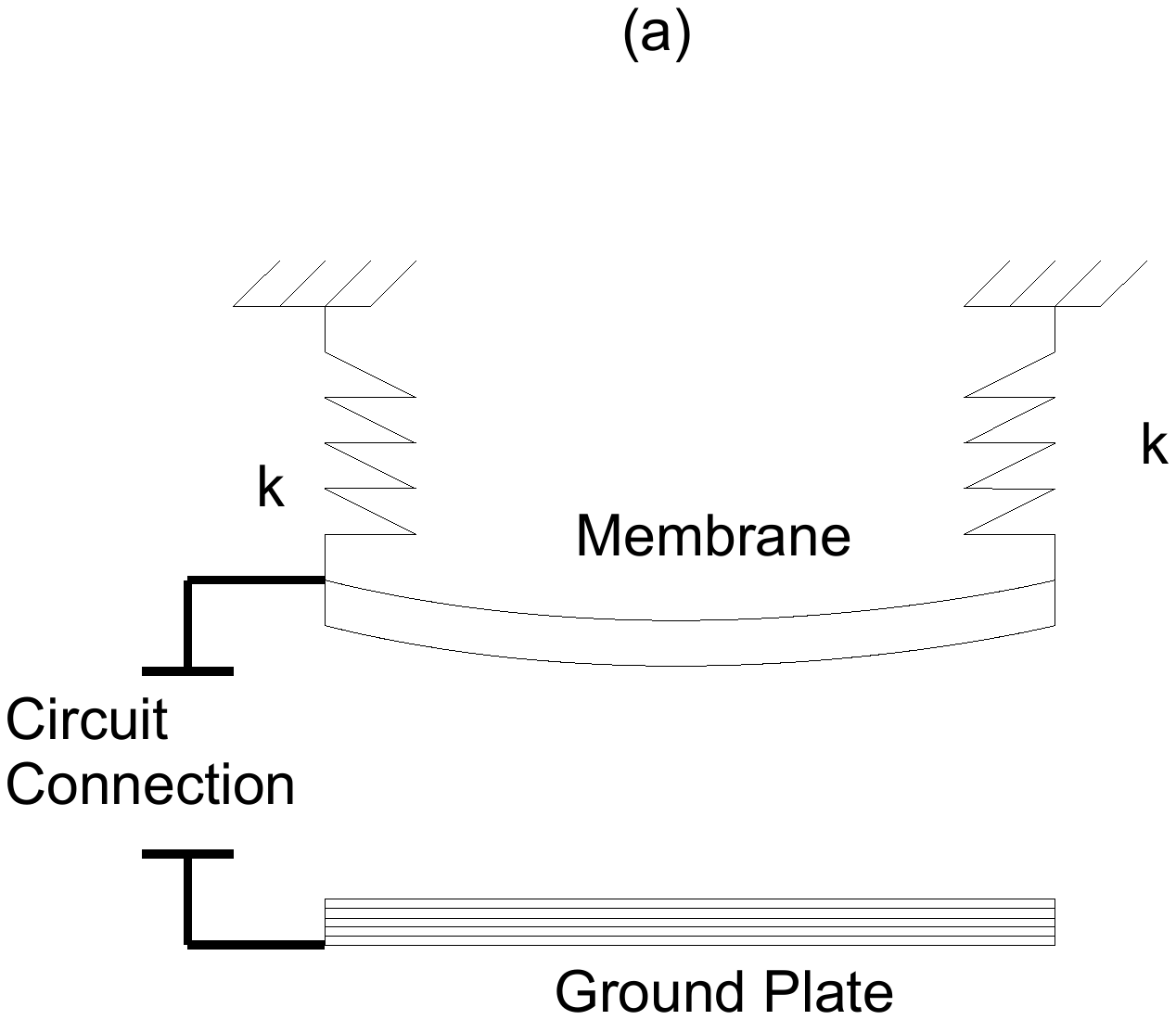}\vspace{-5cm}
   \end{minipage}\hfill
   \begin{minipage}{0.48\textwidth}
     \centering\vspace{-3cm}
     \includegraphics[width=1\linewidth]{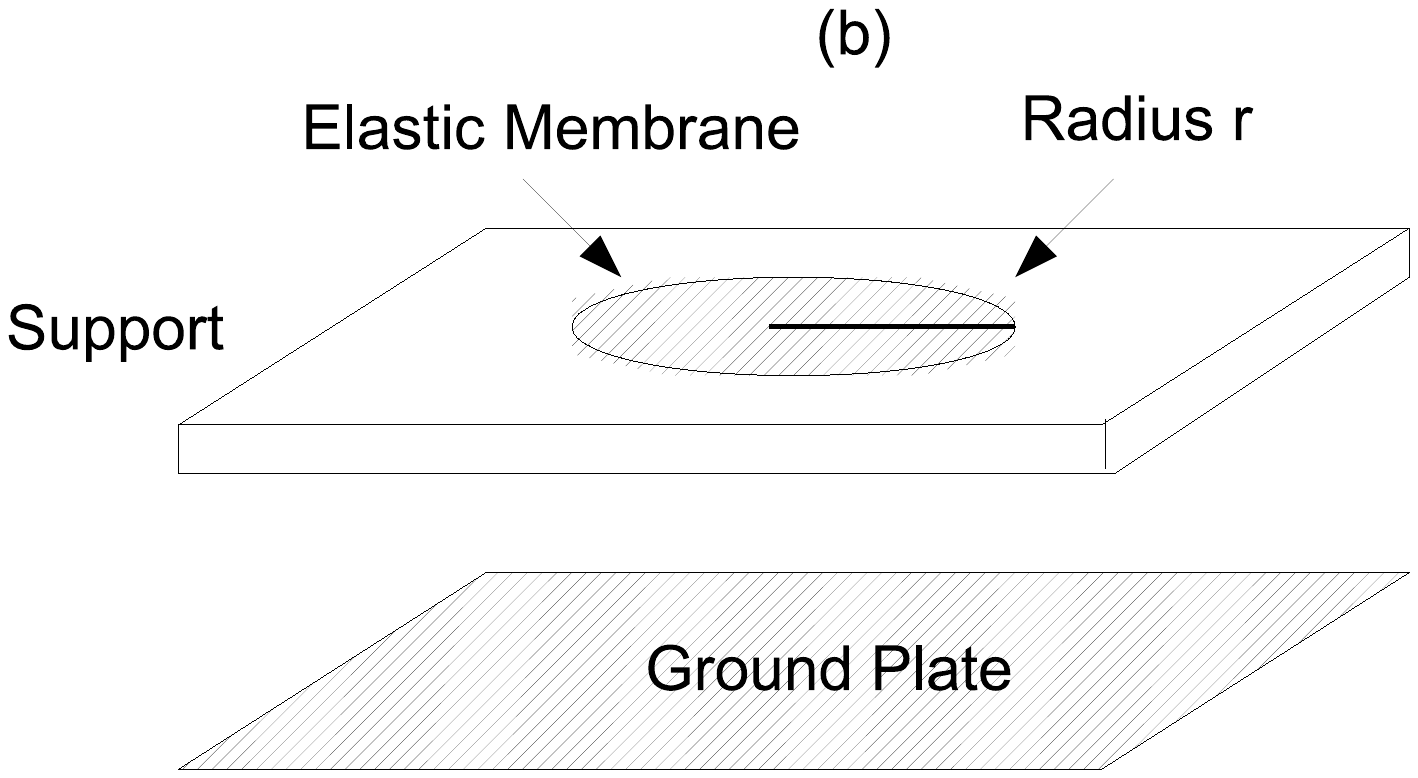}\vspace{-3cm}
   \end{minipage} \caption{(a) Schematic representation of a MEMS device with nonideal support. (b)
   Schematic representation of a MEMS device with radial symmetry}\label{Figmems2}
\end{figure}
%________________________________________________________

  %______________________________________________________________

 Next by introducing the scaling $u' = lu$, $x' = Lx$, $t'=   \frac{L^2 a}{T_m} t$, we end up with the local equation
\begin{equation}
u_{t} = u_{xx} + (\epsilon_0 V^2 L^2/T_m {2l^3})/(1 - u)^2, \label{par2}
\end{equation}
associated with  the aforementioned boundary conditions and some appropriate initial deformation $0<u(x,0)<1.$ Therefore  we end up in the first place with the following  local  problem:
\bse\label{o.onel1d}
 \be\label{o.one1l}
u_t =u_{xx}+ \frac{\lambda}{\left(1-u\right){^2}},\quad
  -1<x<1, \quad t>0,
 \ee
 \be \label{o.one2l}
u_x(\mp1,t)=\pm \beta u(\mp1,t),  \quad t>0,
  \ee
  \be\label{o.one3l}
 u(x,0)=u_0(x), \quad -1<x<1,
   \ee \ese
 for $\beta=Lk$ and $\lambda=\frac{\epsilon_0 V^2 L^2}{T_m {2l^3}}.$

Since pull-in instability is a
ubiquitous feature of electrostatically actuated systems, many
researchers have focused on extending the stable operation of electrostatically
actuated systems beyond the pull-in regime. In particular, in \cite{SC97a, SC97b}
the basic capacitive control scheme was first proposed by Seeger and Crary
to elaborate this kind of stabilization, see also \cite{CD99}. More precisely, this scheme provides  control of the voltage by the addition of a series capacitance to the circuit containing the MEMS device, since the added capacitance
 acts as a voltage divider. So in the event  the MEMS device, which has a capacitance
$C$ depending on displacement,  is connected in series with a capacitor of fixed capacitance $C_f$ and a source of fixed voltage $V_s$, we have that
 \[ V_s = \frac{Q}{C_c} = Q\left(\frac{1}{C} + \frac{1}{C_f}
\right),\]
where $Q$ is the charge on the device and fixed capacitor, and $C_c$ the series capacitance of the two. Then the potential difference $V$ across the MEMS device, by applying Kirchoff's law is equal to
\begin{equation}
V = \frac{V_s}{1 + C/C_f}.  \label{potdif}
\end{equation}
In addition  we also have
\[
Q = \epsilon_0 \int_0^w \int_0^L \phi_{z'}(x,y,0) \, dx' \,dy' =
 V \frac{wL \epsilon_0}l \int_0^1 \frac 1{1-u} \, dx \, ,
\]
and by using relation (\ref{potential}) we get,
\[
 C = C_0 \int_0^1 \frac 1{1-u} \, dx,
\]
for $C_0 =\frac{wL\epsilon_0}{l}$  being  the capacitance of the undeflected device.

When  the latter relation  is combined with equations (\ref{potdif}) and
(\ref{par2}) we finally obtain the nonlocal problem
\bse\label{o.onenl}
\be\label{o.onenl1}
u_{t} = u_{xx} +  \frac \lambda {(1 - u)^2 \left( 1 + \alpha \int_{-1}^1 \frac 1{1-u} \, dx\right)^2},\; -1<x<1, \quad t>0,
 \ee
 \be \label{o.onenl2}
u_x(\mp1,t)=\pm \beta u(\mp1,t),  \quad t>0,
  \ee
  \be\label{o.onenl3}
 u(x,0)=u_0(x), \quad -1<x<1,
   \ee
\ese
with
$\alpha = \frac{wL\epsilon_0}{l C_f}=\frac{C_0}{C_f}$.

 Usually  it is supposed that the elastic membrane is initially in  its unforced position , so that
$u(x,0)\equiv 0.$ However, in this work, we consider
more general non-negative initial conditions, reflecting also the situation when the membrane has an initial displacement.

It has been experimentally observed  that the applied voltage $V_s$
controls the operation of the MEMS device. Indeed,  when $V_s$ exceeds a critical threshold $V_{cr}$, called the {\it pull-in voltage}, then the phenomenon of
{\it touch-down} (or {\it pull-in instability} as it is also known
in MEMS literature) occurs when the elastic membrane touches
the rigid ground plate. The related mathematical  problem has been studied quite extensively in e.g.
\cite{EGG10, G08, YG-ZP-MJW06, KMS08, KLNT15, KS18,  L89, PT01a, Pelesko2}.

Note that the limiting case $\alpha=0$ corresponds to the configuration
where there is no capacitor in the circuit and then  we end up with the local problem \eqref{o.onel1d}, which has been studied in \cite{Guo91}.  A stochastic version of  problem \eqref{o.onel1d} is treated  in \cite{DKN20, K16}.  Besides,   the local problem with  Dirichlet boundary conditions ($\beta=+\infty$) has been extensively studied among others in \cite{EGG10, G08, KMS08, KS18}. Also,  for hyperbolic modifications of  the variation of  \eqref{o.onel1d} an  interested reader can check \cite{G10,  KLNT15}.

 The quenching behaviour of the  nonlocal equation \eqref{o.one1N}  associated  with Dirichlet boundary  ($\beta=+\infty$)   has been  treated in \cite{KLN16} and in references therein as well as in  \cite{GHW08, GN12, H11}. Also, non-local alterations of  parabolic and hyperbolic problems arising in MEMS technology were tackled  in \cite{ DZ19, GHW08, GKWY20, GN12, KLNT11, KLN16, KS18}. However  to the best of our knowledge there are not similar studies available in the literature  for the Robin problem ($0<\beta<+\infty,$)  so in the current work we study problem  \eqref{o.oneN1} and we extend  some of the results given in  \cite{Guo91} for the local problem,  but we also deliver a further investigation related to the steady-state problem and the quenching behaviour of the time-dependent problem. Our mathematical analysis is inspired by ideas  developed in  \cite{GN12, KLN16}, however important modifications are necessary due to the Robin boundary conditions. In particular, a new  Poho\v{z}aev's type identity for Robin boundary conditions is derived which is then used to derive lower estimates of the pull-in voltage. Moreover, a novel argument,  see Theorem \ref{thm:bound}, is developed  to derive an upper estimate of the quenching rate; note that such  a reasoning is missing  from the approach used in \cite{KLN16}. Still, the derivation of a key estimate for the nonlocal term, analogous to the one derived in \cite[Lemma 3.3 ]{KLN16} for the Dirichlet problem, needs more work for Robin problem \eqref{o.oneN1} and it is finally derived under some extra restriction,  cf. Lemma \ref{lem2}.

 The organization of the paper is as follows. In section~\ref{ss}
a thorough study of the steady-state problem is delivered, where among other results some estimates of the supremum of its spectrum (pull-in voltage) are derived.  Uniqueness and local-in-time existence results for  time-dependent  problem \eqref{o.oneN1}  are discussed in the first part of section~\ref{rnc5}. The second part of section~\ref{rnc5} deals with the long-time behaviour of  the solutions of   \eqref{o.oneN1}. In particular, at first  a quenching result is obtained for a genericl domain, whilst a sharper quenching result is derived  for a radially symmetric domain later on. A numerical treatment of  \eqref{o.oneN1} via an adaptive method is presented in section~\ref{nap}. We thus numerically verify all the obtained analytical results as well as we determine the quenching profile which cannot be derived via our theoretical approach. We conclude with a discussion of our main results in section ~\ref{dsc}.

%_____________________________________________________________
\section{Steady-State Problem: estimates of the pull-in voltage}\label{ss}
%____________________________________________________________
The main purpose of the current section is to study the steady-state problem of  \eqref{o.oneN1}. In particular, we are interested in obtaining estimates of the supremum of its spectrum (pull-in voltage) whilst in the one-dimensional case we are also  able to derive  the form of its bifurcation diagram.
\subsection{The one-dimensional case}\label{ss1}
Below we provide  a thorough investigation of the steady-state problem in the one-dimensional case. In particular we study the structure of the solution set of
\bse\label{nss}
\bge\label{nss1}
&&w''+ \frac{\la}{(1-w)^2 \left[1+\alpha\int_{-1}^1 \frac{\dd x}{1-w}\right]^2}=0, \;\;-1<x<1,\\
&&w'(-1)-\beta w(-1)=0,\quad w'(1)+\beta w(1)=0, \label{nss2}
 \ege
\ese
where we always have $0\leq w <1$ in $[-1,\,1]$
for a (classical) solution of \eqref{nss}.

For convenience  we set $W = 1 - w$ and then (\ref{nss}) becomes
\bse\label{ssW}
\bge
&&W''=\frac {\mu}{W^2} \, ,\quad -1<x<1,\label{ssWa}\\
&&W'(-1)+\beta\left (1-W(-1)\right)=0  \, ,\quad W'(1)-\beta\left (1-W(1)\right)=0  \, ,
\ege
\ese
where
\begin{equation}\label{dim}
\mu =\frac {\lambda} {\left[ 1 + \alpha\int_{-1}^{1}\frac{1}{W} dx\right]^{2}}.
\end{equation}
 Note that $W$ is symmetric and thus $m=\min\{W(x), x\in [-1.1]\}=W(0)$, cf. \cite{Gi-Ni-Ni, GKWY20}. Then multiplying both sides of equation (2.2a) by $W'$ and integrating from $m=W(0)$
 to $W(x)=W$ we derive
\[
\int_{0}^{W'} W' dW' = \int_{0}^x W'' W'dx =
\mu \int_{0}^x \frac{W'}{W^2} \, dx = \mu
\int_m^W \frac{dW}{W^2} \, ,
\]
hence
\begin{eqnarray}\label{eqW_x}
\frac12 {(W')}^2=\mu\left(\frac{1}{m}-\frac{1}{W}\right).
\end{eqnarray}
This gives equivalently
\begin{eqnarray}\label{eqdxdW}
\frac{dx}{dW} = \sqrt{\frac{m}{2\mu}}\sqrt{\frac{W}{W-m}} \, ,
\end{eqnarray}
which implies  (see ~\cite{GKWY20, KLNT11})
\be\nonumber
x= \sqrt{\frac{m}{2\mu}}
\left[ \sqrt{W(W-m)} - \frac12m\ln(m) + m\ln \left( \sqrt{W} +\sqrt{W-m} \right) \right] \, .
\ee
 Additionally at the point $x=1$
and for  $W(1)=M:=\max\{W(x),\;x\in [-1,1]\} $  we deduce
\be \label{eq-mM1}
1= \sqrt{\frac{m}{2\mu}}
\left[ \sqrt{M(M-m)} - \frac12m\ln(m) + m\ln \left( \sqrt{M} +\sqrt{M-m} \right) \right] \, .
\ee
Moreover combining the boundary condition,  $W'(1)=\beta\left (1-W(1)\right) $,  with
equation \eqref{eqW_x} we obtain
\be\label{eq-mM2}
\frac{\beta^2(1-M)^2}{2}=\mu\left(\frac{1}{m}-\frac{1}{M}\right).
\ee

At this point, recalling that for $\alpha=0$ we have $\mu=\lambda,$ we can obtain the bifurcation diagram of the local
problem. More specifically rearranging \eqref{eq-mM2}, we have
\be \label{bifurc_l2}
m=\frac{2\lambda M}{2\lambda +M\beta^2 (1-M)^2},
\ee
which together with \eqref{eq-mM1}, for $\mu=\lambda$, namely
\be \label{bifurc_l1}
1= \sqrt{\frac{m}{2\lambda}}
\left[ \sqrt{M(M-m)} - \frac12m\ln(m) + m\ln \left( \sqrt{M} +\sqrt{M-m} \right) \right] ,
\ee
 forms a system of algebraic equations giving an implicit relation of the form $F(\lambda, M)=0$.

Furthermore in order to obtain the bifurcation diagram for the nonlocal problem ($\alpha>0$) we have to express the integral of the nonlocal term in terms of $\lambda, m, M$.

That is, on using equation  \eqref{eqdxdW}
\begin{eqnarray*}
&&\int_{-1}^1 \frac{1}{W}dx
 =  \int_{-1}^1\frac{dx}{dW} \frac{dW}{W} = 2
\sqrt{\frac{m}{2\mu}} \int_m^M\frac{1}{\sqrt{W(W-m)}}dW \nonumber \\
& = &2
\frac{1}{ \sqrt{M(M-m)} - \frac12m\ln(m) + m\ln \left( \sqrt{M} +\sqrt{M-m} \right)  }
\int_m^1\frac{1} {\sqrt{W(W-m)}} dW
\nonumber \\
& = & 2
\frac{1}{ \sqrt{M(M-m)} - \frac12m\ln(m) + m\ln \left( \sqrt{M} +\sqrt{M-m} \right)  }
 \ln \left(\frac{2M - m + 2\sqrt{M(M-m)}}{m} \right).
\end{eqnarray*}

Therefore, using also   \eqref{dim}, \eqref{eq-mM2} to eliminate $\mu$,  we obtain the following system of algebraic equations for $\lambda$, $M$, $m$:
\bse
\be\label{gk1}
1= \sqrt{\frac{m}{2\mu}}
\left[ \sqrt{M(M-m)} - \frac12m\ln(m) + m\ln \left( \sqrt{M} +\sqrt{M-m} \right) \right],
\ee
\be\label{gk2}
\frac{\beta^2(M-1)^2}{2}\frac{mM}{M-m}=
\lambda
{\left[1+
\alpha \frac{2 \ln \left(\frac{2M - m + 2\sqrt{M(M-m)}}{m} \right)}
{ \sqrt{M(M-m)} - \frac12 m\ln(m) + m\ln \left( \sqrt{M} +\sqrt{M-m}\right) }\right]^{-2}},
\ee
\ese
together with \eqref{dim}, which can be solved numerically.
\begin{rem} Note that  by equation \eqref{eq-mM2} for $\beta \gg 1$ we have $(M-1)\sim 0$ or $M\sim 1$
and we retrive the expression for $\lambda$ and $m$ which gives the bifurcation diagram for the local problem with Dirichlet boundary conditions (see  \cite{KLNT11} ), i.e.
\[
\lambda=\frac{m}{2}\left[\sqrt{1-m}-\frac12m\ln(m)+m\ln \left(1+\sqrt{1-m} \right) \right]^2.
\]
\end{rem}

In Figure \ref{FigLBif_ab}(a) we plot the bifurcation diagram for the stationary  local problem (\ref{nss1}) for $\alpha=0$.  We can observe the existence of a critical value of the parameter $\lambda$, say $\lambda^*,$ usually called the {\it pull-in voltage} in MEMS literature,  above which  we have no solution for the steady problem while for values below $\lambda^*$ we have two solutions. We finally derive that
$\lambda^*=0.108711900526435$
 and for this value we have that the maximum of the solution
 $M=W(1)=0.761$.

Regarding the nonlocal stationary problem, equation (\ref{nss1})  with $\alpha =1$ we present a similar plot of the bifurcation diagram in Figure \ref{FigNLBif_ab}(a) (line indicated with $\alpha=1$).
%\ref{FigwBNL}.
In this case the critical value of the parameter $\lambda$ is
 $\lambda^*=2.387086785660011$.
%____________________________________
 In both of  the above cases the parameter in the boundary conditions is taken to be $\beta=1$.

 In this set of graphs we can see also the variation of the bifurcation diagram of the local problem
  with respect to the parameter $\beta$ in Figure \ref{FigLBif_ab}(b).

 %_____________________________________________________________
 %%____________________________________________________________
\begin{figure}[!htb]\vspace{-4cm}\hspace{-2cm}
   \begin{minipage}{0.45\textwidth}
     \centering
     \includegraphics[width=1.2\linewidth]{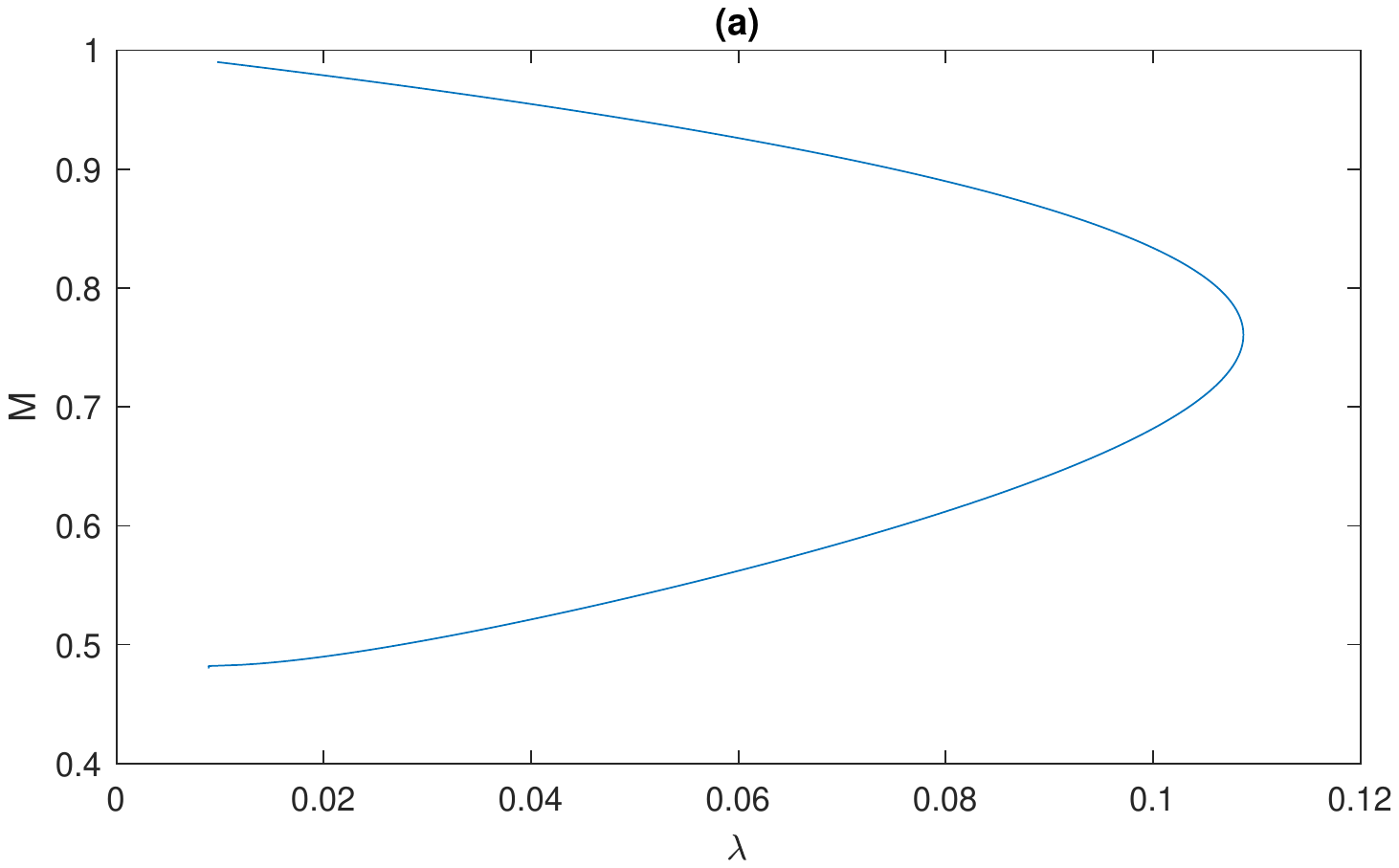}\vspace{1cm}
   \end{minipage}%\hfill
   \begin{minipage}{0.48\textwidth}
     \centering\vspace{-.8cm}
     \includegraphics[width=1.2\linewidth]{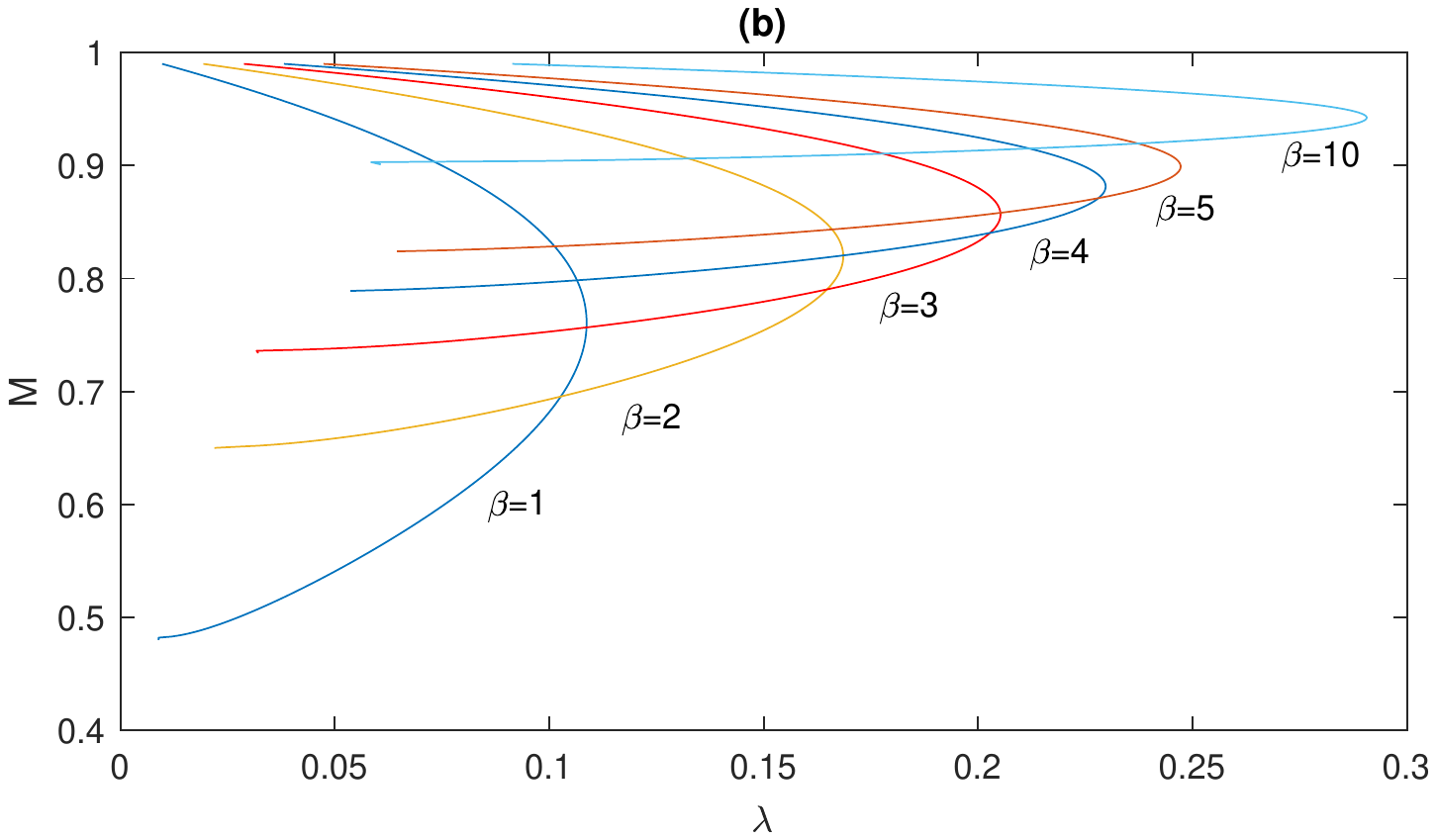}%\vspace{-3cm}
   \end{minipage} \vspace{-5cm}
   \caption{(a) Bifurcation diagram for the local problem.
   (b)  Variation of the bifurcation diagramm of the local problem with respect to the parameter $\beta$.}\label{FigLBif_ab}
\end{figure}
%________________________________________________________
 A similar graph, see Figure \ref{FigNLBif_ab},  investigates the variation of the bifurcation diagram of the nonlocal problem with respect to the parameter $\alpha$ in Figure \ref{FigNLBif_ab}(a) and with respect to the parameter $\beta$ in Figure \ref{FigNLBif_ab}(b).
%_____________________________________________________________
\begin{figure}[!htb]\vspace{-4cm}\hspace{-2cm}
   \begin{minipage}{0.48\textwidth}
     \centering
     \includegraphics[width=1.2\linewidth]{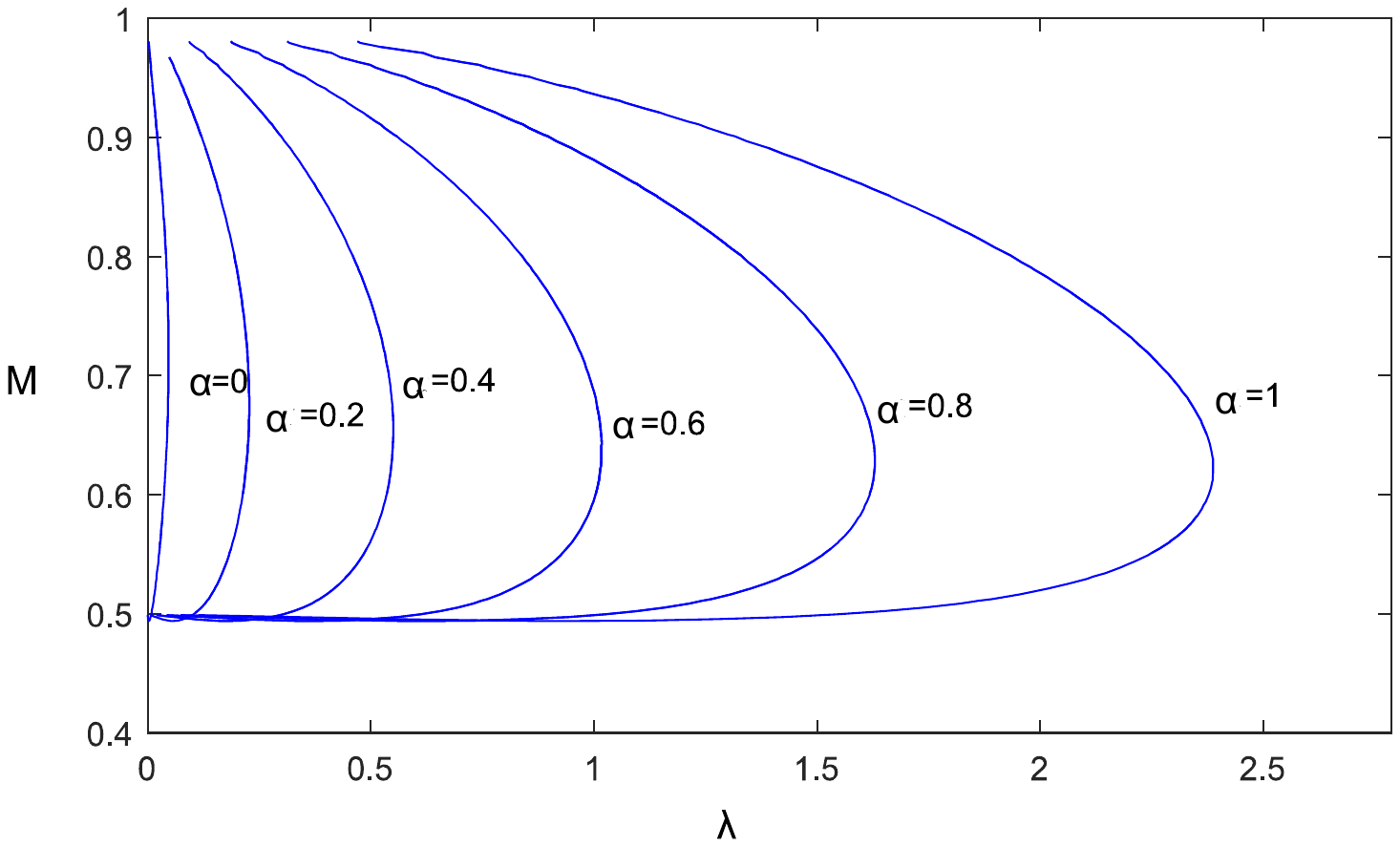}%\vspace{-6cm}
   \end{minipage}%\hfill
   \begin{minipage}{0.48\textwidth}
     \centering\hspace{-1cm}
     \includegraphics[width=1.2\linewidth]{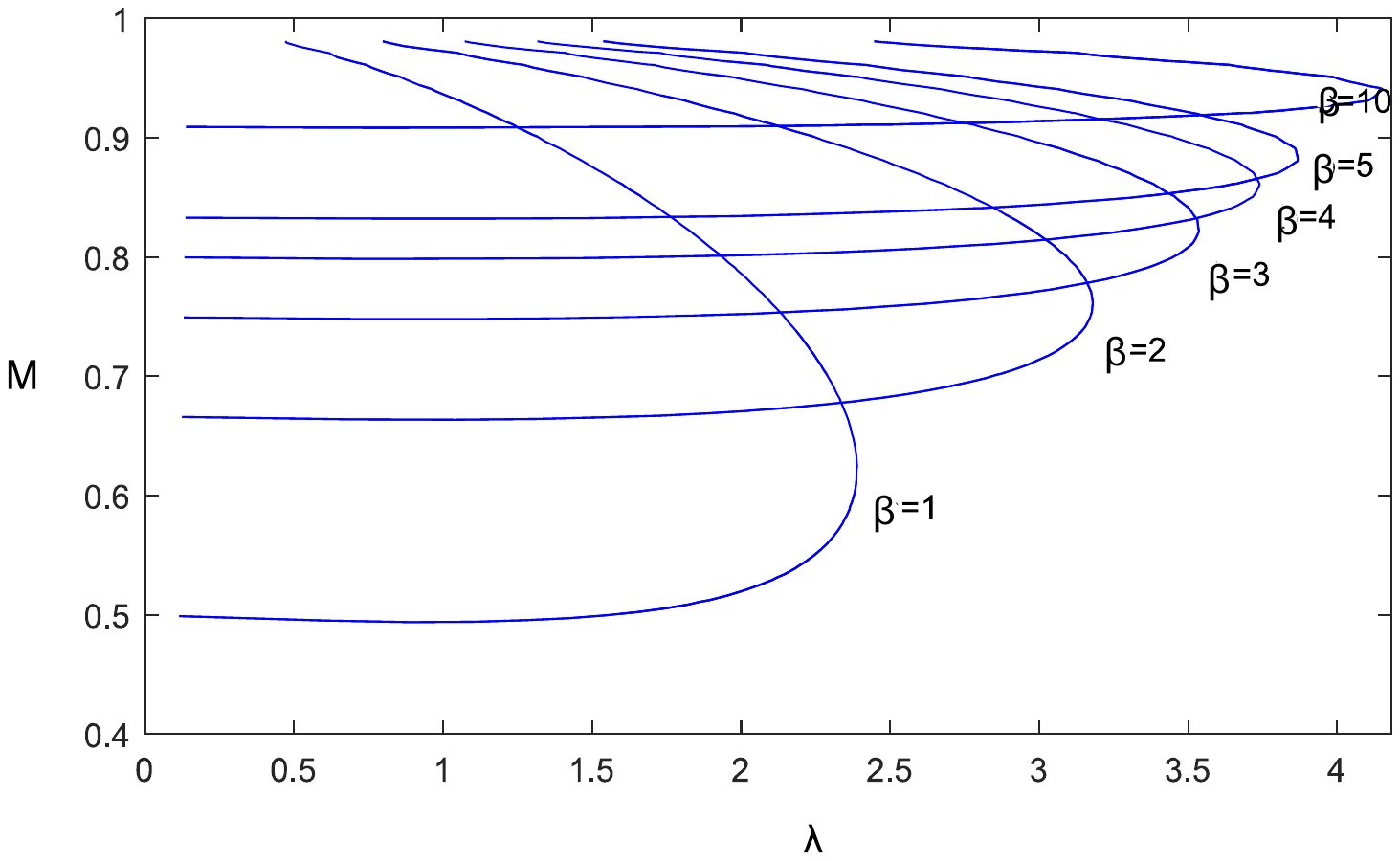}
   \end{minipage} \vspace{-4cm}
   \caption{(a) Variation of the bifurcation diagramm of the nonlocal problem with respect to the parameter $\alpha$  for $\beta=1$. (b)
  Variation of the bifurcation diagramm of the nonlocal problem with respect to the parameter $\beta$ for
   $\alpha=1$.}\label{FigNLBif_ab}
\end{figure}
%________________________________________________________
%_________________________________________________________

  For a rigorous bifurcation analysis of  nonlocal problem \eqref{nss}  we kindly advice the reader to check \cite{GKWY20}.
\subsection{The higher dimensional case}

In this part  we study the steady-state problem of  the $N$-dimensional version  of \eqref{o.oneN1} for $N>1.$
In particular we  perform an investigation of the set of classical solutions $0\leq w=w(x) < 1$ in $\bar{\Omega},$  satisfying the nonlocal problem
\bse\label{ssN}
\bge \label{ssN1}
&&\Delta w + \frac{\lambda }{(1-w)^{2} \left(1+ \alpha \int_{\Omega} \frac{1}{1-w} dx\right)^2} = 0 , \ x \in \Omega\subset \mathbb{R}^N,\; N\geq 1,\\
&&
 \frac{\partial w}{\partial \nu} + \beta w=0 ,\; x \in \partial \Omega.\label{ssN2}
\ege
\ese
In the  following  we denote
\begin{eqnarray}
\label{3}
\lambda^*:= \sup \{\lambda>0 :\ \mbox{problem \ (\ref{ssN})\ has a classical solution}\},
\end{eqnarray}
and we recall that $\la^*$ in MEMS terminology is called {\it pull-in voltage}.
By setting
\begin{equation}
\label{28}
\mu:= \frac{\lambda}{K}= \frac{\lambda}{(1+\alpha  \int_{\Omega} \frac{1}{1-w} dx)^2},
\end{equation}
where
\bge\label{ak1}
K=K(w):=\left(1+ \alpha \int_{\Omega} \frac{1}{1-w} dx\right)^2,
\ege
then (\ref{ssN}) can be written as a local problem
\bse \label{29}
\bge
&&\Delta w+ \frac{\mu}{(1-w)^2}=0,\; x \in \Omega,\label{29a}\\
&&\frac{\partial w}{\partial \nu} +\beta w=0,\; x \in \partial \Omega, \label{29b}
\ege
\ese
and we also define
\begin{eqnarray}
\label{3A}
\mu^*:= \sup \{\mu>0 :\ \mbox{problem \ (\ref{29})\ has a classical solution}\}.
\end{eqnarray}
It is readily seen that problems (\ref{ssN}) and (\ref{29}) are equivalent via  relation (\ref{28}). More specifically $w$ is a solution of (\ref{ssN}) corresponding to $\lambda$ if and only if $w$ satisfies (\ref{29}) for $\mu$ given by (\ref{28}).

 Next we introduce the notion of weak solution for the problem (\ref{ssN}) which will be used in an essential  way to our approach (cf. \cite{KLN16}) towards the study of the quenching (touching down)  phenomenon.
\begin{defn}\label{wes}
 A function $w \in H_0 ^1 (\Omega)$ is called weak finite-energy solution of (\ref{ssN}) if there exists a sequence
 $\{w_j\}_{j=1} ^ {\infty}\in C^2(\Omega)\cap C(\Omega) $ satisfying as $j \rightarrow \infty$

\bse\label{def_weak}
\bge
&&w_j \rightarrow w \ \mbox{weakly} \ in \ H^1 (\Omega), \label{4}\\
&&w_j \rightarrow w\quad\mbox{a.e.},  \label{5}\\
&&\frac{1}{(1-w_j)^2}\rightarrow \frac{1}{(1-w)^2}\;\;\mbox{in}\;\; L^1(\Omega), \label{6}\\
&&\frac{1}{(1-w_j)}\rightarrow \frac{1}{(1-w)} \;\;\mbox{in}\;\; L^1(\Omega), \label{7}\\
&&\Delta w_j + \frac{\lambda}{(1-w_j)^2 (1+ \alpha\int_{\Omega} \frac{dx}{1-w_j})^2} \rightarrow 0 \;\;\mbox{in}\;\;  \ L^2 (\Omega). \label{8}
\ege
\ese
 A weak finite-energy solution of (\ref{ssN}) satisfies
$$- \int_{\Omega} \nabla \phi \cdot \nabla w\, dx
+\int_{\partial \Om} \phi\, \frac{\pl w}{\pl \nu} \,ds +
\lambda \frac{\int_{\Omega} \frac{\phi}{(1-w)^2} dx}{(1+ \alpha\int_{\Omega}\frac{1}{1-w} dx)^2} =0,$$
for any $\phi\in W^{2,2}(\Om)$ satisfying $\frac{\partial \phi}{ \partial \nu}+ \beta \phi=0$ on $\pl \Omega.$

We also denote
\begin{eqnarray*}
\hat{\lambda}:= \sup \{\lambda>0 :\ \mbox{problem \ (\ref{ssN})\ has a weak finite-energy  solution}\}.
\end{eqnarray*}
\end{defn}
In addition and in accordance to \cite[Proposition 2.2]{KLN16} we have the following:
\begin{prop}\label{cn3}
For the radial symmetric case, i.e. when $\Omega=B_1(0):=\{x\in \R^N: |x|<1\},$ the suprema of the spectra for classical and weak energy solutions are identical.  In particular,
 $\lambda^*= \hat{\lambda}.$
\end{prop}
The proof of Proposition \ref{cn3} follows closely the proof of \cite[Proposition 2.2]{KLN16} and so it is omitted.

Next we show that $\mu^*$ defined by \eqref{3A} is well defined and bounded.  More precisely,
\begin{lem}\label{30a}
There exists a finite $\mu^*$ defined by \eqref{3A} such that
\begin{enumerate}
\item If $\mu<\mu^*$ then problem \eqref{29} has at least one (classical) solution.
\item If $\mu>\mu^*$ then problem \eqref{29} has no (classical) solution.
\end{enumerate}
\end{lem}
\begin{proof}
We first establish the existence of $\mu^*$ defined by \eqref{3A}.
Indeed, implicit function theorem implies that problem \eqref{29} has a solution bifurcating from the trivial solution $w=0$ at $\mu=0.$ This solution is positive due to the maximum principle, hence  $\mu^*$ is well-defined and positive.

Next we prove the  boundedness of $\mu^*.$ Let $(\lambda_1,\phi_1(x))$ be the principal normalized eigenpair of the  Laplacian associated with Robin boundary conditions, i.e. $\phi_1(x)$ satisfies
\bge \label{egv}
-\Delta \phi_1= \lambda_1 \phi_1, x \in \Omega,\quad \frac{\partial \phi_1}{\partial \nu}+ \beta \phi_1=0 ,\; x \in \partial \Omega,
\ege
with
\bge\label{ik1}
\int_{\Omega} \phi_1\,dx=1.
\ege
It is known (see, for example, \cite[Theorem 4.3]{Am76}) that $\lambda_1$ is positive and that $\phi_1(x)$ does not change sign in $\Omega$, so  by condition \eqref{ik1}  is positive.

Testing \eqref{29a} by $\phi_1(x)$ and using second Green's identity in conjunction with \eqref{ik1} we obtain for any calssical solution $w$
\bgee
\lambda_1\int_{\Omega} w \phi_1\,dx=\mu\int_{\Omega} \frac{\phi_1}{(1-w)^2}\,dx\geq \mu.
\egee
 The latter inequality, by virtue of \eqref{3A}, implies
\bgee
\mu^*\leq \lambda_1\int_{\Omega} w \phi_1\,dx\leq \lambda_1<\infty,
\egee
and so $\mu^*$ is finite.

Next we focus on  proving  statement $(i).$ We   pick $\mu\in(0,\mu^*),$ then thanks to the definition of $\mu^*$ there exists $\bar{\mu}\in (\mu,\mu^*)$ such that the minimal solution $w_{\bar{\mu}}$  (i.e. the smallest solution corresponding  parameter $\bar{\mu}$) of \eqref{29} satisfies
\bgee
&&-\Delta w_{\bar{\mu}}=\frac{\bar{\mu}}{(1-w_{\bar{\mu}})^2}\geq\frac{\mu}{(1-w_{\bar{\mu}})^2},\; x \in \Omega,\\
&&\frac{\partial w_{\bar{\mu}}}{\partial \nu} +\beta w_{\bar{\mu}}=0,\; x \in \partial \Omega,
\egee
since $\bar{\mu}>\mu.$ The latter implies that $w_{\bar{\mu}}$ is an upper solution of \eqref{29} corresponding to parameter $\mu.$ Additionally, it is easily seen that $w\equiv0$ is a lower solution of \eqref{29} corresponding to  $\mu.$ Consequently by using comparison arguments, cf. \cite{Pao92}, we can construct a solution of \eqref{29} corresponding to parameter $\mu,$ and this completes the proof of $(i).$ On the other hand, by the definition of $\mu^*$ we deduce that problem \eqref{29} has no solution for $\mu>\mu^*$ and statement $(ii)$ is also proven.
\end{proof}
Next we prove the monotonicity of minimal (stable) branch of problem \eqref{29} with respect to (local) parameter $\mu.$

\begin{lem} \label{30}
Let $\mu_1,\mu_2 \in (0,\mu^*).$ Assume that  $w_{\mu_1}$ and $w_{\mu_2}$ are the corresponding minimal solutions of problem \eqref{29}, then
\begin{equation} \label{31}
0<w_{\mu_1}(x)<w_{\mu_2} (x)<1 \quad for \quad x \in \Omega, \mbox{ if }\ 0<\mu_1< \mu_2 < \mu^*.
\end{equation}
\end{lem}
\begin{proof}
It is known, cf. \cite{EGG10, GN12}, that $w(x;\mu)$ is differentiable with respect to $\mu\in(0,\mu^*).$ Set $z=\frac{\partial w}{\partial\mu}$ then differentiating  \eqref{29a} with respect to $\mu$ we derive
$$-\Delta z  - \frac{2\mu}{(1-w)^3}z=\frac{1}{(1-w)^2}>0.$$
In addition due to the  boundary conditions we have similarly $\frac{\partial z}{\partial \nu} + \beta z> 0.$
 Therefore by the maximum principle, since $(1-w)^{-3}$ is bounded for a classical solution, cf.  \cite{Ev}, we obtain that $z> 0,$ that is $\frac{\partial w}{\partial\mu}> 0.$
\end{proof}
Using the preceding monotonicity result we can also prove, as in  \cite{GN12}, the following.
\begin{thm}\label{psks1}
There exists a classical  solution to  problem \eqref{ssN} for any
$\lambda \in (0,\,(1+\alpha\vert \Omega \vert )^2 \mu^*)$ and therefore
\bge\label{aal1}
\lambda^*\geq
   \sup_{(0,\mu^*)}\mu K(w_{\mu})\geq (1+\alpha\vert \Omega \vert )^2 \mu^*,
\ege
    where $\mu^*$ is defined by \eqref{3A}   and recall that $K(w_{\mu})=\left(1+ \alpha \int_{\Omega} \frac{1}{1-w_{\mu}} dx\right)^2.$
\end{thm}
\begin{proof}
By virtue of (\ref{31}) we have
\bge\label{ikk1a}
(1+\alpha\vert \Omega \vert )^2 = K(0)<  K(w_{\mu}),\quad if \quad 0<\mu < \mu^*.
\ege
Next, for any $\lambda \in (0,\,(1+\alpha\vert \Omega \vert) ^2 \mu^*)$ there is a unique $\mu \in (0,\mu^*)$ such that
\begin{equation} \label{32}
\mu = \frac{\lambda}{(1+\alpha\vert \Omega \vert) ^2},
\end{equation}
and hence there is a minimal solution $w_{\mu}$ for local problem (\ref{29}). Since  problems (\ref{ssN}) and (\ref{29}) are equivalent through  \eqref{28}, there exists  $\la_1\in (0,\la^*)$ with
\bge\label{ikk2}
\mu=\frac{\la_1}{K(w_\mu)}.
\ege
Therefore \eqref{32} and \eqref{ikk2} in conjunction with \eqref{ikk1a} imply  that $0<\la<\la_1<\la^*$ and thus nonlocal problem \eqref{ssN} has at least one (minimal) solution  $w_{\la}.$ This completes the proof.
\end{proof}
\begin{rem}
One can derive lower estimates of $\mu^*$ as in the case of Dirichlet boundary conditions, cf. \cite[Proposition 2.2.2]{EGG10}, and thus via \eqref{aal1} can finally obtain lower estimates of the pull-in voltage $\la^*.$
\end{rem}
Next we provide a more delicate  lower estimate of $\la^*$  in the case of the  $N-$dimensional sphere, i.e. when
\bgee
\Omega=B_R=B_R(0)=:\{x\in \R^N: |x|<R\},\;\mbox{for}\;R>0.
\egee
Such a radial symmetric case is rather of high importance from applications point of view as it is indicated in \cite{Ac69, PC03, T68}.
In order to prove such a lower estimate of $\la^*$  we need to use a
Poho\v{z}aev's type identity, cf. \cite{Poh}, for the following problem

 \bse\label{mkt1}
 \bge
&&\Delta v + \mu f(v)=0, \quad  x \in \Omega,\; \mu>0,\label{mkt1a}\\
&&\frac{\partial v }{\partial \nu }+ \beta v=0, \quad for \quad x \in \partial \Omega.\label{mkt1b}
 \ege
 \ese
Since to the best of our knowledge such an identity is not available in the literature for problem  \eqref{mkt1a}-\eqref{mkt1b},  we provide a proof  of it  below.
\begin{prop}\label{propPoh}
Let $f:\R\to \R$ be continuous with antiderivative  $F(v):= \int_0 ^v f(s) ds.$  Assume that $\Omega \subset \R^N$ is open and bounded. If $v\in C^2(\bar{\Omega})$ is a smooth solution of problem \eqref{mkt1} then the following identity holds
\bge
\label{33}
\frac{\mu (N-2)}{2} \int _\Omega v f(v) dx-\mu N \int _\Omega F(v) dx  &=&\frac{(N-2)}{2\beta}
 \int_ {\partial \Omega}\left(\frac{\partial v}{\partial \nu(x)}\right)^2 dS
 +\frac{1}{2} \int_{\pl \Omega} |\nabla v|^2\, \left\langle \nu(x), x\right\rangle\;dS\no\\
&&- \int_{\pl \Omega} \frac{\partial v}{\partial \nu(x)} \left\langle \nabla v, x\right\rangle\,dS
-\mu \int_{\pl \Omega} F(v) \frac{\partial }{\partial \nu(x)}\left(\frac{1}{2} |x|^2\right)\,dS,\qquad\quad
\ege

where $\left\langle \cdot,\cdot\right\rangle$ stands for the dot (inner) product in the Euclidean space $\R^N.$
\end{prop}
\begin{proof}
We first multiply \eqref{mkt1a} by $\left\langle x, \nabla v\right\rangle$ and integrate over $\Omega$ to derive

\bge\label{mkt2}
-\int_{\Omega} \Delta v \left\langle x, \nabla v\right\rangle\,dx=
\mu \int_{\Omega} f(v)\,\left\langle x, \nabla v\right\rangle\,dx.
\ege

The LHS of \eqref{mkt2} via integration by parts   gives
\bge\label{mkt3}
-\int_{\Omega} \Delta v \left\langle x, \nabla v\right\rangle\, dx=
\int_{\Om} \left\langle \nabla v, \nabla\left\langle x, \nabla v\right\rangle\right\rangle\, dx-
\int_{\pl \Om} \frac{\partial v}{\partial \nu(x)}\left\langle  x, \nabla v\right\rangle\, dS.
\ege
Now since
\bgee
\left\langle \nabla v, \nabla\left\langle x, \nabla v\right\rangle\right\rangle&&=\frac{1}{2}\left\langle \nabla\left(|\nabla v|^2\right), x\right\rangle+|\nabla v|^2\\&&=\frac{1}{2}\left\langle \nabla\left(|\nabla v|^2\right), \nabla\left(\frac{1}{2} |x|^2\right)\right\rangle+|\nabla v|^2,
\egee
using again integration by parts we obtain
\bge\label{mkt4}
\int_{\Om} \left\langle \nabla v, \left\langle x, \nabla v\right\rangle\right\rangle\, dx&=&
\frac{1}{2}\int_{\Om}\left\langle \nabla\left(|\nabla v|^2\right), \nabla\left(\frac{1}{2} |x|^2\right)\right\rangle\, dx+\int_{\Om} |\nabla v|^2\, dx \no\\&=&
\frac{1}{2}\int_{\pl \Om} |\nabla v|^2 \frac{\partial }{\partial \nu(x)}\left(\frac{1}{2} { |x^2|}\right)\, dS
-\frac{N}{2}\int_{\Om} |\nabla v|^2\,dx+\int_{\Om} |\nabla v|^2\, dx \no\\
&=&\frac{2-N}{2} \int_{\Om} |\nabla v|^2\, dx
+\frac{1}{2}\int_{\pl \Om} |\nabla v|^2 \left\langle \nu(x),x\right\rangle\, dS,
\ege
taking also into account that $\Delta\left(\frac{1}{2} |x|^2\right)=N$ for any $x\in \Om.$

Next we estimate the first term on the RHS of \eqref{mkt4} using \eqref{mkt1a}. Indeed multiplying \eqref{mkt1a} by $v,$  integrating over $\Om$ and using integration by parts we deduce
\bge\label{mkt5}
\int_{\Om} |\nabla v|^2\, dx &=&\mu\int_{\Om} vf(v)\, dx
+\int_{\pl \Om} v\frac{\pl v}{\pl {\nu(x)}} dS \no\\
&=&\mu\int_{\Om} vf(v)\, dx
-\frac{1}{\beta}\int_{\pl \Om} \left(\frac{\pl v}{\pl \nu(x)}\right)^2\, dS ,
\ege
where the last equality is a result of  boundary condition \eqref{mkt1b}.

Therefore \eqref{mkt3} in conjunction with \eqref{mkt4} and \eqref{mkt5}  implies
\bge\label{mkt6}
-\int_{\Omega} \Delta v \left\langle x, \nabla v\right\rangle\,  dx &=&
\frac{(2-N)}{2}\mu \int_{\Om} v f(v)\, dx
-\frac{(2-N)}{2\beta}\int_{\pl \Om} \left(\frac{\pl v}{\pl {\nu(x)}}\right)^2\,  dS
+\frac{1}{2}\int_{\pl \Om} |\nabla v|^2 \left\langle \nu(x),x\right\rangle\,  dS \no\\
&-&\int_{\pl \Om} \frac{\partial v}{\partial \nu(x)}\left\langle x, \nabla v\right\rangle\, dS .
\ege
Furthermore the RHS of \eqref{mkt2} by virtue of integration by parts implies
\bge\label{mkt7}
\mu \int_{\Omega} f(v)\,\left\langle x, \nabla v\right\rangle\, dx &=&
\mu \int_{\Omega}\left\langle x, \nabla F(v)\right\rangle\,  dx \no
=\mu \int_{\Omega}\left\langle \nabla\left(\frac{1}{2}|x|^2\right), \nabla F(v)\right\rangle\,  dx \no\\
&=&\mu \int_{\pl \Omega} F(v) \frac{\pl}{\pl { \nu(x)}}\left(\frac{1}{2} |x|^2\right)\, dS
-\mu N \int_{\Omega} F(v)\, dx ,
\ege
using again the fact that $\Delta\left(\frac{1}{2} |x|^2\right)=N.$

Consequently identity \eqref{33} arises immediately by \eqref{mkt6} and \eqref{mkt7}.
\end{proof}
Now we are ready to provide a rather measurable (computable)  lower estimate of $\la^*$ given by the following.
\begin{thm}\label{thPoh}
Consider problem (\ref{ssN})  defined  in $\Omega:=B_R=\{x\in \R^N: |x|<R\},\;\mbox{for}\;R>0.$
 Then if $N>2(1+\beta R)$  problem (\ref{ssN}) has a classical solution for any
\bge\label{ak9}
\la\leq \la_*:=\frac{\beta A\left(\partial B_R \right)(N-2)}{[N-2(1+\beta R)]}
\frac{(1+\alpha\omega_N^R)^2}{\omega_N^R},
\ege
where $A\left(\partial B_R\right)$ and $\omega_N^R$ stand for the area of the surface and the volume of the $N-$dimensional sphere $B_R$ respectively.
 Consequently, by \eqref{3}  there holds  $\la^*\geq \la_*.$
\end{thm}
\begin{proof}
Assume $0<\lambda<\lambda^*$,
 in which case problem (\ref{ssN}) has a classical solution,  and  we are working towards the derivation of estimate \eqref{ak9}.
Taking $f(w)=\frac{1}{(1-w)^2}$, hence $F(w)=\frac{w}{(1-w)},$ then Poho\v{z}aev's type identity \eqref{33}, for $\mu=\frac{\la}{K}$ and $K$  given by \eqref{ak1},
  infers
\bge
\label{ak2}
&&\frac{\la (N-2)}{2K} \int _{B_R} \frac{w}{(1-w)^2}  dx - \frac{\la N}{K} \int _{B_R} \frac{w}{1-w}  dx\no\\
&=&\frac{N-2}{2 \beta}  \int_ {\partial B_R} \left(\frac{\partial w}{\partial {\nu(x)}}\right)^2  dS+\frac{1}{2}  \int_ {\partial B_R}  |\nabla w|^2\left\langle \nu(x),x\right\rangle \,dS- \int_ {\partial B_R} \frac{\partial w}{\partial \nu(x)} \left\langle \nabla w,x\right\rangle\, dS\no\\
&& - \frac{\lambda}{K}\int_ {\partial B_R} \frac{w}{1-w} \frac{\partial}{\partial \nu(x)}\left(\frac{1}{2} |x|^2\right)\, dS \no\\
&=&\frac{[N-2(1+\beta R)]}{2\beta}
 \int_ {\partial B_R} \left(\frac{\partial w}{\partial {\nu(x)}}\right)^2  dS+\frac{R}{2} \int_{\pl B_R} |\nabla w|^2\, dS- \frac{\lambda}{K} \int_ {\partial B_R} \frac{w}{1-w}\left\langle x, \nu(x)\right\rangle\, dS, \qquad\qquad
\ege
using the fact that $ \frac{\partial}{\partial \nu(x)}\left(\frac{1}{2} |x|^2\right)=\left\langle x, \nu(x)\right\rangle$  and  $\left\langle \nu(x),x\right\rangle=R$ when  $\Omega=B_R.$   Notably for the case of Dirichlet boundary conditions the term
\bgee
\frac{\lambda}{K} \int_ {\partial B_R} \frac{w}{1-w}\left\langle x, \nu(x)\right\rangle\, dS
\egee
vanishes and then calculations in that case are simpler, which is not the case  for Robin boundary conditions. However, in the sequel we show that even for Robin  boundary conditions this term luckily can be estimated in the right direction. Indeed,
via the divergence theorem we have
\bgee
\int_ {\partial B_R} \frac{w}{1-w}\left\langle x, \nu(x)\right\rangle\, dS=\int_ {\partial B_R} \left\langle \hat{F}, \nu(x)\right\rangle\, dS=\int_{B_R} div(\hat{F})\,dx
\egee
where the vector field  $\hat{F}$ is defined by  $\hat{F}:=\frac{w}{1-w} x.$

Since $ div(\hat{F})=\frac{1}{(1-w)^2}\left\langle \nabla w,x\right\rangle+N \frac{w}{1-w}$ then
\bge
\int_ {\partial B_R} \frac{w}{1-w}\left\langle x, \nu(x)\right\rangle\, dS=\int_{B_R} \frac{1}{(1-w)^2}\left\langle \nabla w,x\right\rangle\,dx+N \int_{B_R}\frac{w}{1-w}\,dx
\ege
and thus by virtue of \eqref{ak2} we derive
\bgee
&&\frac{\la (N-2)}{2K} \int _{B_R} \frac{w}{(1-w)^2}  dx
- \frac{\la N}{K} \int _{B_R} \frac{w}{1-w} dx\no\\
&\geq&
\frac{[N-2(1+\beta R)]}{2\beta}\int_ {\partial B_R} \left(\frac{\partial w}{\partial \nu(x)}\right)^2  dS-\frac{\la}{K}\int_{B_R} \frac{1}{(1-w)^2}\left\langle \nabla w,x\right\rangle\,dx-\frac{\la N}{K} \int _{B_R} \frac{w}{1-w} dx,
\egee
or
\bge\label{ak3}\hspace{-1.3cm}
 \frac{\la (N-2)}{2K} \int _{B_R} \frac{1}{(1-w)^2}  dx \geq
 \frac{[N-2(1+\beta R)]}{2\beta}\int_ {\partial B_R}\left(\frac{\partial w}{\partial \nu(x)}\right)^2 dS-\frac{\la}{K}\int_{B_R} \frac{1}{(1-w)^2}\left\langle \nabla w,x\right\rangle\,dx,
 \qquad
\ege
since $0\leq w<1$   for any classical solution of \eqref{ssN}.

H\"older's inequality infers
\bgee
0\leq -\int_{\partial B_R} \frac{\partial w}{\partial  \nu(x)}\,  dS\leq
\left(\int_{\partial B_R} \left(\frac{\partial w}{\partial \nu(x)}\right)^2\, dS\right)^{1/2} \left(\int_{\partial B_R}  dS\right)^{1/2},
\egee
and so \eqref{29a} and divergence theorem imply
\bge\label{ak4}
\int_ {\partial B_R}\left(\frac{\partial w}{\partial \nu(x)}\right)^2\,   dS &\geq&
 \frac{1}{A\left(\partial B_R\right)}
 \left(\int_{\partial B_R} -\frac{\partial w}{\partial \nu(x)}\,  dS \right)^2\no\\
&=&\frac{1}{A\left(\partial B_R\right)}\left(\int_{B_R} -\Delta w\, dx \right)^2\no\\
&=&\frac{\la^2}{K^2A\left(\partial B_R\right)}\left(\int_{B_R} \frac{1}{(1-w)^2}\,  dx\right)^2,
\ege
where
\bgee\label{ak7}
A\left(\partial B_R\right):=\frac{2\pi^{(N+1)/2} R^{N-1}}{\Gamma(\frac{N+1}{2})},
\egee
and $\Gamma(\cdot)$ is the E\"uler's gamma function.

 On the other hand,
\bge\label{nik3}
\left\langle \nabla w,x\right\rangle=\frac{\partial w}{\partial x}= w_r \frac{\partial r}{\partial x}=w_r |x|, \quad\mbox{for}\quad r=|x|,
\ege
 where $\frac{\partial w}{\partial x}$ is the directional derivative in the $x$ direction and
where $w(r)$ satisfies
\bgee
&&-w_{rr}-\frac{N-1}{r}w_r=\frac{\la}{(1-w(r))^2 K},\quad 0<r<R,\label{nik1a}\\
&& w_r(0)=0,\quad w_r(R)+\beta w(R)=0. \label{nik2}
\egee
Let $\psi:=w_r,$ then $\theta$ satisfies
\bgee
&&-\psi_{rr}-\frac{N-1}{r}\psi_r+\chi(r)\psi=0,\quad 0<r<R,\label{1nik1}\\
&& \psi(0)=0,\quad \psi(R)=-\beta w(R)\leq 0, \label{1nik2}
\egee
where $\chi(r):=\left(\frac{N-1}{r^2}-\frac{2\la}{(1-w(r))^3 K}\right)$ is bounded since $w(r)$ is a classical solution. Thus  maximum principle, \cite{Ev},  infers that $\psi(r)\leq 0$ in $[0,R], $ hence via \eqref{nik3} we obtain
\bge\label{nik4}
\frac{\la}{K}\int_{B_R} \frac{1}{(1-w)^2}\left\langle \nabla w,x\right\rangle\,dx\leq 0.
\ege
Therefore \eqref{ak3} in conjunction with \eqref{ak4} and \eqref{nik4}  implies
\bgee
 \frac{\la (N-2)}2 \frac{\int _{B_R} \frac{1}{(1-w)^2}  dx}
 {\left(1+\alpha\int _{B_R} \frac{1}{(1-w)}  dx \right)^2}
\geq\frac{\la^2[N-2(1+\beta R)]}{2\beta A\left(\partial B_R\right)}\left(\frac{\int_{B_R} \frac{1}{(1-w)^2}\, dx}{\left(1+\alpha\int _{B_R} \frac{1}{(1-w)}  dx\right)^2}\right)^2,
\egee
 or
\bgee
 \frac{ (N-2)}{2}
\geq\frac{\la [N-2(1+\beta R)]}{2\beta A\left(\partial B_R\right)}
\frac{\int_{B_R} \frac{1}{(1-w)^2}\, dx}{\left(1+\alpha\int _{B_R} \frac{1}{(1-w)}  dx \right)^2},
\egee
since $N>2(1+\beta R).$  Then  H\"older's inequality suggests that
\[\left(\int _{B_R} \frac{1}{(1-w)}  dx\right)^2\leq \omega_N^R \int _{B_R} \frac{1}{(1-w)^2}  dx,\]
and thus
\bge\label{ak5}
 \frac{ (N-2)}{2}
&\geq&\frac{\la [N-2(1+\beta R)]}{2\beta \omega^R_N  A\left(\partial B_R \right)}
\left[\frac{\left( \int_{B_R} \frac{1}{(1-w)}\,  dx\right)^2}
{\left(1+\alpha\int _{B_R} \frac{1}{(1-w)}\,  dx \right)^2}\right] \no\\
&=&\frac{\la [N-2(1+\beta R)]}{2\beta \omega^R_N  A\left(\partial B_R \right)}
\left[\frac{\int_{B_R} \frac{1}{(1-w)}\, dx}{1+\alpha\int _{B_R} \frac{1}{(1-w)}  dx}\right]^2,
\ege
where
\bgee\label{ak6}
\omega^R_N=\vert B_R \vert :=  \frac{\pi ^{N/2}R^N}{\Gamma (\frac{N}{2}+1)}.
\egee
Note that for a classical solution $w$ of \eqref{ssN} holds
\bgee
\int_{B_R} \frac{1}{(1-w)}\,  dx\geq \omega_N^R,
\egee
so  using that $g(y)=\frac{y}{\alpha y+1}$ is increasing in $(0,+\infty),$
   and  thus $g(y)\geq \omega_N^R/(1+\alpha \omega_N^R)$
 for any $y\geq \omega_N^R>0,$ then
 inequality \eqref{ak5} yields
\bgee
N-2\geq \frac{\la[N-2(1+\beta R)]}{\beta A\left(\partial B_R \right)}
\frac{\omega_N^R}{(1+\alpha\omega_N^R)^2}.
\egee
The latter inequality finally  gives the desired  estimate
\bgee
\la\leq \la_*:=\frac{\beta A\left(\partial B_R \right)(N-2)}{[N-2(1+\beta R)]}
\frac{(1+\alpha\omega_N^R)^2}{\omega_N^R},
\egee
and thus
\bge\label{aik1}
\la^*\geq \frac{\beta A\left(\partial B_R \right)(N-2)}{[N-2(1+\beta R)]}
\frac{(1+\alpha\omega_N^R)^2}{\omega_N^R},
\ege
 by the definition of $\la^*.$
\end{proof}
\begin{rem}
Estimate \eqref{ak9} in the case of  the $N-$dimensional unit sphere $B_1=\{x\in \R^N: |x|<1\},$ takes the form
\bgee
\la\leq \la_*:=
\frac{\beta A\left(\partial B_1 \right) (N-2) (1+\alpha\omega_N)^2}{ [N-2(1+\beta)]\omega_N},
\egee
provided that $N>2(1+\beta)$ where
$$A\left(\partial B_1 \right)=\frac{2\pi^{(N+1)/2}}{\Gamma(\frac{N+1}{2})}$$
and
\bgee
\omega_N=\vert B_1 \vert =\frac{\pi ^{N/2}}{\Gamma (\frac{N}{2}+1)}.
\egee
\end{rem}
\begin{rem}
Let $\Omega$ be a bounded domain with the same volume as the $N-$dimensional ball $B_R,$  then we can get a lower estimate of $\la^*(\Om)$ by virtue of \eqref{aik1}. Indeed,  one can  adapt the proof of the well known isoperimetric inequality \cite[Theorem 4.10]{b80} holding for regular inequalities to the case of the singular MEMS nonlinearity $f(u)=\frac{1}{(1-u)^2},$ cf. \cite[Proposition 2.2.1]{EGG10}. Therefore,
\bgee
\la^*(\Om)\geq\la^*(B_R),
\egee
hence by virtue of Theorem \ref{thPoh}  we finally derive
\bgee
\la^*(\Om)\geq \la_*(B_R):=\frac{\beta A\left(\partial B_R \right)(N-2)}{[N-2(1+\beta R)]}
\frac{(1+\alpha\omega_N^R)^2}{\omega_N^R}.
\egee
\end{rem}
%________________________________________
Next we present an upper estimate of the pull-in  voltage $\la^*$ for a general bounded domain $\Om.$ In particular  it holds.
\begin{prop}
For a general domain $\Om$ the following upper estimate of the pull-in voltage $\la^*$ holds
\bge\label{ub1n}
\la^*\leq \frac{2\la_1\left(1+\alpha^2|\Om|^2\right)}{m_1 |\Om|}<\infty,
\ege
where $(\la_1,\phi_1)$ is the principal eigenpair of the Laplacian associated with Robin boundary conditions, given by \eqref{egv},  and  $m_1:=\min_{\bar{\Omega}}\phi_1(x)>0.$
\end{prop}
\begin{proof}
Testing equation \eqref{ssN1} by $\phi_1$ over the domain $\Omega$ we obtain
\bge\label{ub1}
\la_1\int_{\Omega}w\phi_1 dx =
 \frac{\la\int _{\Omega} \frac{\phi_1}{(1-w)^2} dx}
 {\left(1+\alpha\int _{\Omega} \frac{1}{(1-w)} dx\right)^2} >
 \frac{\la m_1\int _{\Omega} \frac{1}{(1-w)^2} dx}
 {\left(1+\alpha\int _{\Omega} \frac{1}{(1-w)} dx\right)^2}.
\ege
Next  H\"older's and Young's inequality suggest that
\bge\label{ub2}
\left(1+\alpha\int _{\Omega} \frac{1}{(1-w)} dx\right)^2\leq
2+2\alpha^2 |\Omega| \int _{\Omega} \frac{1}{(1-w)^2} dx.
\ege
Then  inequalities \eqref{ub1} and \eqref{ub2}, and for a classical solution $0\leq w<1,$  imply
\bge
\la_1=\la_1 \int_{\Omega}\phi_1 dx \geq \int_{\Omega}w\phi_1 dx \geq
\frac{\frac{\la m_1}{\alpha^2}\alpha^2\int _{\Omega} \frac{dx}{(1-w)^2} }
 {2+2\alpha^2 |\Omega| \int _{\Omega} \frac{dx}{(1-w)^2} } =\frac{\la m_1}{\alpha^2}\Psi(I_{\alpha}(w))
\ege
where  $\Psi(s)=\frac{s}{2+2\alpha^2 |\Omega| s},$ taking also into account \eqref{ik1}. Note that  $\Psi(s)$ is increasing and thus
\bgee
\Psi(I_{\alpha}(w))>\Psi(\alpha^2 |\Om|)=\frac{\alpha^2 |\Om|}{2+2 \alpha^2 |\Om|^2}
\egee
for $I_{\alpha}(w):=\alpha^2 \int_{\Om} \frac{dx}{(1-w)^2}.$

The latter by virtue of \eqref{ub2} implies
\bgee
\la_1\geq \frac{\la m_1 |\Om|}{2(1+\alpha^2 |\Om|^2)}
\egee
and thus via the definition of $\la^*$ we derive the desired upper bound \eqref{ub1n}.
\end{proof}

\section{The Time Dependent Problem:local,  global existence and quenching }\label{rnc5}

%________________________________________________________________________
%________________________________________________________________________

\subsection{Local existence and uniqueness}
In this subsection we study the local existence and uniqueness of solutions of problem \eqref{o.oneN1}.
Initially we define the notion of lower-upper solution pairs which will be applied for comparison purposes, cf. \cite{AKH11, GN12, L1}.

\begin{defn} \label{25}
A pair of functions $0\leq  v(x,t) , z(x,t) <1$ with $v,z \in C^{2,1}(Q_T)\cap C(\overline{Q}_T)$ is called a lower- upper solution pair of problem \eqref{o.oneN1}, if $v(x,t) \leq z(x,t)$ for $(x,t) \in Q_T, \quad 0< v(x,0)\leq u_0(x) \leq z(x,0)<1$ in
$ \overline{\Omega}, \frac{\pl v}{\pl \nu}(x,t)+\beta v(x,t) \leq 0 \leq \frac{\pl z}{\pl \nu}(x,t)+\beta z(x,t)$ for $(x,t) \in \partial \Omega \times [0,T],$ and
\begin{equation}
v_t \leq \Delta v + \frac{\lambda}{(1-v)^2 (1+\alpha\int_{\Omega} \frac{dx}{1-z})^2}, \quad \mbox{in} \quad Q_T, \nonumber
\end{equation}
\begin{equation}
z_t \geq \Delta z + \frac{\lambda}{(1-z)^2 (1+\alpha\int_{\Omega} \frac{dx}{1-v})^2}, \quad \mbox{in} \quad Q_T. \nonumber
\end{equation}
\end{defn}
Then  local-in-time  existence and uniqueness of problem \eqref{o.oneN1} is then established  by the following.
\begin{prop} \label{26}
Let $(v,z)$ is a lower- upper solution pair to  problem \eqref{o.oneN1} in $Q_T$ for some $T>0$. There is a unique solution $u$ to  problem \eqref{o.oneN1} such that $0<v \leq u \leq z<1$ in $Q_T.$
\end{prop}
 \begin{proof}
  We define $\overline{u}_0=z , \underline{u}_0=v$ and we construct a sequence of lower-upper solutions  of
   problem \eqref{o.oneN1} in the following way:
 \bgee
 &&{\underline{u}_{n}}_t=\Delta \underline{u}_n+
 \frac{\lambda}{(1-\underline{u}_{n-1})^2 (1+ \alpha\int_{\Omega} \frac{dx}{1-\overline{u}_{n-1}})^2},
  \quad \mbox{in} \quad  Q_{T_n}:=\Om\times (0,T_n),
  \\
 &&{\overline{u}_{n}}_t=\Delta \overline{u}_n+
 \frac{\lambda}{(1-\overline{u}_{n-1})^2 (1+\alpha \int_{\Omega} \frac{dx}{1-\underline{u}_{n-1}})^2} \quad \mbox{in} \quad  Q_{T_n},\\
 &&\frac{\partial\underline{u}_n}{\partial \nu } + \beta \underline{u}_n  = 0,
  \quad \mbox{on}  \quad \Gamma_{T_n}:=\partial \Omega \times (0,T_n),
  \\
 &&\frac{\partial\overline{u}_n}{\partial \nu } + \beta \overline{u}_n = 0,
  \quad \mbox{ on} \quad \quad   \Gamma_{T_n},\\
&&\underline{u}_n ({x,0})=\overline{u}_n(x,0)=u_0(x), \quad \mbox{for} \quad x \in \overline{\Omega},
  \nonumber
\egee
  for $n=1,2,\ldots$   where $T_n$ is the maximum existence time for the pair $(\underline{u}_n, \overline{u}_n)$.  Note that by the previous definition we have that the pair  $(\underline{u}_n, \overline{u}_n)$ exist  as long as the pair $(\underline{u}_{n-1}, \overline{u}_{n-1})$ does so, and thus $T_{n-1}\leq T_n.$ for $n=2,3,\dots $.

  The above problems are local and linear and so we can get  local-in-time solutions for them  via the classical parabolic theory. Furthermore using Definition \ref{25} and standard comparison arguments for parabolic problems (see \cite{KThesis2000}), we deduce that the sequences
    $\lbrace\underline{u}_n\rbrace_{n=1} ^{\infty},\lbrace\overline{u}_n\rbrace_{n=1} ^{\infty} \in C^{2,1}(Q_T)\cap C(\overline{Q}_T)$,
     for  $T=:\min\{T_n |n\in \N\}=T_1$,    are positive and satisfy the ordering
$$v\leq \underline{u}_{n-1}\leq \underline{u}_{n}\leq ... \leq \overline{u}_n \leq \overline{u}_{n-1}\leq z. $$
 Let  $u_1:=\lim_{n\to\infty}\underline{u}_n$  and  $u_2:=\lim_{n\to\infty}\overline{u}_n$ then
$u_1,u_2$ satisfy
 \bgee
 &&{ u_{1}}_t=\Delta u_1+ \frac{\lambda}{(1-u_1)^2 (1+\alpha\int_{\Omega} \frac{dx}{1-u_2})^2},
\quad \mbox{in} \quad Q_T,\\
&&{u_{2}}_t=\Delta u_2+ \frac{\lambda}{(1-u_2)^2 (1+\alpha\int_{\Omega} \frac{dx}{1-u_1})^2}, \quad \mbox{in} \quad Q_T,\\
 &&\frac{\partial u_1}{\partial \nu } + \beta u_1 = 0, \quad \mbox{on} \quad \Gamma_T,\\
&&\frac{\partial u_2}{\partial \nu } + \beta u_2 = 0 \quad  \mbox{on} \quad \Gamma_T,\\
&& u_1(x,0)=u_2(x,0)=u_0(x), \quad \mbox{for} \quad x \in \bar{\Omega}.
\egee
 Set   $\psi(x,t)=u_1(x,t)-u_2(x,t)$ then
\begin{eqnarray*}
&&\psi_t=\Delta \psi +A(x,t) \psi + B(x,t) \int_{\Omega} \int _0 ^1 \frac{d\theta}{[1-\theta u_1 - (1- \theta)u_2]^2} \psi dx, \quad \mbox{in} \quad Q_T,
\\
&&\frac{\partial \psi}{\partial \nu}+\beta \psi=0,
 \quad \mbox{on} \quad \Gamma_T,
\\
&&\psi(x)=0,\quad x\in \bar{\Om},
\end{eqnarray*}
where
\bge A(x,t):= 2 \lambda \frac{\int_0 ^1 \frac{d\theta}{[1-\theta u_1 - (1- \theta)u_2]^3}}{(1+ \alpha\int_\Omega \frac{dx}{1-u_2})^2} >0,
\ege
and
\bge
B(x,t):= \frac{\lambda}{(1-u_2)^2} \frac{2+ \int_{\Omega} \frac{dx}{1-u_1}+
 \int_{\Omega} \frac{dx}{1-u_2}}
 {(1+\alpha\int_{\Omega} \frac{dx}{1-u_1})^2 (1+\alpha\int_{\Omega} \frac{dx}{1-u_2})^2}>0,
\ege
cf.  \cite{GN12}. Applying now \cite[Proposition 52.24]{QS} we obtain that $\psi(x,t)=0$  and  therefore  $u_1=u_2:=u$ in $\overline{Q}_T.$
%_____________________
%__________________________

Now assume  there is a second solution $U$ which satisfies $v\leq U \leq z$. Subsequently by the preceding iteration scheme we have that $\underline{u}_n \leq U \leq \overline{u}_n$ for every $n=1,2,....$ and by taking the limit as $n\to \infty$ we finally deduce that $U=u$ by the uniqueness of the limit.
 \end{proof}
\begin{rem}
By the above result we obtain that the solution  of \eqref{o.oneN1} continues to exist as long as it remains less than or equal to $B$ for some $B<1$.
In this  case we say that $u$  ceases to exist  only by quenching, if there is a sequence $(x_n,t_n) \rightarrow (x^*,t^*)$ as
 $ n \rightarrow \infty $ with $t^* \leq \infty$ such that $u(x_n,t_n) \rightarrow 1$ as $n \rightarrow \infty,$  cf. Definition \ref{rnc1}.
\end{rem}

Next we provide a local-in-time existence result for \eqref{o.oneN1} using comparison arguments. To this end we first note that the following (local) problem
\bse\label{LPz0}
\bge
&&z_t=\Delta z + \frac{\lambda}{(1-z)^2(1+\alpha  \vert \Omega \vert) ^2} \quad\mbox{on} \quad Q_T,\label{LPz01}\\
&&\frac{\pl z}{\pl \nu}+\beta z=0 \quad \mbox{on} \quad \Gamma_T,\label{LPz02}\\
&&0\leq z(x,0)=z_0(x)<1 \quad \mbox{for} \quad x  \in \overline{\Omega}.\label{LPz03}
\ege
\ese
 has  a  unique  solution, see \cite{Guo91}. Therefore the following holds:
\begin{prop}\label{nik1}
If $z_0(x)\geq u_0(x)$ for each $x \in \Omega$, then the problem \eqref{o.oneN1} has a unique solution $u$ on
 $\Omega \times [0,T),$ where $[0,T)$ is the maximal existence time interval for the solution $z(x,t)$  of the problem \eqref{LPz0}, and $0\leq u(x,t) \leq z(x,t)<1$ on $\Omega \times [0,T)$.
\end{prop}
\begin{proof}
Let $v(x,t)=0$, then it is readily seen that
\begin{eqnarray}
&& z_t=  \Delta z+ \frac{\lambda}{(1-z)^2 (1+ \alpha \vert \Omega \vert)^2} \geq
 \Delta z+ \frac{\lambda}{(1-z)^2 (1+\alpha  \int_{\Omega} \frac{dx}{1-v})^2}
 \quad \mbox{in} \quad Q_T,
\nonumber\\
&& \frac{\pl z}{\pl \nu}+\beta z=0, \quad \mbox{on} \quad \Gamma_T,
\nonumber\\
&& z(x,0)=z_0(x) \quad\mbox{ for }\quad x \in \bar{\Omega},\nonumber
\end{eqnarray}
while $v$ satisfies
\begin{eqnarray}
&& v_t-\Delta v=0 \leq \frac{\lambda}{(1-v)^2(1+\alpha \int_{\Omega} \frac{dx}{1-z})^2} \quad \mbox{on} \quad  Q_T,
\nonumber\\
&& \frac{\pl v}{\pl \nu}+\beta v=0, \quad \mbox{on} \quad \Gamma_T,
\nonumber\\
&& v(x,0)=0 \quad \mbox{for} \quad x\in \bar{\Omega}.\nonumber
\end{eqnarray}
Therefore  according to Definition \ref{25} $(v,z)$ is a lower-upper solution pair for the problem \eqref{o.oneN1} and thus the result  is an immediate consequence of  Proposition \ref{26}.
\end{proof}
 %_______________________________________

\subsection{Global existence and quenching for general domain}\label{rnc4}

In the current subsection we investigate the global existence and quenching of the solutions of problem \eqref{o.oneN1}.

We first show the following global existence result.
\begin{thm}
Assume that $\lambda \in (0,(1+\alpha  \vert \Omega \vert)^2 \mu ^*)$, recalling that  $\mu^*$ defined by \eqref{3A}. Then problem \eqref{o.oneN1} with initial condition $u_0(x)\leq w_{\lambda}(x)$ has a global-in-time solution converging as $ t\rightarrow \infty$ to the minimal steady state solution $w_{\lambda}(x)$ of \eqref{ssN}, corresponding to $\la.$
\end{thm}
\begin{proof} By Proposition \ref{nik1} we have that $(0,z)$ is  a lower-upper pair for problem \eqref{o.oneN1}, where $z$ is the unique solution of local problem  \eqref{LPz0} with initial data $0\leq z_0=u_0\leq w_{\la}<1.$ Then Proposition \ref{26} infers that $0\leq u\leq z.$
 Moreover, due to Theorem \ref{psks1}  problem \eqref{ssN} has a minimal solution $w_{\la}$  for any $\lambda \in (0,(1+ \alpha \vert \Omega \vert)^2 \mu ^*)$ and thus \eqref{29} has also a minimal solution $w_{\mu}$ for any
\bge\label{psks2}
0<\mu=\frac{\la}{K(w_{\mu})}<\mu^*.
\ege
On the other hand, we can find  $\mu_1\in(0,\mu^*)$ such that
\bge\label{psks3}
\mu_1=\frac{\la}{(1+\alpha  \vert \Omega \vert)^2}.
\ege
Using now \eqref{ikk1a},  then by virtue of  \eqref{psks2} and \eqref{psks3} we get  that $\mu<\mu_1$ and so Lemma \ref{30} finally implies that $w_{\mu}\leq w_{\mu_1}.$ Then via comparison,  cf. Proposiition \ref{26},  $0\leq  z \leq w_{\mu_1}$ since $z_0=u_0\leq w_{\la}=w_{\mu}\leq w_{\mu_1}$ and thus we finally deduce that
\bgee
0<u(x,t)\leq z(x,t)\leq w_{\mu_1}(x)<\infty, \quad\mbox{for any}\quad x\in \Om,\quad\mbox{and}\quad t>0,
\egee
and  therefore a global-in-time solution for problem \eqref{o.oneN1} exists.   Using the dissipative property \eqref{Energy} of energy $E(t),$ see also \cite{NK04}, we can prove convergence of  $u(x,t)$ towards the steady-state solution $w_{\la}(x),$ since $u_0(x)\leq w_{\la}(x).$
\end{proof}
Next we define the notion of finite time quenching, which is closely related   to the mechanical 
phenomenon of touching down.
\begin{defn}\label{rnc1}
The solution $u(x, t)$ of problem \eqref{o.oneN1}  quenches at some point 
$x^* \in \Omega$ in finite time $0< T_q < \infty$ if there exist sequences $\{x_n\}_{n=1} ^{\infty} \in \Omega $
and $\{t_n\}_{n=1} ^{\infty} \in (0, \infty) $ with
$x_n \rightarrow x^*$ and $t_n \rightarrow T_q$ as $n \rightarrow \infty $ such that 
$u(x_n, t_n) \rightarrow 1-$ as $n \rightarrow \infty$.
 When
$T_q =\infty$ we say that $u(x, t)$ quenches in infinite time at $x^*$.
Moreover
$$\textit{Q}= \lbrace x^* \in  \bar{\Omega}  \vert \exists  \ (x_k, t_k )_{k \in \mathbb{N}} \subset \Omega \times (0,T_q) : \ x_k \rightarrow x^*, t_k \rightarrow T_q \mbox{ and} \ u(x_k, t_k ) \rightarrow 1 \mbox{ as} \ k \rightarrow \infty \rbrace,$$
is called the quenching set of $u.$
\end{defn}
Now
we determine the energy of the problem \eqref{o.oneN1}.
Accordingly we multiply (\ref{o.one1N}) by $u_t$ and integrating over $\Omega$  to  derive
\begin{eqnarray*}
\int _{\Omega} u_t^2 \,dx 
&=& - \int _{\Omega} \nabla u_t \nabla u dx - \beta \int_{\partial \Omega} u_t u dS +
\frac{ \lambda }{\alpha}\frac{d}{dt} \left(- \frac{1}{1+\alpha \int_{\Omega} \frac{1}{1-u} dx}\right)
 \\
 &=&- \frac{1}{2}\frac{d}{dt} \int _{\Omega} \vert \nabla u\vert ^2 dx -
  \frac{\beta}{2} \frac{d}{dt} \int_{\partial \Omega} u^2 dS +
 \frac{ \lambda }{\alpha}\frac{d}{dt} \left(- \frac{1}{1+ \alpha \int_{\Omega} \frac{1}{1-u} dx}\right),
\end{eqnarray*}
taking also into account boundary condition \eqref{o.one2N}.

Therefore we obtain
\be \label{Energy}
 \frac{d}{dt} \left[ \frac{1}{2}\int _{\Omega} \vert \nabla u\vert ^2 dx
+ \frac{\beta}{2}  \int_{\partial \Omega} u^2 dS
+ \frac{\lambda/\alpha  }{1+ \alpha \int_{\Omega} \frac{1}{1-u} dx} \right]=-\int _{\Omega} u_t^2 dx,
\ee
which  implies that the energy functional
\begin{equation}
E(t):= \frac{1}{2}\int _{\Omega} \vert \nabla u\vert ^2 dx + \frac{\beta}{2}  \int_{\partial \Omega} u^2 dS+
   \frac{\lambda/\alpha }{1+\alpha  \int_{\Omega} \frac{1}{1-u} dx} \leq E(0) := E_0  < \infty, \label{11}
\end{equation}
decreases in time along any solution of \eqref{o.oneN1}.

Below,  we present a quenching result for a general domain $\Omega$ following an approach introduced in \cite{GN12}, see also \cite{G14}.
 \begin{thm}\label{qgd}
 For any fixed $\la>0,$ there exist initial data such that the solution of problem \eqref{o.oneN1} quenches in finite time provided the associated initial energy
 \bgee
 E_0:= \frac{1}{2}\int _{\Omega} \vert \nabla u_0\vert ^2 dx + \frac{\beta}{2}  \int_{\partial \Omega} u_0^2 dS+
   \frac{\lambda/\alpha }{1+\alpha  \int_{\Omega} \frac{1}{1-u_0} dx}
 \egee  is chosen  sufficiently small, i.e.
\bge\label{q-cond}
E_0< \frac{\la\,q_{\alpha}\left(|\Omega|\right)}{2\alpha },
\ege
where
\bge\label{ikk4}
q_{\alpha}\left(|\Omega|\right):=\left\{\begin{array}{ll}
{\displaystyle 1}, & \textrm{$|\Omega|\leq \frac{1}{3 \alpha},$}\\
{\displaystyle\frac{1}{3 \alpha|\Omega|}}, & \textrm{$|\Omega|\geq \frac{1}{3 \alpha}.$}
\end{array} \right.
\ege
 \end{thm}
\begin{proof}
The proof follows closely that of \cite[Theorem 1.2.17]{KS18}, which deals with Dirichlet boundary conditions,  however for the sake of completeness a sketch of the proof is provided here.

Assume that  problem \eqref{o.oneN1} has a global-in-time (classical) solution $u,$ i.e. $0<u(x,t)<1$ for any $(x,t)\in \Om\times (0,\infty)$ and so
\bge\label{ikk5}
Z(t):=\int_{\Omega} u^2(x,t) dx<|\Om|,\quad\mbox{for any}\quad t>0.
\ege
Multiplying equation \eqref{o.one1N} by $u$ and integrating  by parts over $\Omega$, we deduce
\bge\label{q1}
\frac{1}{2} \frac{d Z}{dt}&&=-\int_{\Omega}|\nabla u|^2 dx-\beta \int_{\pl \Omega} u^2\,ds+\la\frac{\int_{\Omega}\frac{u}{(1-u)^2}dx}{\Big(1+\alpha \int_{\Omega}(1-u)^{-1}dx\Big)^2}.
\ege
Using (\ref{Energy})  then  (\ref{q1}) reads
\bge
\frac{1}{2} \frac{d Z}{dt}&&=-2 E(t)+
\frac{2 \la}{\alpha }\frac1{\left(1+\alpha \int_{\Omega}(1-u)^{-1}dx\right)}
+\la\frac{\int_{\Omega}\frac{u}{(1-u)^2}dx}{\Big(1+\alpha \int_{\Omega}(1-u)^{-1}dx\Big)^2}\nonumber\\
&&\geq -2E_0+\frac{\la}{\alpha}
\frac{2+\alpha\int_{\Omega}\frac{2-u}{(1-u)^2}dx}{\Big(1+\alpha \int_{\Omega}(1-u)^{-1}dx\Big)^2}.\label{q11}
\ege
Besides,  H\"older's and Young's inequalities imply
\bgee\label{q2}
\Big(1+\alpha\int_{\Omega} \frac{dx}{1-u}\Big)^2\leq 2+3 \alpha^2|\Omega|\int_{\Omega} \frac{dx}{(1-u)^2},
\egee
and thus by virtue of (\ref{q11}) we obtain
\bgee
\frac{1}{2} \frac{d Z}{dt}\geq -2E_0+\frac{\la}{\alpha} q_{\alpha}\left(|\Omega|\right),
\egee
or
\bgee
Z(t)\geq 2\left[q_{\alpha}\left(|\Omega|\right)\,\frac{\la}{\alpha}-2 E_0\right]t+Z(0),
\egee
for $q_{\alpha}\left(|\Omega|\right)$ given by \eqref{ikk4}. The latter implies that $Z(t)\to \infty$ as $t\to \infty$ provided that $E_0$ satisfies \eqref{q-cond}, which contradicts to \eqref{ikk5}. Therefore the theorem follows.
\end{proof}
%_________________________________________________
%__________________________________
\begin{rem}
If we fix the initial data $u_0,$ and thus initial energy $E(0),$ then  Theorem \ref{qgd} provides a quenching result for big values of the nonlocal parameter $\lambda.$ In particular, \eqref{q-cond} provides a threshold for parameter $\lambda$ above which finite-time quenching occurs. Namely, if
$$
\la >\tilde{\lambda}:= \frac{2\alpha
\left(\frac{1}{2}\int _{\Omega} \vert \nabla u_0\vert ^2 dx +
 \frac{\beta}{2}  \int_{\partial \Omega} u_0^2 dS \right)}
{q_{\alpha}\left(|\Omega|\right) -  \frac{2\alpha}{1+ \alpha \int_{\Omega} \frac{1}{1- {u_0}} dx}},
$$
then $||u(\cdot,t)||_{\infty}\to 1^{-}$ as $t\to T_q<\infty$ provided that $\mathcal{A}_{\alpha}(|\Om|):=q_{\alpha}\left(|\Omega|\right) -  \frac{2\alpha}{1+ \alpha \int_{\Omega} \frac{1}{1- {u_0}} dx}$ is positive.
Note that
\bgee
\mathcal{A}_{\alpha}(|\Om|)\geq q_{\alpha}\left(|\Omega|\right) -  \frac{2\alpha}{1+ \alpha |\Omega|}=\left\{\begin{array}{ll}
{\displaystyle \frac{1+\alpha(|\Om|-2)}{1+\alpha |\Om|}}, & \textrm{$|\Omega|\leq \frac{1}{3 \alpha},$}\\
{\displaystyle\frac{1-\alpha(6\alpha-1)|\Om|}{3 \alpha|\Omega|(1+\alpha |\Om|)}}, & \textrm{$|\Omega|\geq \frac{1}{3 \alpha},$}
\end{array} \right.
\egee
and so $\mathcal{A}_{\alpha}(|\Om|)>0$ by either choosing  $\alpha <\frac{2}{3}$
 and $\frac{2\alpha-1}{\alpha}<|\Om|\leq \frac{1}{3\alpha}$
 for the first branch of the inequality, and $\frac16<\alpha<\frac23$ with
  $\frac{1}{3 \alpha}\leq |\Om|<\frac{1}{\alpha(6 \alpha-1)}$ or just $\alpha<\frac16$ for the second branch.

Remarkably, an optimal value of $\tilde{\la}$ for the unit sphere $B_1(0)$ is given in Theorem \ref{th_lambda}, where it is actually shown that $\tilde{\la}=\la^*.$
\end{rem}

A first step towards the derivation of sharper quenching results  is the following lemma. Henceforth, we use $C_i, i=1,\dots,$ to denote various positive constants.
\begin{lem}\label{lem1}
Let u be a global-in-time solution of the problem \eqref{o.oneN1}. Then there is a sequence
$\lbrace t_j \rbrace _{j=1} ^ {\infty}  \  \uparrow \ \infty \ \mbox{as}  \ j \rightarrow \infty $ such that
 \begin{equation}
 \label{12}
  \lambda \int _ {\Omega} u_j (1-u_j)^{-2} dx \leq C_1 \left( H(u_j)\right)^2,
 \end{equation}
 for a positive constant $C_1,$   where $u_j=u(\cdot , t_j)$ and
 \begin{equation}
 \label{13}
 H(u_j):= 1+ \alpha \int_{\Omega} \frac{1}{1-u_j}\, dx >1.
 \end{equation}
\end{lem}
\begin{proof}
The  proof follows closely the steps of the proof of  \cite[Lemma 2.1]{KLN16} for the case of Dirichlet boundary conditions and so it is omitted.
\end{proof}
\subsection{Finite time quenching for the radial  symmetric  case}\label{rnc3}
A wide used situation  is a circular MEMS configuration, see Figure 2(b),   cf. \cite{PC03}.  Especially,  in that case  the role of the elastic membrane is played by a soap film  and such configuration was first suggested by the prolific British scientist, G.I. Taylor, who actually investigated the coalescence of liquid drops held at differing electric
potentials, \cite{T68}. Later, R.C. Ackerberg initiated the mathematically study of Taylor's model in \cite{Ac69}.

Under a circular configuration,  i.e. when $\Omega=B_1(0),$ then solution of  problem \eqref{o.oneN1}  is radial symmetric, cf. \cite{Gi-Ni-Ni}, and then we end up with the following
 \bse\label{39}
 \bge\label{39a}
 &&u_t- u_{rr}-(N-1) r^{-1} u_r =F(r,t),\quad (r,t) \in (0,1) \times (0,T),\quad N\geq 1,\\
 &&u_r(0,t)=0, \quad u_r(1,t)+\beta u(1,t)=0 ,\quad t \in (0,T),\label{39b}\\
 &&0\leq u(r,0)=u_0(r)<1 ,\quad 0<r<1,\label{39c}
 \ege
 \ese
 where \begin{equation} \label{40}
 F(r,t)=\lambda  (1-u(r,t))^{-2}k(t),
\end{equation}
and
\bgee
k(t)=\left[1+ \alpha N \omega_N \int_0 ^1 r ^{N-1} (1-u(r,t))^{-1}dr\right]^{-2},
\egee
recalling that $\omega_N$ stands for the volume
of the $N$-dimensional unit sphere $B_1(0)$ in $ \mathbb{R} ^N.$
Note that condition  $u_r(0,t)=0$ is imposed to guarantee the regularity of the solution $u$.
 We also, for simplicity, consider that $u_0'(r)\leq 0$ for $0\leq r\leq 1, $ and thus  via maximum principle $u_r(r,t)\leq 0$ for $(r,t) \in [0, 1]\times [0,T).$

For convenience we define   $0< v := 1-u\leq 1$ and  so $v$ satisfies
\bse\label{39rad}
 \bge
 &&v_t - v_{rr}- (N-1)r^{-1}v_r= - fv^{-2}, \quad (r,t)\in (0,1) \times (0,T), \label{41}\\
 &&v_r(0,t)=0,\quad v_r (1,t) + \beta v(1,t)=\beta, \quad t \in (0,T)\label{42}\\
 &&0<v(r,0)=v_0(r)\leq 1, \quad 0<r<1 ,\label{43}
 \ege
 \ese
where
 $$f=f(t):= \frac{\lambda}{ \left[1+\alpha N \omega_N \int _0 ^1 r^{N-1}v^{-1} dr \right]^{2}}$$
  and
\bge\label{ikk1}
v_r(r,t)> 0\quad\mbox{for}\quad (r,t) \in (0, 1]\times [0,T).
\ege
For the rest of the our analysis we need a lower estimate for $v$, which infers a uniform in time  upper estimate of the nonocal term,  and is  shown in  the following.
\begin{lem}\label{lem2}
Consider radial symmetric $v_0(r)$ with $v'_0(r)>0$ and  assume also that  $N>\beta+1$. Then  for any $k>2/3$ there is a constant $C=C(k)$ such that
\begin{equation}
v(r,t)\geq C(k) r^k \quad \mbox{for} \quad (r,t) \in (0,1) \times (0,T ). \label{44}
\end{equation}
Moreover, there exists a constant $C_2$ which is independent of time $t$ and uniform in $\lambda$ such that
\begin{equation}
H(u)=H(1-v)\leq C_2 \quad \mbox{for any} \quad 0<t<T. \label{45}
\end{equation}
\end{lem}
\begin{proof}
Considering $1<b<2$, there exist some $t_1>0$ and $\epsilon_1>0$ such that
\begin{equation}\label{cn1}
 v_r>\epsilon_1 r v^{-b} \quad \mbox{at} \quad t=t_1 \quad for \quad 0<r\leq 1,
\end{equation}
 since $v>0$ with a bounded spatial derivative  is a  classical solution of \eqref{41}-\eqref{43}.

Next differentiating equation (\ref{41}) with respect to  $r$ gives
$$(v_r)_t-(v_{rr})_r-(N-1)(-r^{-2}v_r+r^{-1}(v_r)_r)=2fv^{-3}v_r,$$
which after multiplying  with $r^{N-1}$ reads
\begin{equation}
z_t-z_{rr}+(N-1)r^{-1}z_r=2fv^{-3}z, \label{47}
\end{equation}
for
$
z:=r^{N-1} v_r \label{46}.
$

As a next step we define the functional
\begin{equation} \label{48}
J=z-\epsilon r^Nv^{-b} \quad \mbox{for } \quad 0<\epsilon< \epsilon_1,
\end{equation}
and  note that
\begin{equation}
J>0 \quad for \quad 0<r\leq 1 \quad at \quad t=t_1, \label{54}
\end{equation}
thanks to \eqref{cn1}.

Moreover
\bgee
J_t&=&z_t+b \epsilon r^N v^{-b-1}v_t, \label{49} \\
J_r&=&z_r-\epsilon N r^{N-1}v^{-b}+\epsilon b r^N v^{-b-1}v_r, \label{50}
\egee
and
\bgee
J_{rr}&=& z_{rr}+b \epsilon r^N v^{-b-1} v_{rr} + 2Nb \epsilon r^{N-1} v^{-b-1} v_r - b(b+1) \epsilon r^N v^{-b-2} v_r^2 - N(N-1) \epsilon r^{N-2} v^{-b}. \label{51}
\egee

Notably as long as $J>0$, then $v_r>\epsilon r v^{-b}$ and so
\bgee\label{55}
v>\left(\frac{b+1}{2} \epsilon\right)^{\frac{1}{b+1}} r^{\frac{2}{b+1}},
\egee
which retrieves \eqref{44} for $C=\left( \frac{b+1}{2} \epsilon \right) ^{\frac{1}{b+1}}$and $k=\frac{2}{b+1}.$

The latter inequality infers
\begin{eqnarray}
\int_0^1 r^{N-1} v^{-1}dr & <&
\int_0 ^1 r^{N-1} \left( \frac{2}{(b+1)\epsilon} \right) ^{\frac{1}{b+1}} \frac{1}{r^{\frac{2}{b+1}}}dr
\nonumber\\
 &\leq&\left( \frac{2}{(b+1)\epsilon} \right) ^{\frac{1}{b+1}} \left(\frac{b+1}{Nb+N-2}\right)=
  C_2 \epsilon ^{-1/(b+1)}, \label{intbound}
 %C_2,
\end{eqnarray}
and  thus estimate \eqref{45} is also retrieved.

We now introduce the function
\begin{equation}
G(\epsilon):= \frac{\epsilon^{\frac{2}{b+1}}} {\left( \epsilon ^{\frac{1}{b+1}}+
 \alpha N \omega_N C_2\right)^2 },
%\frac{1}{N b +N-2} \left(b+1 \right)^{\frac{b}{b+1}} 2^{\frac{1}{b+1}}
\end{equation}
where parameter  $\epsilon$  is small enough $0<\epsilon\ll 1$,
and  $\epsilon_2$  imposed to fulfill
\begin{equation}\label{53}
\epsilon_2 < \sup \left\lbrace \epsilon: \epsilon \leq
\min \left\lbrace \frac{1}{N}, \left(\frac{2-b}{2b}\right) \right\rbrace
\lambda G(\epsilon) \right\rbrace.
\end{equation}
Remarkably, such an $\epsilon_2$ satisfying \eqref{53} exists since $G(\epsilon)=O(\epsilon^{2/(b+1)}))\gg\epsilon$  for $\epsilon$ small with $0<\epsilon< \min\{\epsilon_1, \epsilon_2\}$, taking also into account that $b>1.$

By virtue of  \eqref{54} and \eqref{intbound}
 there holds
\bge\label{56}
f(t_1)=\frac{\lambda}{\left(1+\alpha N  \omega_N \int_0 ^1 r^{N-1} v^{-1} dr \right)^2}> \lambda G(\epsilon),
\ege
thus, in a neighborhood of $t=t_1$ we obtain that  $f(t)> \lambda G(\epsilon).$

Now we  claim that  $f(t)>\lambda G(\epsilon)$ for any $t\in (t_1,T).$  Let us assume to  the contrary that:
\bge\label{ikk12}
\mbox{there exists}\; t_2\in (t_1,T)\;\mbox{such that}\; f(t_2)=\lambda G(\epsilon)\;\mbox{
with}\; f(t)>\lambda G(\epsilon)\;\mbox{for}\; t_1\leq t <t_2.
\ege
By the definition of $J$ and $z$ we immediately get
\bgee
J=0  \quad on \quad r=0, \label{57}
\egee
whilst on  the boundary $r=1$,  due to \eqref{42}, we have
\bge\label{ikk9}
J&=& z(1,t)-\epsilon v^{-b}(1,t)
\no\\
&=& \beta\left( 1-v(1,t)\right)-\epsilon v^{-b}(1,t)=v_r(1,t)-\epsilon v^{-b}(1,t)>0,
\ege
provided that
 \bgee\label{ikk11}
0<\epsilon\leq \epsilon_3:=\inf_{t_1<t<t_2}\frac{v_r(1,t)}{v^{-b}(1,t)},
\egee
and taking also into account \eqref{ikk1}.

In addition
$$
J_r=z_r- \epsilon N r^{N-1} v^{-b}+ \epsilon b r^N v^{-b-1} v_r= (N-1)r^{N-2} v_r +r^{N-1}v_{rr}
+ \epsilon r^N v^{-b}(-Nr^{-1}+b v^{-1} v_r),
$$
and for   $r=1$ we obtain
\begin{eqnarray*}
J_r &=& (N-1)v_r(1,t)+v_{rr}(1,t)+ \epsilon v^{-b}(1,t)\left[-N+b v^{-1}(1,t)v_r(1,t)\right].
\end{eqnarray*}
Moreover at   $r=1$
\begin{eqnarray*}
J_r-b \epsilon J &=& (N-1)v_r(1,t)+v_{rr}(1,t)- \epsilon v^{-b}(1,t)\left[N-b\beta v^{-1}(1,t)+b \beta\right]
-b \epsilon \left[\beta-(\beta v(1,t)+ \epsilon v^{-b}(1,t))\right] \\
 &=& (N-1)v_r(1,t)+v_{rr}(1,t) \\
&& - \epsilon\left[ v^{-b}(1,t) N-b \beta v^{-b-1}(1,t) +b \beta v^{-b}(1,t)\right.
+\left.  b\beta -b \beta v(1,t)-b \epsilon v^{-b}(1,t)\right],
\end{eqnarray*}
and therefore,  after dropping all the positive terms,
\bgee \label{60}
J_r-b \epsilon J >(N-1)v_r(1,t)+v_{rr}(1,t)-\epsilon \left[v^{-b}(1,t)N+b \beta v^{-b}(1,t)+b \beta\right].
\egee
%________________________________
Next differentiating the second of the  boundary conditions  \eqref{39b} with respect to $r$ we get
\bgee
v_{rr}(1,t)=-\beta v_r(1,t),
\egee
and thus
\bgee \label{60a}
J_r-b \epsilon J >\left(N-1-\beta \right)v_r(1,t)-\epsilon \left[v^{-b}(1,t)N+b \beta v^{-b}(1,t)+b \beta\right].
\egee
Therefore for $J_r(1,t)$ and $J_r(1,t)- b\epsilon J(1,t)$ to be positive we need
$$\left(N-1-\beta \right)v_r(1,t)-\epsilon \left[v^{-b}(1,t)N+b \beta v^{-b}(1,t)+b \beta\right]>0,$$
or it is sufficient to choose $\ep\leq\min\{\ep_3,\ep_4\}$ for
\bgee
\ep_4:=\inf_{t_1<t<t_2} \frac{\left(N-1-\beta \right)v_r(1,t)}{(N + b \beta) v^{-b}(1,t)+ b \beta}>0, \label{61}
\egee
since $N>\beta+1$.

Therefore we  have
\begin{eqnarray*}
J_t-J_{rr}&+& (N-1)r^{-1}J_r \geq 2J(f v^{-3} - b \epsilon v^{-b-1}) + \epsilon f r^N v^{-b-3} (2-b) - 2 \epsilon^2 r^N v^{-2b-1}b,
\end{eqnarray*}
and hence
\bge\label{ikk13}
J_t- J_{rr} + (N-1) r^{-1} J_r > 2 J (fv^{-3}- b \epsilon v^{-b-1}), \label{58}
\ege
as far  as
$$\epsilon f r^N v^{-b-3} (2-b)- 2\epsilon^2 r^N v^{-2b-1} b>0,$$
or
$$\epsilon f\,(2-b)> 2 \epsilon^2 b,$$
which in turn gives
\bgee
\label{59} \epsilon <\ep_5:=\inf_{t_1<t<t_2} \frac{f(t) (2-b)}{2b}.
\egee
After all  by maximum principle we derive that $J>0,$  for $0<r \leq 1, t_1\leq t \leq t_2$ and for $\epsilon$ small enough satisfying $\ep<\min\{\ep_1,\ep_2,\ep_3,\ep_4,\ep_5\}.$
 In $0<r \leq 1, t_1\leq t \leq t_2$, and since $v>0$ then the coefficient of $J$ in equation \eqref{ikk13} is bounded, so
we can define a new variable $\tilde{J}= e ^{-D_1 t} J$ which  then satisfies the boundary condition (\ref{57}), the boundary inequality  (\ref{60}) and
\begin{equation}
\label{62} \tilde{J}_t - \tilde{J}_{rr} + (N-1) r^{-1} \tilde{J}_r > -D_2 \tilde{J},
\end{equation}
where $D_1$ and $D_2$  are positive constants.
Should  $\tilde{J}$ be non-positive, it must take a non-positive minimum at $(r_3,t_3)$ with $0<r_3\leq 1$ and $t_1< t_3\leq t_2$.
At  $r_3=1$,  by the fact that $J_r(1,t)>0$ we have $\tilde{J}_r(1,t)>0$  leading to a contradiction. Thus the supposed minimum must have  $0<r_3<1$, where $ \tilde{J}_t\leq 0 , \tilde{J}_r=0 \;\mbox{and}\; \tilde{J}_{rr}\geq 0 .$
If we have $\tilde{J}\leq 0$ then equation  (\ref{62}) gives another contradiction.
 Therefore  $\tilde{J} $ and $J$ remain positive in $0<r<1$ for $t_1 \leq t \leq t_2.$

 The latter infers that equation (\ref{56}) holds at $t=t_2$,  contradicting to the initial assumption \eqref{ikk12}. So, as long as  solution $u$ exists then $f(t)> \lambda G(\epsilon)$ for $t\geq t_1.$ It then follows that $J>0,$ and  estimate (\ref{44}) holds together with
\begin{eqnarray}\label{63}
\int _0 ^1 r^{N-1} v^{-1} dr
&<& \frac{2^{\frac{1}{b+1}}}{(b+1) \epsilon)^{\frac{1}{b+1}}} \frac{b+1}{N b+N-2} = \frac{1}{Nb +N-2} \left(\frac{2}{\epsilon} \right) ^ {\frac{1}{b+1}} \left(b+1 \right) ^{1- \frac{1}{b+1}} \nonumber\\
&=&
 \frac{1}{Nb+N-2} \left(\frac{2}{\epsilon}\right)^{\frac{1}{b+1}} \left(b+1\right) ^{ \frac{b}{b+1}},
\end{eqnarray}
for $t\geq t_1,$ in case (\ref{41})-(\ref{43}) has a global solution $u$ or up to and including the quenching time $T_q$ when  $u$ quenches.
Finally by the definition of $H(u)$ and inequality \eqref{63} we obtain the desired estimate, \eqref{45}, and the lemma follows.
\end{proof}
\begin{rem}
Note that we can alternatively obtain that
\bgee
\epsilon_4=\inf_{t_1<t<t_2} \frac{f v^{-2}(1,t)+v_t(1,t)}{(N + b \beta) v^{-b}(1,t)+ b \beta}>0
\egee
without any restrictions on the spatial dimesnion $N,$ by choosing $\la$ large enough, i.e. $\la>\la^{**}\geq \la^*,$ so that
\bge\label{aal2}
f(t)=\frac{\lambda}{\left(1+\alpha N  \omega_N \int_0 ^1 r^{N-1} v^{-1}(r,t) dr \right)^2}>-v_t(1,t)v^{2}(1,t)\quad\mbox{for}\quad t\in(t_1,t_2),
\ege
which is always possible for a classical (and thus smooth enough) solution $u(r,t).$ Therefore, we can recover the result of Lemma \ref{lem2} independently of the dimension $N,$ but for $\la>\la^{**}$ so that \eqref{aal2} is satisfied. Consequently, in the sequel all the derived quenching results can alternatively be obtained for $\la$ large enough, in particular for  $\la>\la^{**},$ but without imposing any restrictions on the spatial dimesnion.
\end{rem}
Now having in place Lemmata \ref{lem1} and \ref{lem2} we are ready to prove the following quenching result. This result is sharp (optimal) in the sense that predicts quenching in the parameter range for the pull-in voltage $\la$ where no classical steady-states exist.
\begin{thm}\label{th_lambda}
Consider radially symmetric  initial data $u_0(r)$ with  $u'_0(r)<0.$ Assume also that   $N>\beta+1$ then for any  $\lambda > \lambda^*$ the solution of the problem (\ref{39}) quenches in finite time $T_q<\infty.$
\end{thm}
 \begin{proof}
Let assume to the contrary that for   some $\la>\la^*$ problem (\ref{39})
has a global-in-time solution. Then thanks to \eqref{12} and
\eqref{45} we can get a sequence $\{t_j\}_{j=1}^{\infty}$ with $t_j\to \infty$ as $j\to \infty$ such that
\bge\label{q3}
\la N \omega_N\int_{0}^1 r^{N-1} u_j(1-u_j)^{-2} \Ir \leq C_3 \, ,
\quad \mbox{for any} \quad t>0,
\ege
where the constant $C_3$ is independent of $j$.

Then by \eqref{45} it is readily seen that
\bge\label{q4}
N \omega_N \int_{0}^1 \frac{r^{N-1} \Ir}{(1-u_j)^2} & = &
N \omega_N \int_{0}^1 \frac{r^{N-1} \Ir}{(1-u_j)} +
N \omega_N \int_{0}^1 \frac{r^{N-1} u_j \Ir}{(1-u_j)^2}\nonumber \\
& \leq & (C_2 - 1) + \frac{C_3}{\la} := C_4,
\ege
where $C_4$ is independent of $j$.

Additionally by virtue of \eqref{11} we have
\bge\label{ps4}
||\nabla u_j||_{L^2(B_1)}^2\leq C_5<\infty,
\ege
where $C_5$ is again independent of $j.$

Passing to a sub-sequence, if necessary, relation \eqref{ps4}
infers the existence of a function $w$ such that
\bge
&&u_j\rightharpoonup w\quad\mbox{in}\quad H^1(B_1),\label{q8}\\
&&u_j\to w\quad\mbox{a.e.}\,\quad\mbox{in}\quad B_1,\label{q9}
\ege
as $j\to \infty.$
For $N\geq 2$ and by  \eqref{q4} we immediately obtain that $1/(1-u_j)^2$ is uniformly integrable and since
$$\frac{1}{(1-u_j)^2}\to \frac{1}{(1-w)^2},\;j\to \infty\quad\mbox{a.e.\; in}\quad B_1,$$
due to \eqref{q9},
we finally deduce
\bge\label{q10}
\frac{1}{(1-u_j)^2}\to \frac{1}{(1-w)^2} \quad \mbox{as}
\quad j\to \infty\quad\mbox{in}\quad L^1(B_1),
\ege
by virtue of  Lebesque dominated convergence theorem. Similarly we also derive
\bge\label{ps2}
H(u_j)\to H(w) \quad \mbox{as}
\quad j\to \infty\quad\mbox{in}\quad L^1(B_1).
\ege
Next note also that by relation \eqref{Energy}, see also \cite{KLN16}, we derive the following estimate
\bgee
\int_{\tau}^{\infty}\int_{B_1} u_t^2(x,s)\,dx\,ds\leq C<\infty,
\egee
for a constant $C$ independent of $\tau>0,$ and thus passing to a sub-sequence if it is necessary we obtain
\bge\label{cn5}
||u_t(\cdot,t_j)||_2^2=\int_{B_1} u_t^2(x,t_j)\,dx\to 0\;\mbox{as}\; j\to \infty.
\ege
A weak formulation of (\ref{39})  along
the sequence $\{t_j\}_{j=1}^{\infty}$ can be written as
\bge\label{ps3}
\int_{B_1} \frac{\partial{u_j}}{\partial t}\,\phi \Ix =
- \int_{B_1} \nabla u_j\cdot\nabla \phi \Ix
+\int_{\pl B_1} \frac{\pl u_j}{\pl \nu}\,\phi\,ds+ \la H^{-1}(u_j)\int_{B_1} \phi (1-u_j)^{-2} \Ix,
\ege
for any $\phi \in H^1(B_1).$

For any $\phi\in W^{2,2}(B_1)$ with $\frac{\partial \phi}{ \partial \nu}+ \beta \phi=0,\;\mbox{on}\;\;\pl B_1,$ then Green's identities imply
\bgee\label{rd1}
\int_{\pl B_1} \frac{\pl u_j}{\pl \nu}\,\phi\,ds&=&\int_{B_1} \nabla u_j \cdot \nabla \phi\,dx+\int_{B_1} (\Delta u_j)\,\phi\,dx\no\\
&=&\int_{B_1} \nabla u_j \cdot \nabla \phi\,dx+\int_{B_1}  u_j\,(\Delta\phi)\,dx
\egee
and thus  by virtue of \eqref{q8}, \eqref{q9} and Lebesque dominated convergence theorem we derive
\bge\label{cn4}
\int_{\pl B_1} \frac{\pl u_j}{\pl \nu}\,\phi\,ds\longrightarrow\int_{B_1} \nabla w \cdot \nabla \phi\,dx+\int_{B_1}  w\,(\Delta\phi)\,dx=\int_{\pl B_1} \frac{\pl w}{\pl \nu}\,\phi\,ds,
\ege
since $w\in H^1(B_1).$

Passing to the limit as $j\to\infty$ in \eqref{ps3}, and
in conjunction with \eqref{q8}, \eqref{q10}, \eqref{ps2},\eqref{cn5} and \eqref{cn4} we derive
\bgee
- \int_{B_1} \nabla \phi \cdot \nabla w\, dx+\int_{\partial B_1} \phi\, \frac{\pl w}{\pl \nu} \,ds +
\lambda \frac{\int_{B_1} \frac{\phi}{(1-w)^2} dx}{(1+ \int_{B_1}\frac{1}{1-w} dx)^2} =0,
\egee
for any $\phi\in W^{2,2}(B_1)$ satisfying $\frac{\partial \phi}{ \partial \nu}+ \beta \phi=0$ on $\pl B_1.$

The latter, according to Definition \ref{wes}, infers that $w$ is a weak finite-energy solution of problem
(\ref{39})  corresponding to $\la>\la^*$ which contradicts with
the result of Proposition \ref{cn3}.

For $N=1,$ using a similar approach and trace theorem, see also \cite[Theorem 3.5]{KLN16},
we obtain that $u_j$ converges to a weak finite-energy solution of problem
(\ref{39})  arriving again at a contradiction. This completes the proof of theorem.
 \end{proof}
 \begin{rem}
Notably  the quenching predicted by Theorem \ref{th_lambda} is single-point quenching. In particular, due to \eqref{44} we derive that  $u(r.t)$ can only quench at the origin $r=0.$
\end{rem}
\subsection{Quenching for large initial data}
 In the following we investigate the behaviour of the problem \eqref{39}  for large initial data. Namely, the following result holds.
\begin{thm}
For any $\lambda>0 $  and for $N>\beta+1$ we can choose initial data $u_0$ close enough to $1$ such that the solution $u$ of problem (\ref{39}) quenches in finite time $T_q<\infty.$
\end{thm}
\begin{proof}
We denote  by $(\lambda_1,\phi_1)$ be the principal eigenpair of
$$-\Delta \phi_1= \lambda_1 \phi_1, x \in B_1,\quad \frac{\partial \phi_1}{\partial \nu}+ \beta \phi_1=0 ,\; x \in \partial B_1,$$
where  again $\phi$ is normalized so that  $$\int_{B_1} \phi_1 (x) dx=1.$$
Let us suppose that problem (\ref{39}) has a global-in-time solution  $0<u(x,t)<1$ for any $(x,t) \in B_1 \times (0, \infty).$

Testing equation (\ref{39a}) with $\phi_1$ and integrating over $B_1$ then Green's second identity and Lemma \ref{lem2} infer,
\begin{eqnarray}
\frac{d}{dt} \int _{B_1} \phi_1 u \,dx &=& \int_{B_1} \phi_1 \Delta_r u dx + \lambda \int_{B_1} \phi_1 (1-u)^{-2}  (H(u))^{-2} dx\nonumber\\
&=& \int_{B_1} \Delta_r\phi_1  u dx + \int_{\partial B_1}\left(u \frac{\partial \phi_1}{\partial  \nu}-
 \frac{\partial u}{\partial  \nu}\phi_1\right) ds+ \lambda \int_{B_1} \phi_1 (1-u)^{-2}  (H(u))^{-2} dx\nonumber\\
 &=&-\int_{B_1} \lambda_1 \phi_1 udx+
  \frac{\lambda \int _{B_1} \phi_1 (1-u)^{-2} dx}{\left(H(u)\right)^2}.\label{73}
\end{eqnarray}
Set $A(t):=\int_{B_1} u \phi_1 dx,$ then
applying Jensen's inequality to equation (\ref{73}), we obtain
\begin{equation} \label{74}
\frac{dA}{dt} \geq -\lambda_1  A(t) + \frac{\lambda}{C^2 _2}(1-A(t))^{-2}, \quad \mbox{ for any } \quad t>0.
\end{equation}
Next we choose suitable $\gamma \in (0,1)$ such that
$$\Psi (s) := \frac{\lambda}{C^2 _2} (1-s)^{-2} - \lambda_1 s >0 \quad \mbox{for all } \quad s \in [\gamma,1),$$
and then by choosing $u_0$ such that $A(0)=\int_{B_1} u_0 \phi_1 dx\geq \gamma,$ then (\ref{74}) infers
$$\frac{dA}{dt}\geq \Psi(A(t))>0 \quad \mbox{for any} \quad t>0,$$
or by integrating
$$t\leq \int_{A(0)} ^{A(t)} \frac{ds}{\Psi(s) } \leq \int_{A(0)} ^1 \frac{ds}{\Psi(s)}<\infty.$$
The latter is in contradiction with our initial  assumption that $T=\infty,$ and the theorem is proved.
\end{proof}
\subsection{Behaviour at quenching }
In the current subsection we  give  more details regarding the behaviour of  quenching solutions  close to quenching time  $T_q.$

We first  obtain the quenching rate. Let us recall that a solution  $u(r,t)$   of   \eqref{39} with radial decreasing initial data $u_0$ then $u$  is also radial decreasing and thus
\bgee
M(t):= \max_{x\in\bar{B_1}} u(x,t)=u(0,t).
\egee
The next result determines the quenching rate   of $u$  for singular solutions of \eqref{39}.
\begin{thm} \label{thm:bound}
Let $u(r,t)$ be a quenching solution of \eqref{39}. Then  for $N>\beta+1$ there are positive constants $\widehat{C}, \widetilde{C}$ indpendent on time $t$ such that
\bge\label{lb}
1-\widehat{C}(T_q-t)^{1/3}\leq M(t)\leq 1-\widetilde{C}(T_q-t)^{1/3}\quad\mbox{for}\quad 0<t-T_q\ll 1.
\ege
\end{thm}
\begin{proof}
Since  $M(t)$ is Lipschitz continuous then by Rademacher's theorem, is almost everywhere differentiable,
cf.  \cite{FM, KN07}. Furthermore, since $u$ attains a maximum at $r=0$ then
$\Delta_r u(0,t)\leq 0$ for all $t\in(0,T_q)$. Therefore, for any $t$
where $\dd M/\dd t$ exists, we derive
\bgee
\dif Mt \leq \la \frac{(1-M(t))^{-2}}{\left( 1+ \int_{B_1}
\frac{1}{1-u} \Ix \right)^2} \leq \la \frac{(1-M(t))^{-2}}
{\left( 1 + N \omega_N \right)^2} \quad \mbox{for a.e.} \quad t\in(0,T_q),
\egee
which yields
\bgee
\int_{M(t)}^1 (1-s)^2 \I s \leq \la C (T_q-t),
\egee
for $C=1/(1 + N \omega_N)^2. $ The latter implies
\bge\label{lb_ineq1}
M(t) \geq 1 - \widehat{C}(T_q-t)^{1/3} \quad \mbox{for} \quad 0<t<T_q \, ,
\ege
where $\widehat{C}=(3\la C)^{1/3}.$
%_______________________________________

 Note that inequality \eqref{45}
implies that $H(u)$ is uniformly integrable so then via,  \eqref{44}
 and parabolic regularity estimates in the region $r\in(0,1),$ cf. \cite{lsu68},  we obtain that
\bge\label{lb1}
\lim_{t\to T_q}u(r,t)=u(r,T_q) \quad \mbox{for any} \quad 0<r<1 .
\ege
Estimate  \eqref{44} also implies that
\bgee
(1-u)^{-1}\leq \overline{C}(k)r^{-k},
\egee
for $k>\frac23$, and $\overline{C}(k)=\frac{1}{C(k)}$
and thus from relation \eqref{lb1}, and the Lebesque dominated convergence theorem we get that
\bgee
\lim_{t\to T_q}\int_{B(0,1)}\frac{1}{1-u(x,t)}dx=\int_{B(0,1)}\frac{1}{1-u(x,T_q)}dx<\infty
\egee
and finally
\bgee
\lim_{t\to T_q}\left(H(u)\right)^2=K<\infty.
\egee
Therefore for $0<t- T_q\ll 1$ we have that
\bgee
& u_t(x,t)\simeq \Delta u + \frac{\lambda}{K}\frac{1}{(1-u)^2}, \quad x\in B(0,1),\\
& \frac{\partial u}{\partial \nu}(x,t) + \beta u(x,t)=0,\quad x\in\partial B(0,1),\\
& u(x,0)=u_0(x).
\egee
But for the above local problem it is known, cf. \cite{FG93, MZ97}, that
\bge\label{lb_ineq2}
M(t)=u(0,t)\lesssim  1-{\widetilde{C}}(T_q-t)^{1/3},
\ege
for some $\widetilde{C}>0.$

Therefore combining inequalities \eqref{lb_ineq1},  \eqref{lb_ineq2} we obtain the required estimation
\eqref{lb}
\end{proof}

It is worh noting that due to the uniform bounds of nonlocal term $ \left( H(u)\right)^2$ we can treat nonlocal problem \eqref{39} as a local one and therefore the quenching profile is given as follows, cf. \cite{FG93, MZ97}
\bge\label{rnc7}
1-u(r,T_q)\sim C^*\left[\frac{|r|^2}{|\ln|r||}\right]^{1/3}\quad\mbox{as}\qquad r\to 0^+,
\ege
for some positive constant $C^*$. For a more rigorous approach, which is out of the scope of the current work,  one should follow similar arguments as in \cite{DZ19, GH18} to derive \eqref{rnc7} where it is conjectured that  $C^*=\lim_{t\to T_q} H(u(r,t)).$
\section{Numerical Approach }\label{nap}

 In the current section we present a numerical study of problem (\ref{o.oneN1}) both in the one-dimensional  as well as in the two-dimensional radial symmetric case. For that purpose an adaptive method monitoring the behaviour of the solution near a singularity, such as the detected quenching behaviour of (\ref{o.oneN1}),  is used (e.g.  see  \cite{Budd, KLNT15}).

\subsection{One-dimensional case}
For the one-dimensional case and for the sake of simplicity, taking advantage of the symmetry of the solution,  we may consider the problem in the interval $[0,1]$ with Neumann condition at $x=0,\;u_x(0,t)=0$ and the original Robin condition at the point $x=1.$

 Initially we take a partition of $M+1$ points in the interval [0,1], $\xi_0=0, \xi_1=\xi_0+\Delta \xi,..., \xi_M=1.$ For $u=u(x,t),$ we introduce a computational coordinate $\xi$ in [0,1] and we consider the mesh points $X_i$ to be the images of the points $\xi_i$ under the map $x(\xi,t)$ so that $X_i(t)=x(i \Delta \xi,t).$ By the latter relation we obtain  $\frac{d u(X(t),t)}{dt}=u_t(X_i,t)+u_x X_i '$ for the approximation of the solution $u_i(t)\simeq u(x_i(t),t).$

 Moreover the map $x(\xi,t)$ is determined by the function $\mathcal{M}(u)$ which in a sense, follows the evolution of the singularity in case of quenching. This function is determined by the scale invariants of the problem. In particular,  for the semilinear parabolic equation
$$v_t =v_{xx} - \frac{\lambda}{v^2 \left[1+ \alpha\int_{-1} ^1 1/v\, dx\right]^{2}},$$ where $v=1-u$, an appropriate monitor function should be of the form $\mathcal{M}(v)=\vert 1-u \vert ^{-2}$  or $ \mathcal{M}(v)=\vert v \vert ^{-2}$.

We need also a rescaling of time of the form $\frac{du}{dt}=\frac{du}{d \tau} \frac{d \tau }{dt}$ where $\frac{dt}{d  \tau}=g(u)$, and $g(u)$ is a function determining the way that the time scale changes as the solution approaches the singularity. In particular, we have $g(u)=\frac{1}{\parallel \mathcal{M}(u)\parallel _{\infty}}.$

In addition the evolution of $X_i(t)$ is given by a moving mesh PDE which is of  the form $x_{\tau \xi \xi} =\epsilon^{-1}g(u)(\mathcal{M}(u)x_ \xi)_{\xi}.$ Here $\epsilon$ is a small parameter accounting for the time scale. Thus finally we obtain a system of ODEs for $X_i$ and $u_i.$  The undelying  ODE system takes the form
\begin{eqnarray}
\nonumber
&&\frac{dt}{d\tau}  = g(u),\\
\label{neq2}
&& u_{\tau}- x_{\tau} u_x = g(u)
\begin{pmatrix} u_{xx} + \frac{\lambda}{(1-u)^2 (1+ \alpha \int_0 ^1 \frac{1}{1-u} dx)^2} \end{pmatrix},\\
\nonumber
&& -x_{ \tau \xi \xi}= \frac{g(u)}{\epsilon} (\mathcal{M}(u) x_ {\xi})_{\xi}.
\end{eqnarray}

We  apply  a discretization in space to  derive
\begin{eqnarray*}
u_x(X_i, \tau) \simeq \Delta _x u_i(\tau) &:=&- \frac{u_{i+1}(\tau)- u _{i-1}(\tau)}{X_{i+1}(\tau)- X _{i-1}(\tau)},\\
u_{xx}(X_i, \tau)\simeq \Delta^2 _x u_i(\tau)&:=&
\left(
 \frac{u_{i+1}(\tau)-u_{i}(\tau)}{X_{i+1}(\tau)-X_{i}(\tau)}- \frac{u_{i}(\tau)-u_{i-1}(\tau)}{X_{i}(\tau)-X_{i-1}(\tau)}
 \right)\frac{2} {X_{i+1}(\tau) - X_{i-1}(\tau)}
 ,\\
 x_{\xi \xi}(\xi _i, \tau) \simeq \Delta^2 _{\xi} x_i(\tau)&:=& \frac{X_{i+1}(\tau)-2X_i(\tau)+ X_{i-1}(\tau)}{\delta \xi ^2},\\
(\mathcal{M} (u) x_ {\xi})_{\xi} \simeq \Delta_{\xi} (\mathcal{M}  \Delta_{\xi} x) &:=& - \left(
 \frac{\mathcal{M} _{i+1}+\mathcal{M} _{i}}{2} \frac{x_{i+1}-x_i}{\Delta \xi}- \frac{\mathcal{M} _{i}+\mathcal{M} _{i-1}}{2} \frac{x_i -x_{i-1}}{\Delta \xi}
 \right) \frac{1}{\Delta \xi}.
\end{eqnarray*}
Notably  at the boundary point $X_M=1$ the discretized boundary condition $u_M=u_{M-1}-\beta u_M\left( X_M-X_{M-1}\right)$ has been  used.

The preceding spatial discretization leads to an ODE system of the form
\bge\label{rnc2}
A(\tau, y) \frac{dy}{d \tau}= b (\tau,y),
\ege
with the vector $y \in \mathbb{R}^{2n+1}$ defined as
 $$y = (t(\tau ), u_1(\tau ), u_2(\tau ), . . . , u_M(\tau ),X_1(\tau ),X_2(\tau ), . . . ,X_M(\tau )) ,
= (t(\tau ),\, \textbf{u},\, \textbf{X}), \textbf{u, X} \in  \mathbb{R} ^M,$$
and  $A \in \mathbb{R}^{{2n+1},{2n+1}} .$  System \eqref{rnc2}  has the block form
$$A=\begin{bmatrix} 1 & 0 & 0 \\ 0 & I & -\Delta_xu   \\
 0 &0 & -\Delta_{\xi} ^2\end{bmatrix} , \quad y=\begin{bmatrix}
 t(\tau) \\ u \\ X
 \end{bmatrix}, \quad b=g(u) \begin{bmatrix}
 1 \\ \Delta_x^2 u + \lambda \frac{1}{(1-u)^2 (1+\alpha \textbf{\textit{I}}(u))^2} \\ \Delta _{\xi}( \mathcal{M}  \Delta_{\xi} x)
 \end{bmatrix},
$$
where $\textbf{\textit{I}}(u)$ is an approximation of the integral $\int_0 ^1 \frac{1}{1-u} dx,$ using  Simpsons' method.
For the solution of \eqref{rnc2}  a standard ODE solver, such as the matlab function ``ode15i",  can be  used.
%_______________________________________

\paragraph{\bf The Local Problem} Initially we present a simulation for the local problem,  (\ref{o.onel}), i.e. problem  (\ref{o.oneN1})  for $\alpha=0$.
 In Figure \ref{figsim1} some numerical experiments presented for the case where a global-in-time solution exists.
In the first of these graphs (top left) we plot the solution against space and time. In the second one (top right) we plot the moving mesh $X(i,t)$ against time, while in the third (bottom left)  a sequence of profiles of the solution ($u(x,t_i)$) for various time steps $t_i$ is presented. Finally in the fourth graph we plot the maximum of the solution $u(0,t)$ against time. The latter plot shows the convergence  towards a steady state. The initial condition here, as well as in the rest of the simulations, was taken to be zero, $u_0(x)=0$.
Also the parameters used here were $\lambda=0.05$,  $\beta=1$, $t\in [0,T_f]$, $T_f=40$, $M=141$.
\begin{figure}[h]\vspace{5cm}
  \centering
 \includegraphics[ bb= 320 10 100 100, scale=.8]{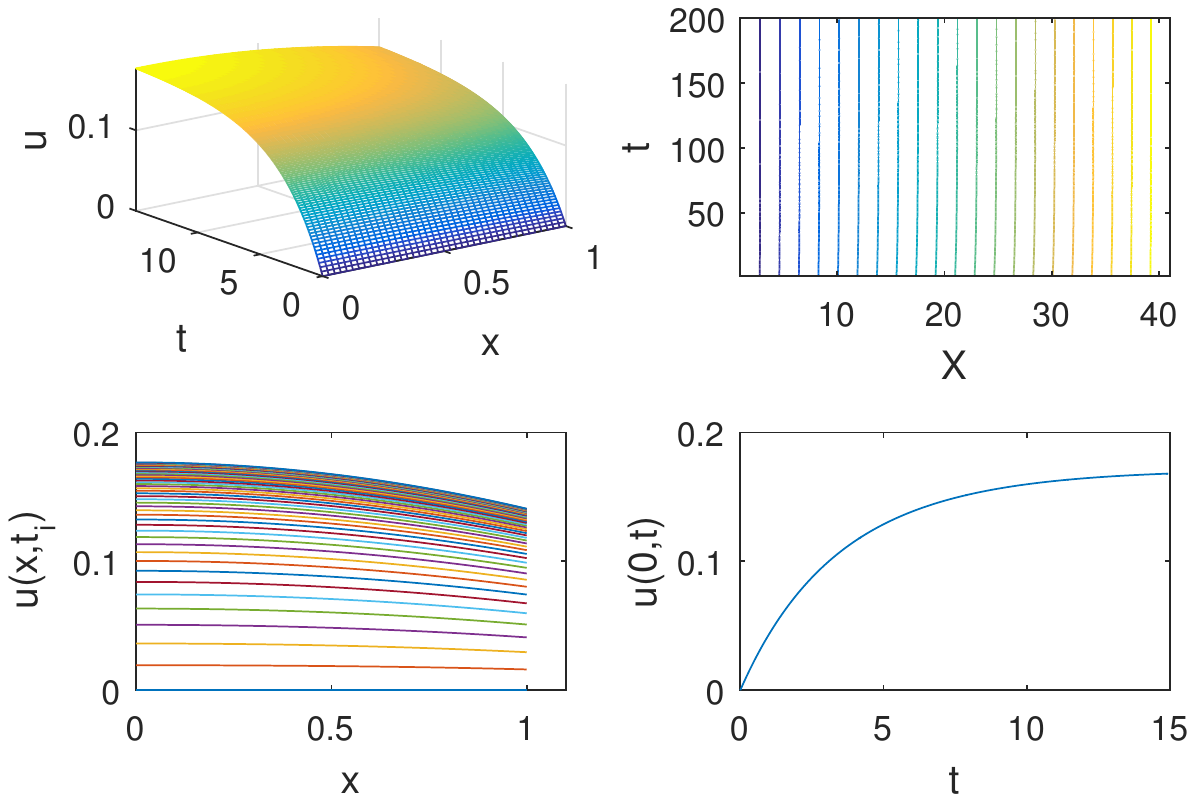}
 \caption{Form of the solution and various profiles of the local problem  for $\lambda=0.05$, $\beta=1$}.
  \label{figsim1}
\end{figure}
 Figure \ref{figsim2} depicts the situation where the solution quenches in finite time. Again in the first of these graphs (top left) we plot the solution against space and time. In the second one (top right) we plot the moving mesh $X(i,t)$ against time. Here the motion of $X_i$'s captures the observed singularity, i.e. the finite-time quenching. In the third (bottom left) a sequence of profiles of the solution ($u(x,t_i)$) for various time steps $t_i$ is presented. We can observe the increasing with time  profiles of the solution.  Finally in the fourth graph we plot the maximum of the solution $u(0,t)$ against time from which the quenching behaviour is revealed. The same parameters  as in Figure \ref{figsim1} are used  but with $\lambda=1$.
\begin{figure}[h]\vspace{4.5cm}
  \centering
  \includegraphics[ bb= 320 10 100 100, scale=.8]{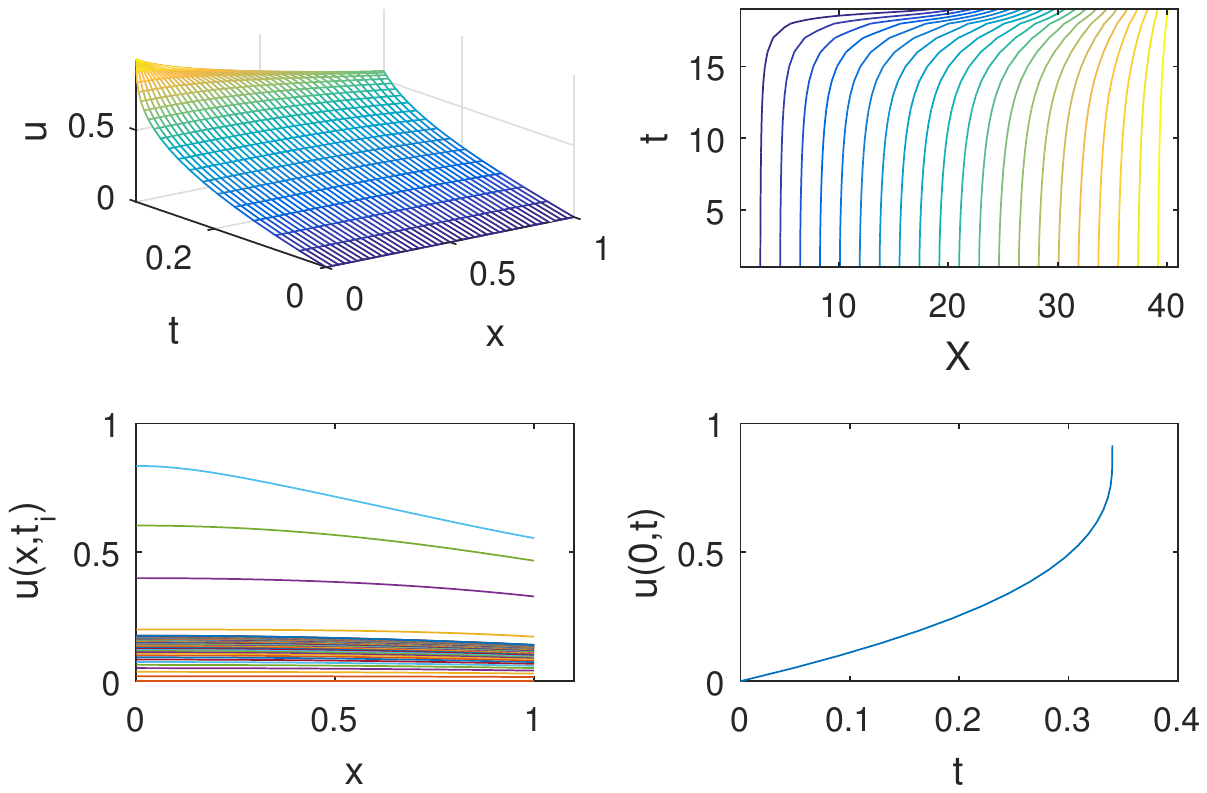}
 \caption{Form of the solution and various profiles of the local problem in the case of quenching for
 $\lambda=1$, and $\beta=1$}.
  \label{figsim2}
\end{figure}
In the next Figure, \ref{figsim3}  we plot the profiles of the solution maximum, $u(0,t)$ against time, for various $\lambda$'s and specifically for $\lambda=.7,\,.8,\,.9,\,1$. We observe that by increasing the value of the parameter $\lambda$  the quenching time decreases as it is expected.
\begin{figure}[h]\vspace{-8cm}
  \centering
 \includegraphics[width=.8\textwidth]{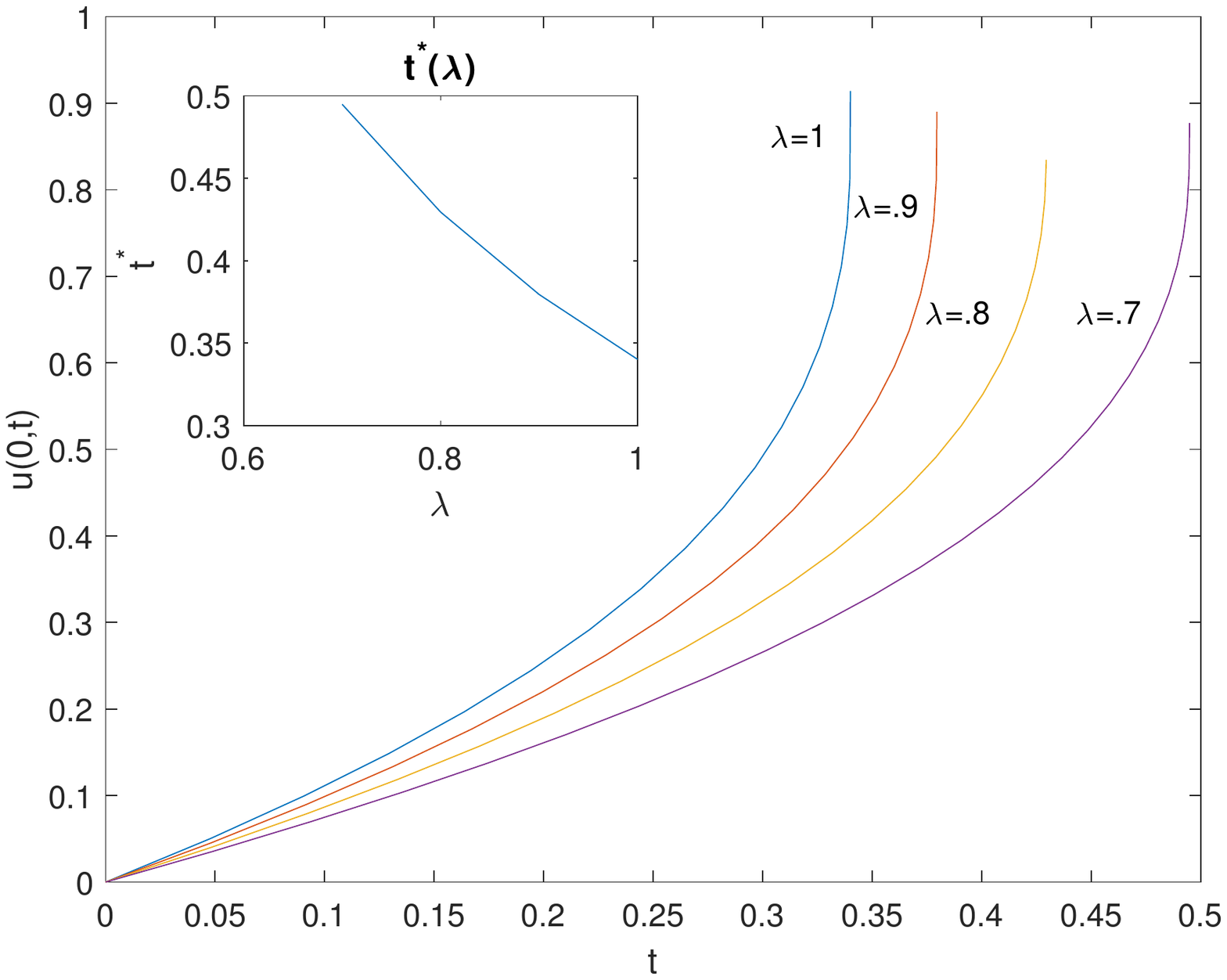}
 \caption{Form of the solution maximum against time for  various values of the parameter $\lambda$ for the local problem for $\beta=1$}.
  \label{figsim3}
\end{figure}

%____________________________________________________________
\paragraph{\bf The Non-Local Problem} A similar set of simulations is presented for the case that
$\alpha=1$
 while the rest of the parameters, unless otherwise stated, are kept the same as in the experiment of Figure \ref{figsim1}.
 In Figure \ref{figsim4} and for $\lambda=0.5$ the convergence of the solution towards a steady state is depicted.
 %_________________________________
\begin{figure}[h]\vspace{6cm}
  \centering
 \includegraphics[ bb= 320 10 100 100, scale=.8]{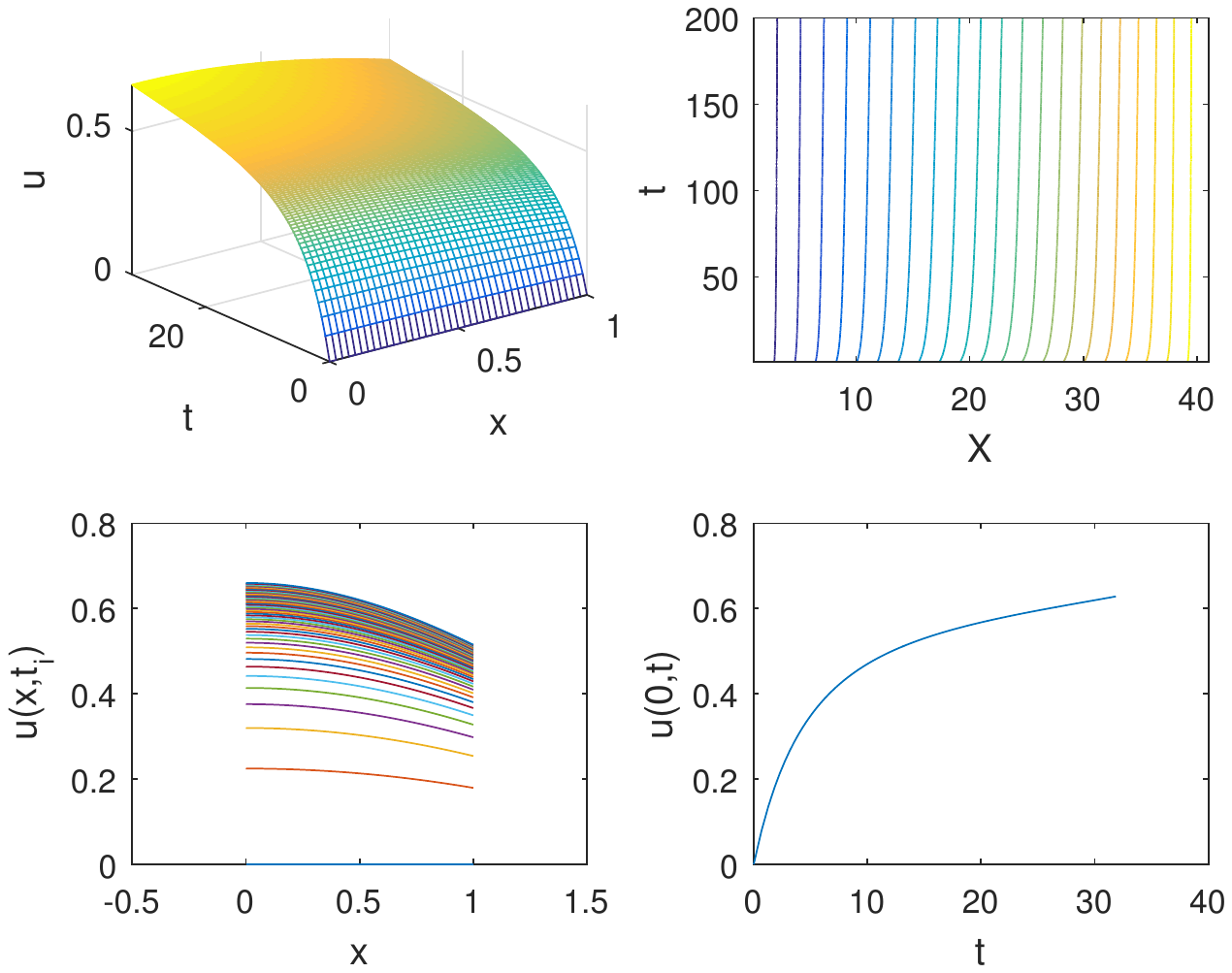}
 \caption{Form of the solution and various profiles of the nonlocal problem,  for
   $\la=0.5$, $\alpha=1$, $\beta=1$}.
  \label{figsim4}
\end{figure}
%_______________________________________
 In a similar set of graphs, see Figure \ref{figsim5} and for $\lambda=3$, we present
 the quenching behaviour of the solution.
 %_____________________________________
\begin{figure}[h]\vspace{5cm}
  \centering
   \includegraphics[ bb= 320 10 100 100, scale=.8]{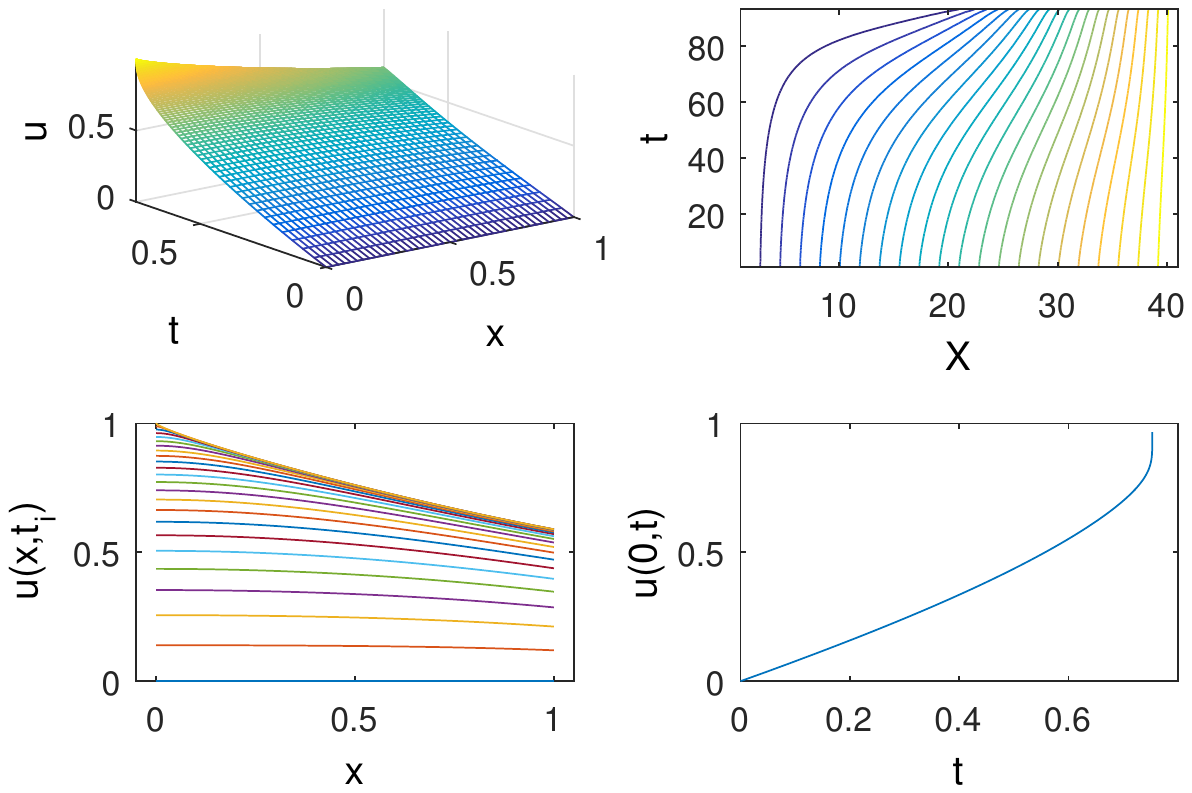}
 \caption{Form of the solution and various profiles of the nonlocal problem
  for  $\la=3$, $\alpha=1$,  $\beta=1$}.
  \label{figsim5}
\end{figure}
%_____________________________________________________
Moreover in Figure \ref{figsim6} we can observe the evolution of the quenching time
as  the value of the parameter $\lambda$ varies, something cannot be seen via our theoretical results. In particular,by  increasing the parameter $\lambda$ results in a decreasing  of quenching time. Here $\lambda=2.5,\,3,\,3.5,\,4.$
%________________________________________________________
\begin{figure}[h]\vspace{-8cm}
  \centering
  \includegraphics[width=0.8\textwidth]{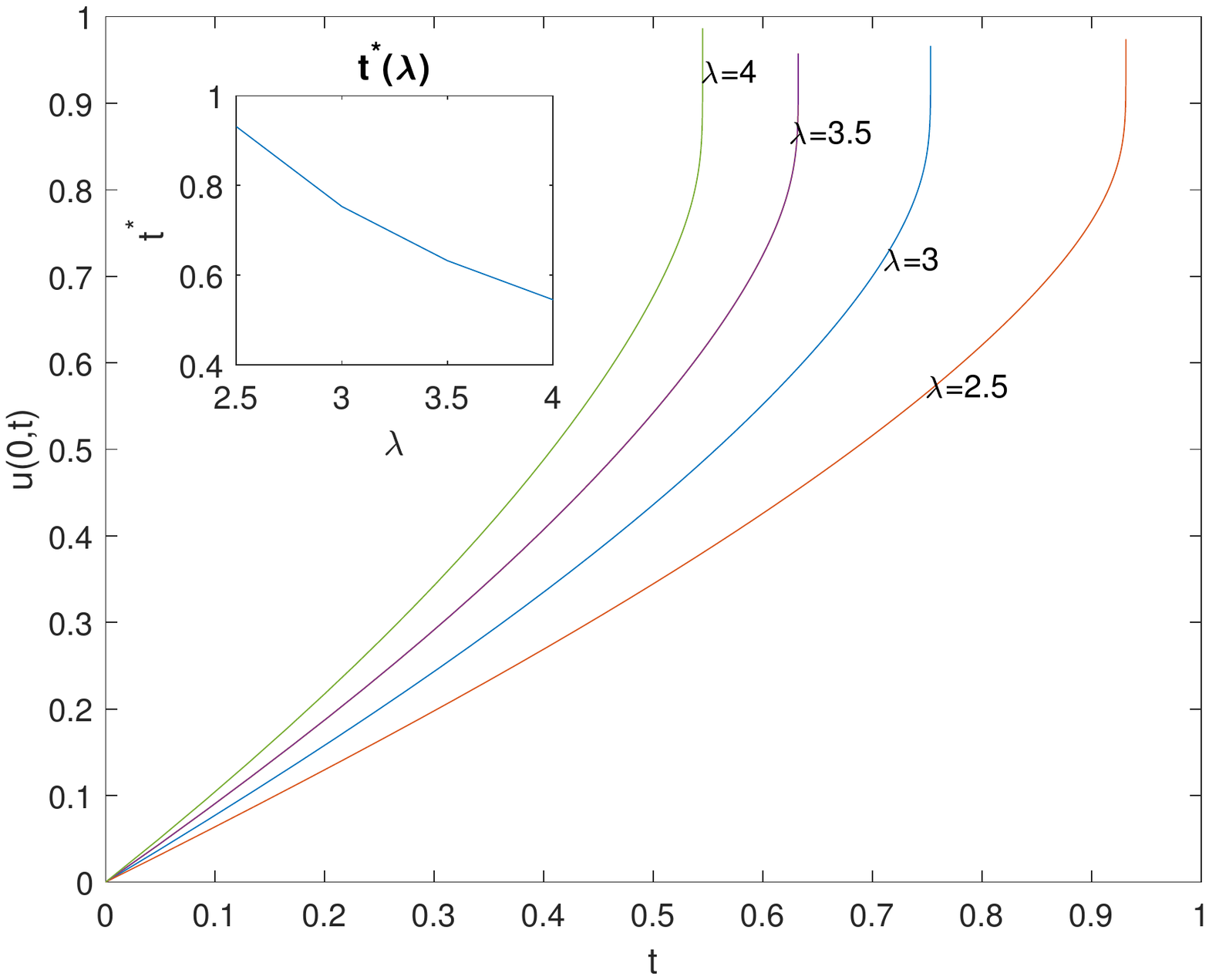}
 \caption{Form of the solution maximum against time for  various values of the parameter $\lambda$
  and with  $\alpha=1$, $\beta=1$}.
  \label{figsim6}
\end{figure}
%____________________________________________________________
 Next in Figure  \ref{figsim7}(a) we plot a series of profiles for the maximum of the solution as the parameter $\alpha$ varies. Again such a behaviour cannot be unveiled via our analystical results in subsections \ref{rnc4} and \ref{rnc3}. It is easily seen that by decreasing $\alpha$ the quenching time decreases too. The parameter $\alpha$ decreases from $1$ to the value $0$ whilst the parameter $\lambda$ is kept constant and equal to $\lambda=2$.
 
%_________________________________________________
%_____________________________________________________________
\begin{figure}[htb]\vspace{-8cm}
   \begin{minipage}{0.48\textwidth}
     \centering
     \includegraphics[width=.9\linewidth]{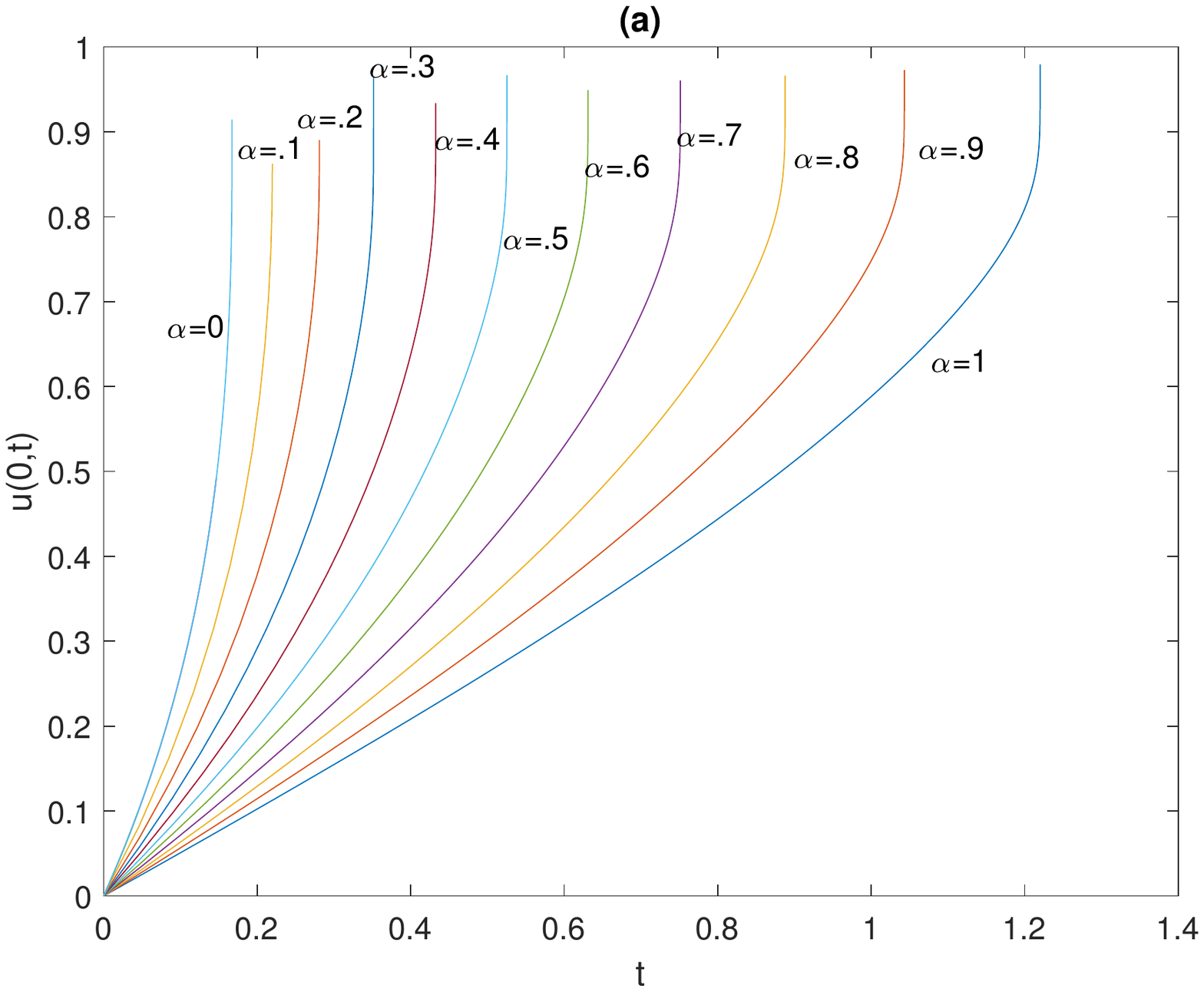}%\vspace{-6cm}
   \end{minipage}%\hfill
   \begin{minipage}{0.48\textwidth}
     \centering%\vspace{-3cm}
     \includegraphics[width=1.3\linewidth]{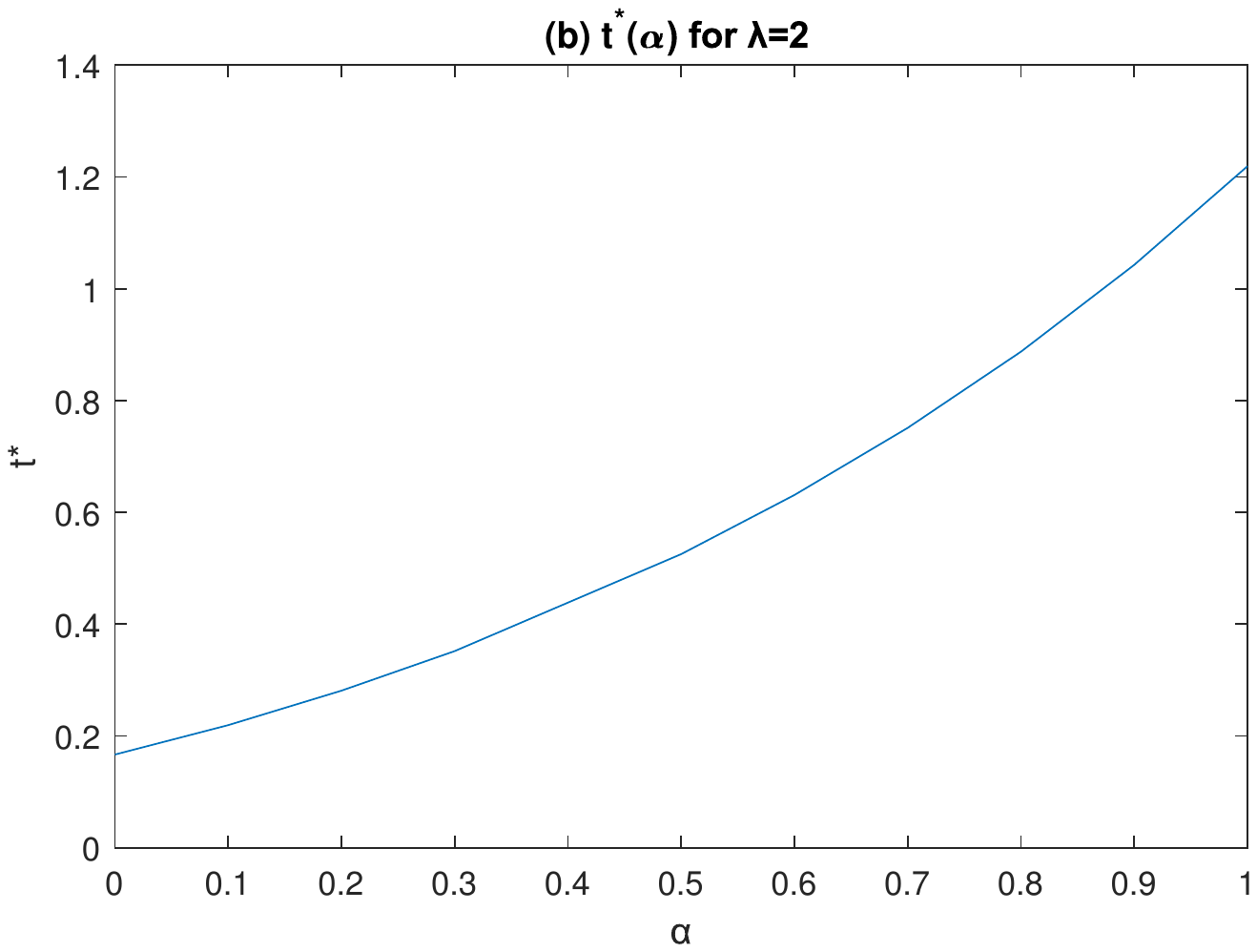}\vspace{4.5cm}
   \end{minipage} \vspace{-4cm}
   \caption{
   (a)  Form of the solution maximum against time for  various values of the parameter $\alpha$ for $\lambda=2$
   and $\beta=1$.
 (b)   Variation of the quenching time of the nonlocal problem  with respect to the parameter $\alpha$.
 }\label{figsim7}
\end{figure}
%________________________________________________________
The effect of the boundary parameter $\beta$ is unveiled  by Figure  \ref{figsim8}(a), a fact cannot be easily seen by our theoretical results in section \ref{rnc5}. Indeed, it is seen that by increasing $\beta$ a long-time behaviour resembles the one of the Dirichlet problem is derived.
%_________________________________________
%__________________________________________
 The variation of the quenching time $t^*$ of the nonlocal problem is depicted in a series of plots in Figures \ref{figsim7}(b) and
 \ref{figsim8}(b). In the first of them, Figure \ref{figsim7}(b), we present a plot of $t^*(\alpha)$ while in the second \ref{figsim8}(b),  a plot of $t^*(\beta)$. In both cases was taken $\lambda=2$.
%_____________________________________________________________
\begin{figure}[htb]\vspace{-8cm}
   \begin{minipage}{0.48\textwidth}
     \centering
     \includegraphics[width=1.2\linewidth]{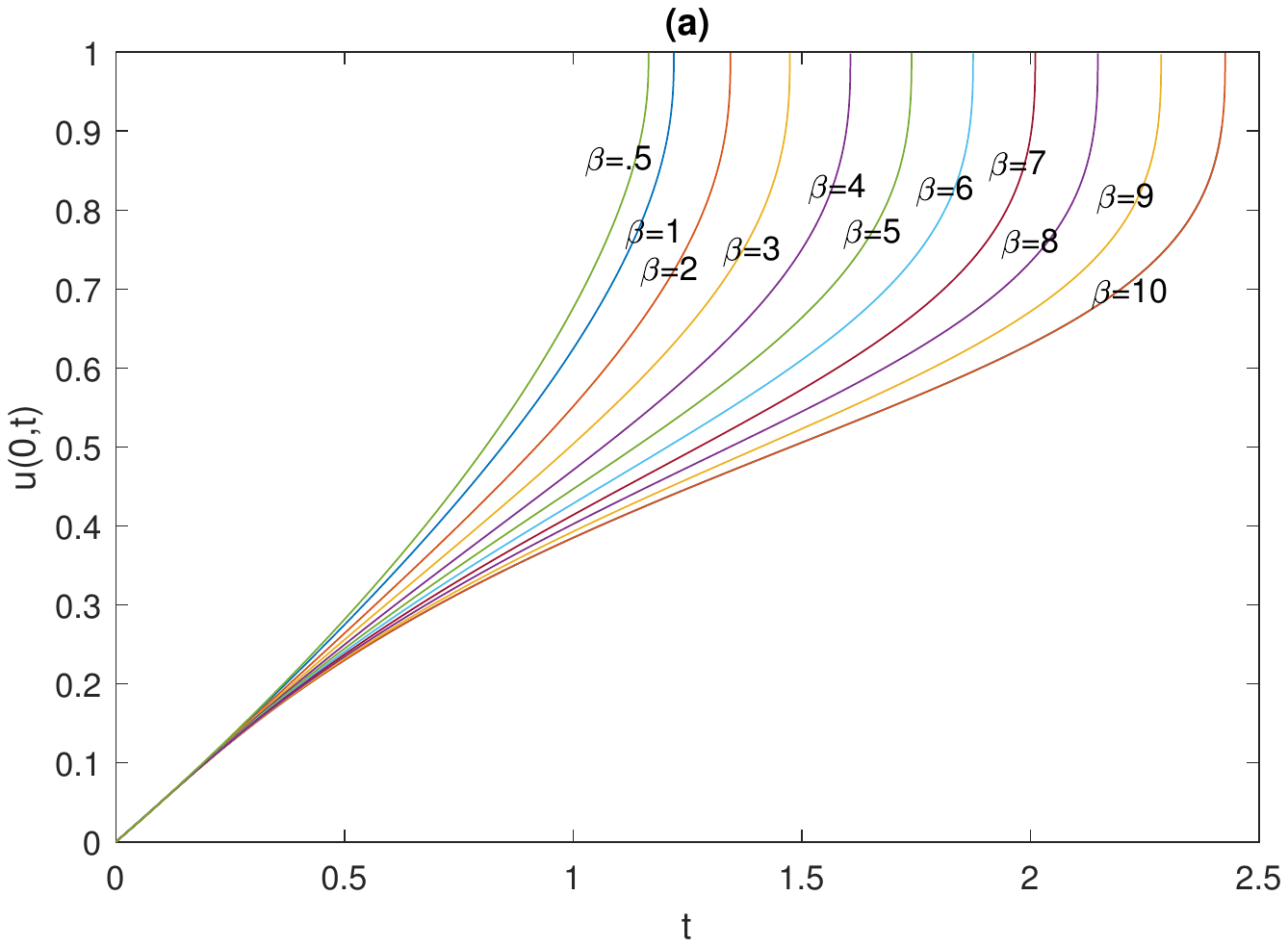}\vspace{1cm}
   \end{minipage}%\hfill
   \begin{minipage}{0.48\textwidth}
     \centering%\vspace{-3cm}
     \includegraphics[width=1.2\linewidth]{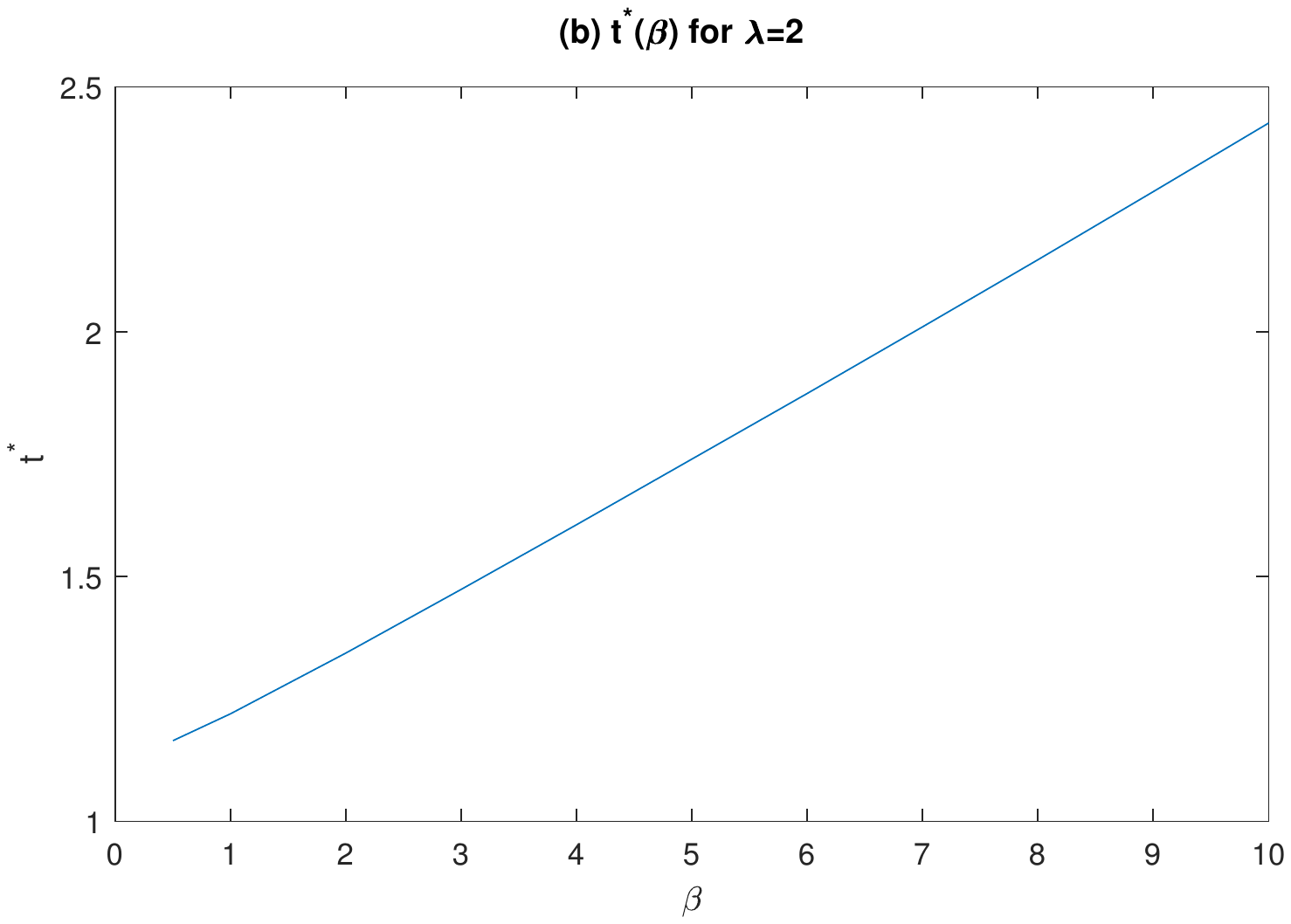}%\vspace{-3cm}
   \end{minipage}% \vspace{-4cm}
   \caption{(a) Form of the solution maximum against time for  various values of the parameter $\beta$
  for $\lambda=2$  and $\alpha=1$. (b)
  Variation of the quenching time  of the nonlocal problem with respect to the parameter $\beta$.}\label{figsim8}
\end{figure}
%________________________________________________________
\subsection{The Radial Symmetric Case} It has been already pointed out that the  $2-$dimensional problem in the radially symmetric case is very interesting from the point of view of applications and thus we choose to provide a numerical treatment for it in  the current subsection. For this purpose the   aforementioned
adaptive  numerical scheme and specifically equation (\ref{neq2}) can be modified accordingly with
 $ u_{xx}+(N-1) r^{-1} u_x$  used in place of $u_{xx}$.

Initially we solve the local problem, i.e. problem  \eqref{39} for $\alpha=0$ and the results are presented in Figure
\ref{figrad1}. Here we take  $\beta=1$, $\lambda=0.05$ and we observe that the solution converges towards a steady state.
%_______________________________________________
 \begin{figure}[h]
  \centering
  \vspace{6cm}
 \includegraphics[bb= 320 10 100 100, scale=.8]{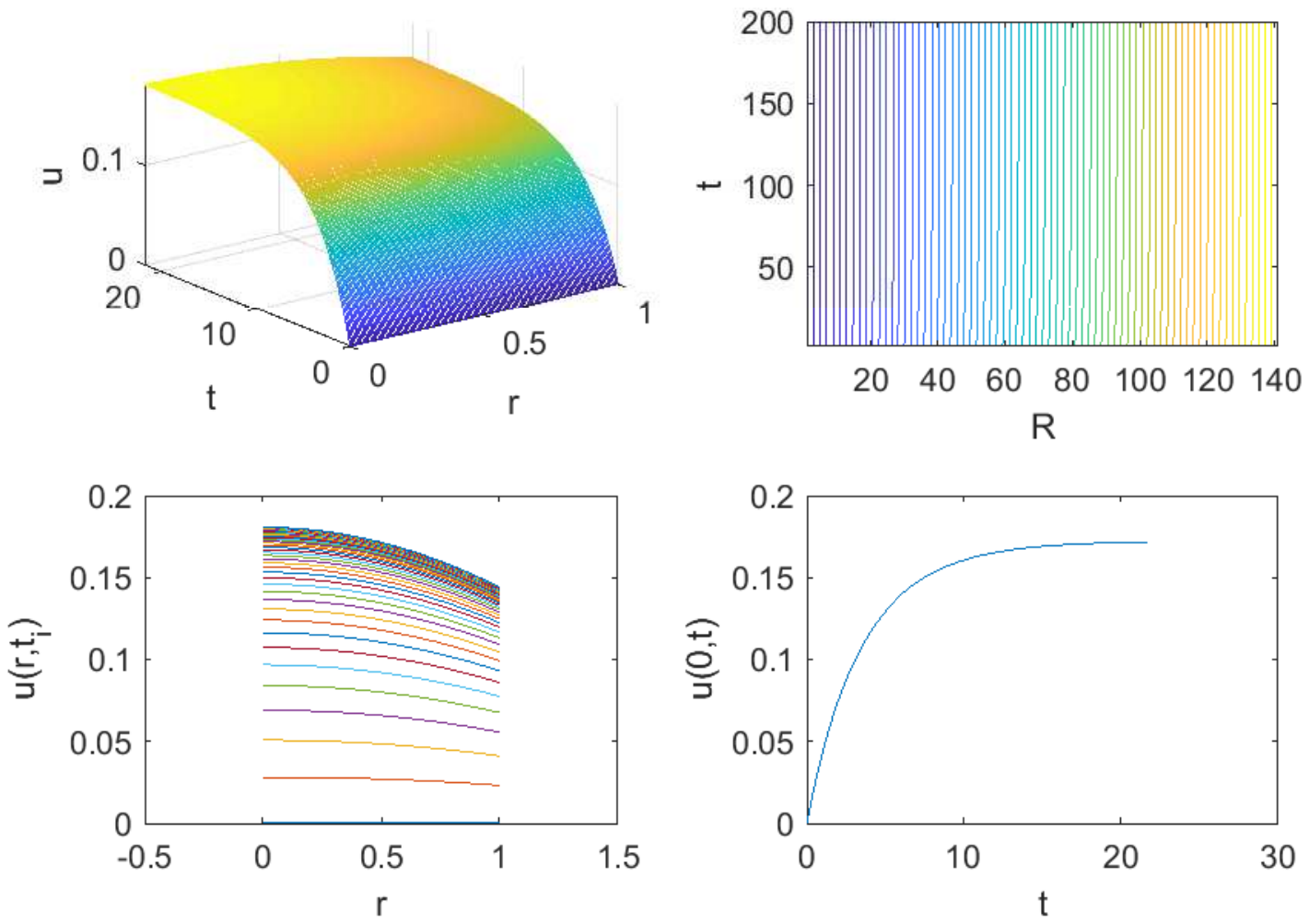}
 \caption{Form of the solution and various profiles of the local problem for the radial symmetric case
 for $\lambda=0.05$ and $\alpha=0$, $\beta=1$.}
  \label{figrad1}
\end{figure}
%________________________________________
 In Figure \ref{figrad2} we present an analogous  simulation for the nonlocal problem.  In that case we take
 $\alpha=1, \beta=1$, and $\lambda=0.2$ and we derive that the solution quenches in finite time.
 %______________________________________________
 \begin{figure}[h]
  \centering
 \includegraphics[bb= 300 230 250 580, scale=.8]{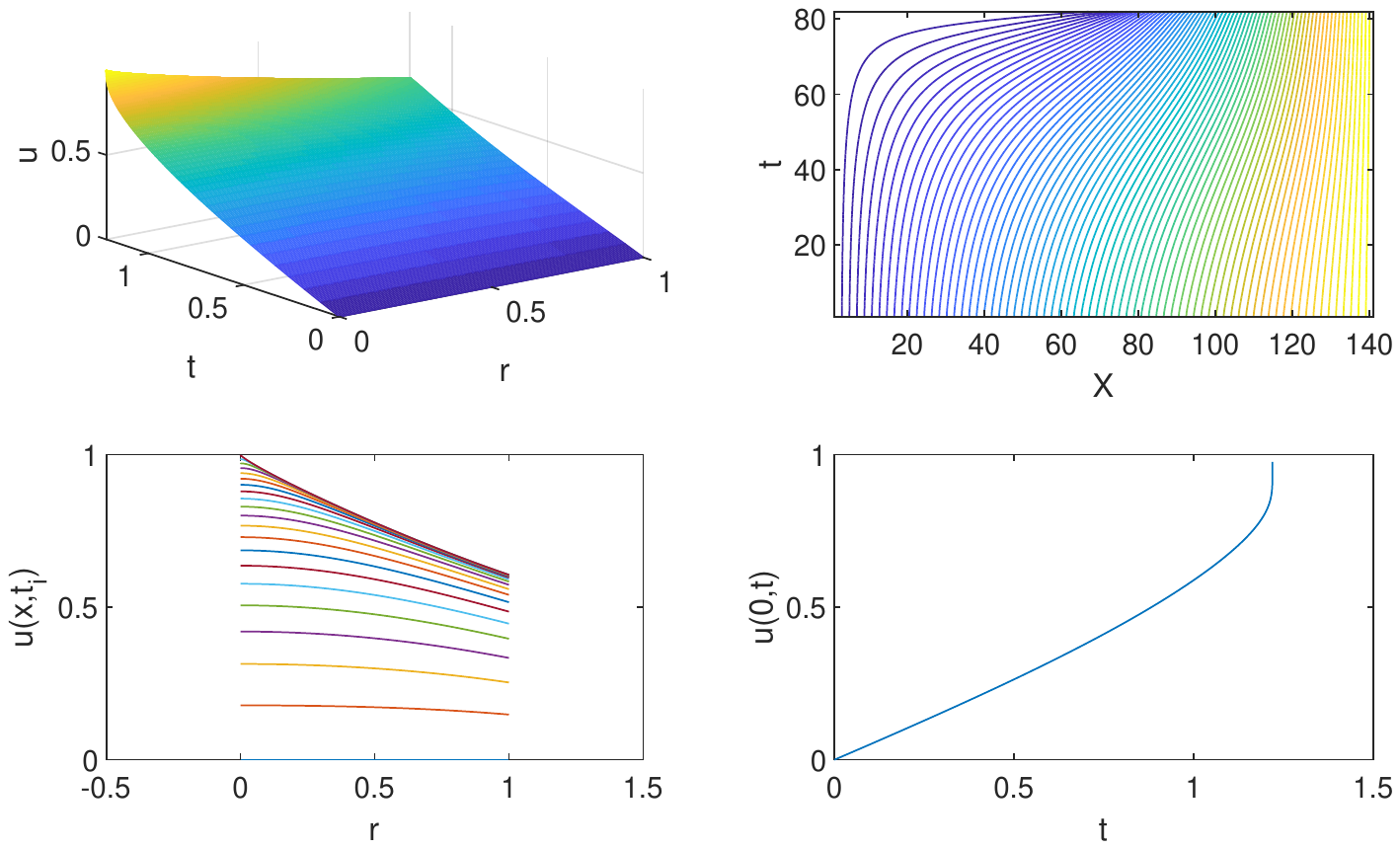}
 \vspace{-1.5cm}
 \caption{Form of the solution and various profiles of the nonlocal problem for the radial symmetric case for $\lambda=0.2$ and $\alpha=1$, $\beta=1$.}
  \label{figrad2}
\end{figure}
%_____________________________________________________

\section{Discussion}\label{dsc}

  In the current work we investigate a nonlocal parabolic problem with Robin boundary conditions associated with the operation of some idealized MEMS device. In the first part we deliver  a thorough  investigation of  the  associated steady-state problem
 and we derive some estimates of the {\it pull-in voltage},  which is the controlling parameter of the model. In particular, and for the  $N-$dimensional case , $N>1, $  in order to derive sharp estimates for the {\it pull-in voltage}  we had to show,  as  a very interesting  by-product,
 a Poho\v{z}aev's type identity for  Robin boundary conditions.  To the best of our knowledge such a result  has not been available in the literature.

 In the second part of this work, existence and  uniqueness results  together with long time behaviour of  time-dependent problem are discussed. In particular, we focus  on the investigation of  the phenomenon of  quenching  (i.e. the so called {\it touching down} in the context of MEMS literature). We first examine the quenching behaviour on a general domain, whilst later  in order to derive  an optimal quenching result we restrict ourselves to the radially symmetric case.
	
   Finally we close our investigation by the implementation of an adaptive numerical method, \cite{Budd}, for the solution of the time-dependent problem. We actually perform  a series of numerical experiments verifying the obtained analytical results as well as revealing qualitative features of nonlocal problem \eqref{o.oneN1} do not arise from our analytical approach. Additionally, some further numerical experiments are performed to determine the quenching profile of the solution in the radially symmetric case.
   
\section*{Acknowledgments}
The authors would like to thank the anonymous referees for the carefull reading of the manuscript. Actually, their fruitful comments and suggestions  improved substantially the final form of this work.
%__________________________________________

%__________________________________________

%__________________________________________
%_______________________________________________
%_______________________________________________

%_____________________________________________________________________

\begin{thebibliography}{1}

\bibitem{Ac69} R.C. Ackerberg, {\it On a nonlinear differential equation of electrohydrodynamics}, Proc. Roy. Soc. A \textbf{312} (1969) 129--140.

\bibitem{AKH11} M. Al-Refai, N.I. Kavallaris, M. A. Hajji, {\it Monotone iterative sequences for nonlocal elliptic problems}, Euro. Jnl. Appl. Mathematics \textbf{22(6)}, 533--552.


\bibitem{Am76} H. Amann, Fixed point equations and nonlinear eigenvalue problems in ordered Banach
spaces, SIAM Rev. \textbf{18} (1976), 620–-709.

 \bibitem{b80} C. Bandle, {\it Isoperimetric inequalities and applications}, Monographs and  Studies in Mathematics, 7., Pitman, Boston-London, 1980.

\bibitem{Budd} C. J. Budd, J. F. Williams,
{\it How to adaptively resolve evolutionary singularities
in differential equations with symmetry},
J. Eng. Math. {\textbf 66} (2010) 217--236.

\bibitem{CD99} E.K. Chan \& R.W. Dutton, \textit{Effects of Capacitors, Resistors and Residual Change on the Static and Dynamic Performance of Electrostatically Actuated Devices}, Proceedings of SPIE, {\textbf 3680}, (1999), 120--130.

\bibitem {DKN20} O. Drosinou, N.I. Kavallaris and C.V. Nikolopoulos, {\em  Impacts of noise on quenching of some models arising in MEMS technology},  {\bf arXiv:2012.10922v1}.

\bibitem{DZ19}  G. K. Duong \& H. Zaag, {\it Profile of a touch-down solution to a nonlocal MEMS model}, Math. Models Methods Appl. Sci. {\bf  29 (7)} (2019) 1279--1348.

\bibitem{EGG10}
P. Esposito, N. Ghoussoub, Y. Guo,
{\it Mathematical analysis of partial differential equations
modeling electrostatic MEMS},
Courant Lecture Notes in Mathematics, 20. Courant Institute of Mathematical Sciences,
New York, American Mathematical Society, Providence, RI, 2010.

\bibitem{Ev}
L. C. Evans, {\it Partial Differential Equations}, Second Edition, American Mathematical Society, 2010.

\bibitem{FG93} S. Filippas \& J-S.  Guo, {\it Quenching profiles for one-dimensional semilinear heat equations}, Quart. Appl. Math. {\bf  51}  (1993) 713--729.

\bibitem{FM} A. Friedman, B. McLeod,
{\it Blow-up of positive solutions of semilinear heat equations},
Indiana Univ. Math. J. \textbf{34} (1985) 425--447.

\bibitem{Gi-Ni-Ni}
Gidas, B., Ni, Wei Ming \&  Nirenberg, L.,
{\it Symmetry and related properties via the maximum principle,}
Comm. Math. Phys. \textbf{68} (1979), no. 3, 209--243.

\bibitem{G10}  Y. Guo,  {\em Dynamical solutions of singular wave equations modeling electrostatic MEMS}, SIAM
J. Appl. Dyn. Syst., {\bf 9} (2010), 1135--1163.

\bibitem{Guo91} J-S. Guo,
{\it On a quenching problem with the Robin boundary condition},
Nonlinear Analysis: Theory, Methods \& Applications,
\textbf{17}9, (1991), 803--809.


\bibitem{G14} J.-S. Guo, {\it Recent developments on a nonlocal problem arising in the micro-electromechanical system}, Tamkang Jour. Mathematics \textbf{45(3)}, (2014), 229--241.

\bibitem{GH18}  J.-S. Guo \&  B. Hu {\it Quenching rate for a nonlocal problem arising in the micro-electro mechanical system},  J. Differential Equations \textbf{264} (2018), no. 5, 3285--3311.

\bibitem{GHW08} J.-S. Guo, B. Hu \& C.-J. Wang,
{\it A nonlocal quenching problem arising in micro-electro mechanical systems}, Quart. Appl. Math. \textbf{67} (2009) 725--734.


\bibitem{GN12}
J.-S. Guo \& N.I. Kavallaris,
\textit{On a nonlocal parabolic problem arising in electromechanical MEMS control. }
Disc. Cont. Dynam. Systems \textbf{32} (2012) 1723--1746.

\bibitem{GKWY20}
J.-S. Guo,  N.I. Kavallaris,  C.-Y. Yu \&  C.-Y. Yu  \textit{Bifurcation diagram of a Robin boundary value problem arising in MEMS},  {\bf arXiv:2007.03977v1.}

\bibitem{G08} Y. Guo,
{\it On the partial differential equations of electrostatic MEMS devices III: refined
touchdown behavior},
J. Diff. Eqns. \textbf{244} (2008) 2277--2309.


\bibitem{YG-ZP-MJW06}
Y. Guo, Z. Pan, M.J. Ward,
{\it Touchdown and pull-in voltage behavior of a MEMS device
with varying dielectric properties},
SIAM J.Appl. Math. \textbf{166} (2006) 309--338.

\bibitem{H11} K-M. Hui,
{\it The existence and dynamic properties of a parabolic nonlocal MEMS equation}, Nonlinear Analysis: Theory, Methods \& Applications \textbf{74} (2011) 298--316.

\bibitem{KThesis2000}  N.I. Kavallaris,
Blow-up and global existence of solutions of some nonlocal problems arising in Ohmic heating process,
Ph.D Thesis, National Technical University of Athens (2000) (in Greek).

\bibitem{NK04} N. I. Kavallaris, Asymptotic behaviour and blow-up for a nonlinear diffusion problem
with a nonlocal source term, Proc. Edinb. Math. Soc. {\bf 47}, (2004) 375–-395.

\bibitem{KN07} N.I. Kavallaris \& T. Nadzieja,
{\it On the blow-up of the nonlocal thermistor problem},
Proc. Edin. Math. Soc. \textbf{50}, (2007), 389--409.

\bibitem{KMS08} N.I. Kavallaris, T. Miyasita, T. Suzuki,
Touchdown and related problems in electrostatic MEMS device equation,
Nonlinear Diff. Eqns. Appl. \textbf{15} (2008), 363--385.

\bibitem{KLNT11} N. I. Kavallaris, A. A. Lacey, C. V. Nikolopoulos, D. E. Tzanetis,
A hyperbolic nonlocal problem modelling MEMS technology,
Rocky Mountain J. Math. \textbf{41} (2011), 505--534.

\bibitem{KLNT15} N. I. Kavallaris, A. A. Lacey, C. V. Nikolopoulos, D. E. Tzanetis,
On the quenching behaviour of a semilinear wave equation modelling MEMS technology,
Discrete and Continuous Dynamical Systems - Series A, \textbf{35}(3), (2015),  1009-1037.

\bibitem{KLN16} N. I. Kavallaris, A. A. Lacey, C. V. Nikolopoulos,
On the quenching of a nonlocal parabolic problem arising in electrostatic MEMS control,
Nonlinear Analysis, \textbf{ 138}, (2016),  189--206.

\bibitem{K16}  N. I. Kavallaris,
{\em Quenching solutions of a stochastic parabolic problem arising in electrostatic MEMS control},
Math. Methods Appl. Sci. 41 (2018), no. \textbf{3}, 1074–1082.

\bibitem{KS18} N.I. Kavallaris \& T. Suzuki, {\it Non-Local Partial Differential Equations for Engineering and Biology:
Mathematical Modeling and Analysis, Mathematics for Industry}, {\bf Vol. 31} Springer Nature 2018.

\bibitem{lsu68}  O. Lady\v{z}enskaja, V.A. Solonnikov \& N.N. Ural'ceva, {\it Linear and Quasi-Linear Equations of Parabolic Type}, Amer. Math. Soc. Providence, R.I. 1968.

\bibitem{L1} A.A. Lacey, \textit{Thermal runaway in a nonlocal
problem modelling Ohmic heating: Part II : General proof of blow-up and
asymptotics of runaway}, Euro Jl. Appl. Maths. \textbf{6}, (1995), 201--224 .

\bibitem{L89} H.A. Levine, \textit{Quenching, nonquenching,
and beyond quenching for solution of some parabolic equations},
Ann. Mat. Pura Appl. \textbf{155} (1989), 243--260.

\bibitem{MZ97} F. Merle \& H. Zaag, {\it Reconnection of vortex with the boundary and finite time
quenching}, Nonlinearity \textbf{10} (1997) 1497--1550.

\bibitem{Pao92} C.V. Pao, \textit{Nonlinear Parabolic and Elliptic Equations}, Springer 1992.

\bibitem{JAP-DHB02}
J.A. Pelesko, D.H. Bernstein,
Modeling MEMS and NEMS,
Chapman Hall and CRC Press, 2002.

\bibitem{PC03}
J.A. Pelesko and X.Y.Chen, \textit{Electrostatic deflections of circular elastic
membranes}, J. Electrostatics \textbf{57} (2003), 1--12.


\bibitem{PT01a}
J.A. Pelesko and A.A. Triolo, \textit{Non-local problems in
MEMS device control}, J. Engrg. Math., \textbf{41} (2001), 345--366.

\bibitem{Pelesko2} J.A. Pelesko,
\textit{Mathematical Modeling of Electrostatic MEMS
with Taylored Dielectric Properties} SIAM Journal of Applied Mathematics,
\textbf{ 62}, 3 (2002) pp. 888--908.

\bibitem{Poh} S. I. Poho\v{z}aev, {\it On the eigenfunctions of the equation $\Delta u +\lambda f(u) = 0$}, Dokl. Akad. Nauk SSSR, \textbf{165} (1965), 36--39.

\bibitem{QS} Quittner, P. \& Souplet, P.
 {\it Superlinear parabolic problems. Blow-up, global existence \& steady states.}
 Birkh\"{a}user Adv. Texts Basler Lehrb\"{u}cher. Birkh\"{a}user 2007.

\bibitem{SC97a}
J.J. Seeger \& S.B. Crary, {\it Stabilization of electrostatically actuated mechanical devices}, Proceedings of the 1997 International Conference on Solid-State Sensors and Actuators, (1997), 1133--1336 .

\bibitem{SC97b}
J.J. Seeger \& S.B. Crary, {\it Analysis and simulation of MOS capacitor feedback for stabilizing  electrostatically actuated ,echanical devices}, Second International Conference on the Simulation and Design of Microsystems and Microstructures-MICROSIM97, (1997), 199--208.

\bibitem{T68} G.I. Taylor, {\it The coalescence of closely spaced drops when they are at different electric potentials},
Proc. Roy. Soc. A, \textbf{306} (1968) 423--434.

\bibitem{y} M. Younis,
MEMS Linear and Nonlinear Statics and Dynamics,
Springer, New York, 2011.

\end{thebibliography}
\end{document}